\documentclass[aos]{imsart}
\usepackage{amsmath}
\usepackage{amsfonts}
\usepackage{amsthm}
\usepackage{hyperref}
\usepackage{bbm}
\usepackage{scrextend}
\usepackage{amssymb}

\usepackage{graphicx}
\newcommand*{\p}{\mathbb{P}}
\usepackage{xcolor}
\usepackage{dsfont}
\usepackage{natbib}
\usepackage{graphicx}
\usepackage{float}
\usepackage{multibib}

\newcites{SM}{References}

\newcommand{\norm}[1]{\left\lVert#1\right\rVert}

\newtheorem{Lemma}{Lemma}[section]
\newtheorem{definition}[Lemma]{Definition}
\newtheorem{theorem}[Lemma]{Theorem}
\newtheorem{corollary}[Lemma]{Corollary}
\newtheorem{remark}[Lemma]{Remark}
\newtheorem{example}[Lemma]{Example}
\usepackage[utf8]{inputenc}

\date{}

\definecolor{darkblue}{rgb}{.1, 0.1,.8}
\definecolor{darkgreen}{rgb}{0,0.8,0.2}
\definecolor{darkred}{rgb}{.8, .1,.1}

\usepackage{mathtools}
\mathtoolsset{showonlyrefs}
\newcommand*{\E}{\mathbb{E}}
\newcommand*{\R}{\mathbb{R}}
\renewcommand{\P }{{\mathbb P}}
\newcommand{\1}{\mathbbm{1}}
\newcommand{\nto}{n \to \infty}

\begin{document}
\begin{frontmatter}
	\title{Testing for practically significant  dependencies in high dimensions via bootstrapping maxima of U-statistics}
	\runtitle{Independence testing in high dimensions}
	
	\begin{aug}
		\author[A]{\fnms{Patrick}~\snm{Bastian}\ead[label=e1]{patrick.bastian@rub.de}},
		\author[A]{\fnms{Holger}~\snm{Dette}\ead[label=e2]{holger.dette@rub.de}}
		\and
		\author[B]{\fnms{Johannes}~\snm{Heiny}\ead[label=e3]{johannes.heiny@math.su.se}}
		\address[A]{Ruhr-Universität Bochum\printead[presep={,\ }]{e1,e2}}
		\address[B]{Stockholm University\printead[presep={,\ }]{e3}}
	\end{aug}
	
	\begin{abstract}
		This paper  takes a different  look  on the problem of testing 
the mutual independence of the components of a high-dimensional vector. Instead of testing if all pairwise associations (e.g. all pairwise Kendall's $\tau$) between the components vanish, we are interested in the (null)-hypothesis that all pairwise associations
do not exceed a certain threshold in absolute value.
The consideration of these hypotheses is motivated by the observation  that in the high-dimensional regime, it  is rare, and perhaps impossible, to have a null hypothesis that can be exactly modeled by assuming that all pairwise associations  are  precisely equal to zero.

The formulation of  the null hypothesis as a composite hypothesis  makes the problem of constructing tests non-standard and in this paper we provide a solution 
for a broad class of dependence measures, which can be estimated by $U$-statistics. In particular
we develop an asymptotic and a bootstrap level $\alpha$-test for the new hypotheses in the high-dimensional regime. We also prove that the new tests are minimax-optimal  and  investigate their finite  sample properties by means of a small simulation study and a data example.
	\end{abstract} 
	
	
	
\end{frontmatter}

\section{Introduction} 
\label{sec1} 
  \def\theequation{1.\arabic{equation}}	
	\setcounter{equation}{0}
	
Measuring dependence and testing for independence 
are fundamental problems in
statistics and since the  early work 
of  \cite{Pearson1920}, \cite{Kendall1938}, \cite{Hoeffding1948} and \cite{blum1961} 
numerous authors have worked in this area
\citep[for some more recent references, see][among many others]{gretton2008,Szekely2007,Heller2012,Dette.2012,bergsma2014,Albert2014,Geenens.2020,Chatterjee.2020}. 
 Similarly,  testing for mutual independence of the components 
of a vector has found considerable attention in the literature and exemplary we refer to \cite{Narain1950}, \cite{roy1957},  \cite{Lee1971}, \cite{nagao1973},
and Chapter 9 in the book of \cite{anderson2003}.  However, it is well known that the last-named tests do not perform well
if  the dimension, say $p$,   is comparable to or even larger than the sample size, say $n$, and 
in recent  years many authors have worked on  testing for  mutual independence of the components 
in the high-dimensional regime, where the dimension $p$ converges with the  sample size $n$ to infinity. 

Independence testing of  high-dimensional (mostly) Gaussian data has been considered by
\cite{baietal2009},  \cite{NormalLikelihood}, \cite{Jiang2015}, \cite{chen2017}, \cite{boddetpar2019} and  \cite{dettedoernemann2020}, 
among others, who investigated the asymptotic
properties of  likelihood ratio tests. 
Other authors consider   more  general distributions, where the dependence between two components of the vectors is estimated by different 
covariance/correlation statistics such as Pearson's $r$, Spearman's $\rho $, and Kendall's $\tau$, and different functions are used to aggregate these estimates of the pairwise dependencies. For example, \cite{Bao2015} and \cite{Li2021} use linear spectral statistics of the matrix of estimates,
while \cite{schott2005,Bandedness,yaoetal2017} and \cite{Leung2018}  propose tests based on the Frobenius norm.  
Further very popular methods of aggregating  estimates of the pairwise dependencies
are  maximum-type tests, which have good power properties
against sparse alternatives and have 
been investigated for various
covariance/correlation statistics in \cite{jiang2004,zhou2007,liuetal2008,Li2010,Cai2012,Shao2014,hanetal2017,drttonetal2020, heiny:mikosch:yslas:2021} 
and \cite{heetal2021} among others.

These tests  differ in the distributional assumptions, the way  of aggregation and in the considered  measures to quantify the dependence between two components. However, a common feature of all cited references  consists 
in the fact that  statistical tests are proposed for the  hypotheses
\begin{equation}
	\label{hd1a}
    \begin{split}
   &  H_0^{\rm exact}: ~~  d_{ij} = 0  ~~ \text{ for all } 1 \leq i < j \leq p\,, \\
   & H_1^{\rm exact}:  ~~ d_{ij} \not = 0    ~~   \text{ for at least one pair }   (i,j) \text{ with }   1 \leq i < j \leq p \,,
    \end{split}
\end{equation}
where $d_{ij} = d(X_{1i},X_{1j})$ is a (population) measure of dependence between the two components
$X_{1i}$ and $X_{1j}$ of the $p$-dimensional random vector $X_1 = (X_{11}, \ldots ,X_{1p})^\top$, such as the covariance ${\rm Cov}(X_{1i},X_{1j})$.

In the present paper we take a different point of view on the problem of testing 
the mutual independence of the components of a high-dimensional vector. Our work is motivated by the paper  of 
\cite{berger1987} who  argue that it  {\it is rare, and perhaps impossible, to have a null
hypothesis that can be exactly modeled by a parameter being exactly $0$.}
Similarly, \cite{tukey1991} argues in the context of multiple comparisons of means that  
{\it  $\ldots $
``All we know about the world teaches us that the
 effects of A and B are always different  - in some 
 decimal place - for any A and B.  Thus asking ``Are
 the effects different?'' is foolish''.  $\ldots $
} 

In the context of independence testing this means that in  many applications, in particular in the high-dimensional regime,  it is often unlikely that all $p(p-1)/2$ 
associations  (measured by $d_{ij}$) satisfy  $d_{ij} = 0$  ($1 \leq i <  j \leq p$).
As a consequence  one uses a  formulation of the null hypothesis in \eqref{hd1a},
which is believed to be not true, and for sufficiently large  sample size any consistent test will detect an arbitrary small deviation from the null hypothesis, which might not be of interest scientifically. 
Problems of this type are  particularly relevant in the big-data era.  Here the sample size (and dimension)  is usually large and one will reject the null hypothesis $H_0^{\rm exact}$ in \eqref{hd1a}, even if some of the dependence measures $d_{ij}$ do not vanish exactly but are small and scientifically not relevant. 
For example \cite{yaoetal2017} consider the   hypotheses in \eqref{hd1a} for a band 
 of the form  $125 \leq |i-j| \leq p= 150$. Their test rejects the null hypothesis, although the authors observe 
that   the dependencies are relatively weak, see 
 Section 7 in in this reference.

As an alternative we propose   to investigate if all associations  (measured by the quantities $d_{ij}$) are in some sense ``small''.  For this purpose 
we consider the hypotheses 
\begin{equation}
	\label{hd1b}
    \begin{split}
   &  H_0: ~~  | d_{ij} | \leq \Delta   ~~ \text{ for all } 1 \leq i < j \leq p\,, \\
   & H_1:  ~~ | d_{ij} | > \Delta   ~~   \text{ for at least one pair }   (i,j) \text{ with }   1 \leq i < j \leq p\,, 
    \end{split}
\end{equation} 
where $\Delta > 0 $ is a given threshold, which defines when a dependence between the components $i$ and $j$ is considered as (scientifically) not relevant.
Note that \eqref{hd1a} is obtained from \eqref{hd1b} for $\Delta=0$, but in the present paper we are not interested in arbitrary small deviations from \mbox{$d_{ij}=0$}, because we are aiming to detect only dependencies 
exceeding a given positive threshold. This  threshold can in fact be determined studying the robustness properties of the subsequent inference  tools which use  the independence assumption. The  rejection of $H_{0}$ in \eqref{hd1b} allows to decide at a controlled  type I error that at least one association  is larger than the given  threshold $\Delta$. On the other hand,
 interchanging the null hypothesis in \eqref{hd1b} and developing an appropriate test allows to decide  at a controlled 
 type I error that all  dependence measures $ | d_{ij} | $  are  smaller than  $\Delta$ (see Remark \ref{rem11}  for more details).
Note that interchanging the null-hypothesis and alternative does not make sense for the hypotheses in
 \eqref{hd1a} and here we are only able to control the error probability of ``deciding for dependence''.

An essential ingredient in this approach is  the specification of the threshold $\Delta$, and its choice depends sensitively on the particular problem under consideration. 
Essentially, this boils down to the important question when a correlation (or another dependence measure) is  {\it practically significant}, which has a long history in the application of statistics. It is  related to the specification of the effect size \citep[see, for example,][]{cohen1988}, which is used in various disciplines to obtain a better interpretation of $p$-values for comparing sample means. For a dependence measure $d_\star$  varying in absolute value between $0$ and $1$ several authors transfer this concept to classify the 
strength of the association in three categories ``small'' ($\Delta  \leq  d_\star \leq \Delta_1$), ``medium'' ($\Delta_1 <  d_\star \leq \Delta_2$) or ``large''  ($\Delta_2 <  d_\star \leq 1$), where the exact definition of the classes varies between  the disciplines and the considered measure. For example, 
\cite{cohen1988} proposes
$\Delta=0.1, \Delta_1=0.3$ and $\Delta_2=0.5$ for  Pearson's correlation and these  thresholds  are refined by \cite{Brydges2019}  to $\Delta=0.1$, $\Delta_1=0.2$ and $\Delta_2=0.3$ in the context of gerontology research  and by  \cite{Lovakov2021} to $\Delta=0.12, \Delta_1=0.24$ and $\Delta_2=0.41$  for social psychology. \cite{Huang2022} consider Spearmans $\rho$ and  propose to use  $\Delta=0.3, \Delta_1=0.5, \Delta_2=0.7$ in a study ascertaining risk factors for colorectal cancer.  
There are numerous other papers discussing these issues from different perspectives
\citep[see][for example]{Hemphill2003,Bosco2015,Quintana2016}, but a common aspect consists in the fact that this literature usually defines a threshold $\Delta$, which should be exceeded to consider an association as {\it practically significant}. With this point of view rejecting the null hypothesis in \eqref{hd1b} means that {\it there exists at least one 
practically significant association between the components of the vector $X_1$.}

The purpose of the present paper is the development of statistical tests for hypotheses of the form \eqref{hd1b} in the 
high-dimensional regime, where the dependence measures
$d_{ij}$ can be estimated by  $U$-statistics. Typical examples include the classical covariance, Kendall's  $\tau$, Hoeffding's $D$, Blum-Kiefer-Rosenblatt's $R$, 
Bergsma–Dassios–Yanagimoto’s $\tau^*$
and a dominating term of   Spearman's rank correlation $\rho$.
 
As 
 \cite{jiang2004,zhou2007,liuetal2008,hanetal2017} and \cite{drttonetal2020}  we consider maximum-type tests,
and allow the dimension  $p$ to grow exponentially with $n$.
We develop a new asymptotic  and a new bootstrap
test   for the hypotheses in  \eqref{hd1b} and investigate their statistical properties. Compared to the ``classical'' hypotheses in \eqref{hd1a} 
the composite structure of the hypotheses in \eqref{hd1b} makes both  tasks   non-standard from a mathematical  point of view. 
On the one hand, 
the  asymptotic analysis  of  estimators  of $\max_{1\leq i < j \leq p} | d_{ij} | $   by  
Poisson approximation techniques \citep[see, for example,][]{Arratiaetal1989} is  very demanding  due to the additional dependencies 
under the null hypothesis in \eqref{hd1b}. 
On the other hand, 
further challenges arise in the development of bootstrap procedures, since ``generating data  under the null hypothesis 
$H_{0}: \max_{1\leq i < j \leq p} | d_{ij} | \leq \Delta$'' 
is not straightforward for the composite hypotheses in \eqref{hd1b}. 
For this purpose a novel bootstrap method is developed and its consistency for testing hypotheses of the form \eqref{hd1a} is established.  Note  that even for a fixed dimension $p$  it is by no means obvious  that the bootstrap  is applicable as
aggregation is performed by the maximum operator, which is not Hadamard differentiable  (see Theorem 3.1 in \cite{fang2019inference} which actually  gives an ``if and only if'' condition for the standard bootstrap distribution coinciding with the asymptotic distribution). If the dimension  is increasing  with the sample size such an approach is even more challenging, as there does not necessarily exist a limit distribution.
Moreover, we also emphasize that  our results are valid for discrete and continuous distributions, independently of the degree of  degeneracy of the $U$-statistic. Note that  the tests for classical hypotheses with degenerate $U$-statistics require continuous distributions, see
\cite{drttonetal2020}.

In Section \ref{sec2} we consider testing problems
of the form \eqref{hd1b} in a more general context and propose an 
asymptotic level $\alpha$ test, which  is (uniformly)  consistent 
against local alternatives, where the maximum deviation is at least
$\Delta + c {\sqrt{\log d }}/{ \sqrt n}$
for some constant $c>0$
(here $d=p(p-1)/2$ is  the number of terms over which  the maximum is taken). 
The proof of these properties is based 
on the weak convergence of an appropriately normalized 
maximum  statistic  to a Gumbel distribution under suitable assumptions on the dependence structure, sample size and dimension.
 As such assumptions are often hard to justify 
 in statistical practice and the convergence rates 
 in extreme value theory are usually very slow,
 we develop  in Section \ref{sec22} 
 a non-standard  bootstrap test for the hypotheses of the form \eqref{hd1b} and  prove its validity. 
 Moreover, our approach also provides an asymptotic confidence interval for the quantity  $\max_{1 \leq i < j \leq p}  | d_{ij} |$ (note that confidence intervals are rarely presented in this field).
 In Section~\ref{sec3} we specialize these results to the problem of testing hypotheses of the form \eqref{hd1b}, where the associations $d_{ij}$ are given by the covariances,  Kendall's $\tau$,  a dominating term of   Spearman's $\rho$, Hoeffding's $D$, Blum-Kiefer-Rosenblatt's $R$ and 
Bergsma–Dassios–Yanagimoto’s $\tau^*$. In particular, we prove that for  many dependence measures the tests  proposed in this paper are  minimax-optimal 
 against local alternatives of the  form $\max_{1\leq i < j \leq p} | d_{ij} |= \Delta + c {\sqrt{\log p }}/{ \sqrt n}$.
 Note that these rates coincide with the  minimax-optimal  rates 
 for testing the classical hypotheses  \eqref{hd1a}, that is $\Delta =0$, 
if dependencies are measured by Spearman's $\rho$ and Kendall's $\tau$ correlations, see, for example, 
\cite{hanetal2017}. 
 In Section \ref{sec4} we 
  investigate the finite   sample properties of  the developed methodology by means of a simulation study and a data example.
Finally, all technical proofs and details are deferred to an online supplement.

\section{Testing for relevant deviations} \label{sec2} 
  \def\theequation{2.\arabic{equation}}	
\setcounter{equation}{0}

In this section we consider the testing problems in a slightly more general 
but notationally simpler form as described in the introduction. The case of testing for relevant deviations  of the entries in  a matrix of pairwise dependence measures is a special case of the following discussion (see Example \ref{ex1}) and will be addressed in Section~\ref{sec3} in more detail. To be precise, let     $X_1,\ldots ,X_n$ denote independent identically distributed $p$-dimensional random vectors with distribution function $F$. Note that formally $F$ depends on the dimension $p$, which  varies with $n$, but we will not reflect this dependence in our notation throughout this paper.
For some positive integer $m$ let 
\begin{align} \label{hd2}
h =(h_1, \ldots , h_d)^\top :\left(\R^p\right)^m \rightarrow \R^d
\end{align}
denote a measurable symmetric function  with finite expectation 
\begin{align} \label{hd0}
\theta_F=(\theta_1,\ldots,\theta_d)^\top =\E_F[h(X_1, \ldots ,X_m)] \in \mathbb{R}^{d }~,
\end{align}
 which defines our  parameter of interest.  In order to estimate the parameter $\theta_F$ we consider the $U$-statistic of order $m$
\begin{align} \label{hd3}
    U= (U_1 , \ldots , U_d)^\top = {n \choose m}^{-1} \sum_{1 \leq l_1< \ldots <l_m \leq n}h(X_{l_1}, \ldots ,X_{l_m})\,.
\end{align}
In the high-dimensional regime $U$-statistics have recently found considerable interest in the literature and we refer to \cite{chen2017,UApprox,songetal2019,kim2020,wang2021inference,chengetal2022} among others.

\begin{example} \label{ex1}
{\rm
We briefly illustrate the notation for dependence measures between the components of high-dimensional vectors as introduced in Section \ref{sec1}. In particular, such $U$-statistics have been investigated by \cite{hanetal2017,chen2017a,Zhou2019,drttonetal2020} and \cite{heetal2021} 
in the context of independence testing by means of the classical hypotheses \eqref{hd1a}.

To be precise,  for $ 1 \leq i < j \leq p$ let 
\begin{align} \label{hd21}
d_{ij} = 
d(X_{1i},X_{1j})
= \mathbb{E}_F [\tilde h ( X_{1i},X_{1j}, \ldots ,X_{mi},X_{mj})] 
\end{align}
denote a dependence measure between the $i$th and $j$th 
components of the random vector $X_{1}= (X_{11} , \ldots  , X_{1p} )^\top $, which can be expressed 
as the expectation of a kernel  $\tilde h: \R^{2m} \to \mathbb{R}$ of order $m $
evaluated at $( X_{1i},X_{1j}, \ldots  , X_{mi},X_{mj} )$.   In this case 
the function $h$  in \eqref{hd2}  is  defined by
\begin{align*}
h (X_1,...,X_m)  & = {\rm vech} \big ( ( h_{ij} (X_1,...,X_m) )_{i,j=1, \ldots , p}  \big ) \\
& = {\rm vech} \big ( (\tilde h ( X_{1i},X_{1j},...,X_{mi},X_{mj}) )_{i,j=1, \ldots , p}  \big )\,,
\end{align*}
where the second equality defines the functions $h_{ij} : \mathbb{R}^{pm} \to \mathbb{R}$ in an obvious manner and 
vech($\cdot$) is the operator that stacks the columns  above  the diagonal of a symmetric 
$ p \times p$ matrix as a vector with $d = p(p -1)/2$ components. Note that the index $(i,j)$ 
in   the definition of the function  $ h_{ij}$
is only used to emphasize that each $h_{ij}$ acts on different components of the vectors  $X_{1}, \ldots ,X_{m}$. 
Similarly, the vector $\theta_F $ is defined by
$\theta_F  =  {\rm vech} \big ( ( d_{ij} )_{i,j=1, \ldots , p}  \big )$, and the components of the vector
$U= {\rm vech} \big ( ( U_{ij} )_{i,j=1, \ldots , p}  \big ) $ in \eqref{hd3} are given by
\begin{align*}
 U_{ij}  &= {n \choose m}^{-1}\sum_{1 \leq l_1<...<l_m \leq n}  h_{ij} ( X_{l_1},...,X_{l_m})\\
 &= {n \choose m}^{-1}\sum_{1 \leq l_1<...<l_m \leq n} \tilde h (X_{l_1i},X_{l_1j},...,X_{l_mi},X_{l_mj})\,.
\end{align*}
A more detailed discussion of specific dependence measures is postponed to Section \ref{sec3}.
}
\end{example}

Recall that, in this paper, we are not interested in testing the ``classical'' 
hypotheses $H_0:\theta_F=0 \text{ versus } H_1:\theta_F \neq 0$, but want to investigate 
if at least one of the components  $\theta_i$ of  the vector $\theta_F =(\theta_1, \ldots , \theta_d)^\top$ exceeds a given threshold
$\Delta>0$, that is 
\begin{align} \label{hd4}
   H_0: \max_{i=1}^d  | \theta_i | \leq \Delta     ~\text{ versus } ~  H_1: \max_{i=1}^d | \theta_i | > \Delta ~,
\end{align}
where $\Delta$  denotes the largest deviation 
that is still considered  as negligible. Hypotheses of this form are often called {\it relevant}  hypotheses. In the case $d=1$ these hypotheses (more precisely the interchanged hypotheses 
$ H_0: \max_{i=1}^d  | \theta_i | > \Delta     \text{ versus }   H_1: \max_{i=1}^d | \theta_i | \leq \Delta$) have found considerable attention 
in the biostatistics literature \citep[see, for example  the monographs by][]{chowliu1992,wellek2010testing}, but - despite their importance -  they  have not been studied intensively in
the high-dimensional regime. 
In what follows, we will  construct tests for hypotheses of the form  \eqref{hd4} based on  asymptotic 
theory  of a (standardized)  estimator of 
$\max_{i=1}^d | \theta_i |$ and also develop (under substantially weaker assumptions) a  non-standard bootstrap test
in the high-dimensional regime, 
where we allow the dimension $d$ to grow exponentially with $n$.\\

\subsection{An asymptotic level $\alpha$ test}
\label{sec21}

Recall the definition of the parameter  $\theta_F = \E_F[U] = \E_F[h(X_1, \ldots  ,X_m)]  \in \mathbb{R}^d$ in \eqref{hd0}, where $X_1, \ldots  ,X_m \sim F$ are independent $p$-dimensional random vectors with  distribution $F$  (the dependence on $p$ is omitted here for simplicity). In Example \ref{ex1} and in most cases of practical interest, $d$ is given as a function of $p$, but our theoretical results are more generally stated in a $U$-statistics framework that only depends on the dimension of the vector $\theta_F$. We denote by  
${\cal F}$ the class of all  distribution functions on $\mathbb{R}^p$ for which  the expectation $\E_F[U]$ exists, and we set
$\theta_{i}=\theta_{F,i}=\E_F[h_{i}(X_{1}, \ldots ,X_{m})] $
to be the $i$th component of $\theta_F $, where $h_i$, the $i$th component of the vector $h$ in \eqref{hd2}, is  a
 symmetric kernel  of order $m$. 
 Define  by 
\begin{align}
\label{UDef}
U_{i}={n \choose m}^{-1}\sum_{1\leq l_1< \ldots <l_m\leq n}h_{i}(X_{l_1}, \ldots ,X_{l_m})\,,\qquad i=1,\ldots,d\,,
\end{align}
the corresponding estimate  of $\theta_i $.  
Under standard assumptions the statistics $U_i$  are   unbiased and consistent estimators
of the parameters $\theta_i$ ($i=1, \ldots ,d$), and therefore it is reasonable
to reject the null hypothesis in \eqref{hd4} for large values of $\max_{i=1}^d | U_i | $.
For technical reasons we consider the quantities  $U_i ^2$  instead of $| U_i |$ and compare 
their maximum with $\Delta^2$.
Corresponding results for $|U_i|$, which estimates $|\theta_i|$, can be derived in a similar way and
are discussed  in Remark \ref{nv} (b).

We note that 
\begin{align}
 \label{pb10}
  U_i^2  - \Delta^2  =  (U_i- \theta_i)^2  + 2 \theta_i (U_i - \theta_i)  - (\Delta^2 - \theta_i^2)   
\end{align}
and introduce the notations
	\begin{align} \label{x1b}
	\zeta_{1,i}={\rm Var}_F (h_{1,i}(X_1))
	\quad \text{ and } \quad
	      h_{1,i}(x) =\E_F[h_i(X_1, \ldots ,X_m)|X_1 =x]~.
	\end{align}
	If  $\zeta_{1,i} >0 $,  the kernel
 $h_i$  of   the statistic  $U_i$ is called non-degenerate.
Note that  this property 
depends on the kernel 
 $h_i$ and on the  distribution $F$. In particular,  
 for composite null hypotheses of the form \eqref{hd4}, there may exist  different distributions, say  $F_1, F_2 \in \mathcal{F} $, 
both  corresponding to  parameters $\theta_{F_1}$ and  
$\theta_{F_2}$ in the null hypothesis 
such that the kernel  is degenerate under $F_1$ and non-degenerate  under $F_2$, that is  
$0={\rm Var}_{F_1} (h_{1,i}(X_1))  < {\rm Var}_{F_2} (h_{1,i}(X_1)) $. 
If $\zeta_{1,i}>0 $   the statistic $U_i$
 is asymptotically normal distributed
with mean $\theta_i$ and variance  $m^2\zeta_{1,i}/n$.
Therefore, it is reasonable to standardize 
the differences $U_i^2  - \Delta^2  $  appropriately before taking the maximum. We propose 
to use the   test statistic 
\begin{align}
\label{TestStatDefin}
\mathcal{T}_{n,\Delta}:=\underset{1\leq i\leq d}{\max}\ \frac{U^2_{i}-\Delta^2}{2 \, \hat{\sigma}_{i}\,\Delta}
\end{align}
for testing the hypotheses in \eqref{hd4}, where 
\begin{align}
	\label{varDef}
\hat{\sigma}^2_{i}:=\frac{m^2(n-1)}{n(n-m)^2}\sum_{k=1}^{n}(q_{k,i}-U_i)^2
\end{align}
is a Jackknife based estimator  of the variance of $U_i$ 
 and  $q_{k,i}$  is  defined by 
\begin{align*}
q_{k,i}:={n-1 \choose m-1}^{-1} \sum_{1 \leq l_1<..<l_{m-1} \leq n, l_j \neq k}h_i(X_k,X_{l_1}, \ldots ,X_{l_{m-1}})
\end{align*}
 \citep[see][for details]{Zhou2019}.  The null hypothesis
 in \eqref{hd4} is rejected, whenever 
 \begin{align}	 \label{hd5}	
\mathcal{T}_{n,\Delta}>\frac{q_{1-\alpha}}{a_d}+b_d\,,
	\end{align} 
	where $q_{1-\alpha} = -\log(\log(\frac{1}{1-\alpha})) $ is the 
$({1-\alpha})$-quantile of the standard Gumbel distribution with distribution function $\exp ( - \exp ( - x) ), x\in \R$, and 
\begin{equation*}
a_d=\sqrt{2\log d} \quad ~\text{  and } ~ \quad b_d=a_d-\frac{\log(\log d)+\log(4\pi)}{2a_d}. 
\end{equation*}	
In the following discussion we will show that this test has asymptotic level $\alpha$.
An important step in these arguments is a proof of the weak convergence  
 \begin{align}	 \label{hd5a}	
 \lim_{n \to \infty } \mathbb{P} \big (
a_d  (\mathcal{T}_{n,\Delta}  - b_d  ) \leq x \big ) 
= \exp ( - \exp ( - x) )\,,\qquad x\in \R\,,
	\end{align}
in the case $ |\theta_1 |= | \theta_2 |= \ldots = |\theta_d | = \Delta > 0 $. 
Note that this choice corresponds to the most extreme case in the null hypothesis \eqref{hd1b} which means that  the rejection
probabilities  of  the test \eqref{hd5}
for all other parameter
constellations  under the null hypothesis $H_0: \max_{i=1}^d |\theta_i| \leq \Delta$
are bounded by  this scenario 
and in many cases substantially smaller.

 Under additional assumptions on the kernels $h_i$ we can also prove that the test \eqref{hd5} is
minimax optimal, see Section \ref{sec35} for a discussion of this property in the context of dependence measures. Interestingly, it turns out that for deriving these properties it is not necessary to assume that the kernels $h_i$ are non-degenerate for all distributions $F$ corresponding to the null hypothesis (see the discussion below, in particular Assumption (A2)).

In what follows, we will need the function $\psi_\beta(x)=\exp(x^\beta)-1$ and the corresponding Orlicz norm 
	\begin{align} \label{orl}
	 \norm{Z}_{\psi_\beta}:=\inf\{\nu>0 : \E [\psi_\beta(|Z|/\nu)]\leq 1\}
	\end{align}
	of  a real-valued random variable $Z$.
	We continue by spelling out several regularity assumptions that are required for proving the weak convergence in \eqref{hd5a}.

\begin{description}
	\item \textbf{(A1)} For some constant $ \beta \in (0,2] $ there exist a non-negative sequence $(B_n)_{n \in \mathbb{N}}$ and a constant $D > 0 $ such that for all  $d=d(n), n \in  \mathbb{N}$, 
	\begin{align*}
\max_{1 \leq i \leq d}
	    \norm{h_{i}(X_{1}, \ldots ,X_m)-\theta_i}_{\psi_\beta} &\leq B_n\,,\\
	\max_{1 \leq i \leq d}    \zeta_{1,i}& \leq D \,, \\
	 	\max_{1 \leq i \leq d}   \E_F[(h_{1,i}(X_1)-\theta_i)^4] & \leq DB_n^2\,.
	\end{align*} 
	\item \textbf{(A2)} There exist constants  $\underline{b}>0$ and $c \in (0,\Delta)$ such that 
	$
	\min_{1 \leq i \leq d, |\theta_i|>c}\zeta_{1,i}>\underline{b}
	$
	 for all $d=d(n), n \in \mathbb{N}$. Here and in the following,  a minimum over the empty set is defined as  $+\infty$.
	\item \textbf{(A3)} Let $\kappa_{i,j}
	= {\rm Corr}_F (h_{1,i}(X_1) , h_{1,j}(X_1) ) 
	\in (-1,1) $ denote the correlation between  $h_{1,i}(X_1)$ and $h_{1,j}(X_1)$. There exist a constant $\epsilon>0$ and a sequence $\gamma_n=o(1)$ such that for all  $d=d(n)$, $n\in  \mathbb{N}$
	\begin{align*}	
	\sum_{1\le i\neq j\le d}\frac{|\kappa_{i,j}|}{\sqrt{1-\kappa_{i,j}^2}}\exp\left(-\frac{(2-\epsilon)\log d}{1+|\kappa_{i,j}|}\right)\leq \gamma_n\,.
	\end{align*}
\end{description}

Assumption $(A1)$ is a technical condition that captures a uniform tail probability decay from which we will deduce concentration inequalities for 
the components of the $U$-statistic defined in \eqref{hd3}.  
It is possible to  weaken this assumption to a polynomial one at the cost of an only polynomial growth of $p$ in $n$.
Note that   condition  (A1) is
always satisfied  if the kernel $h$ is bounded. 
Assumption $(A2)$ is a uniform non-degeneracy requirement which is a standard condition for deriving Gaussian approximation results, see for instance \cite{UApprox, ImprovedApprox} among others. 
  We emphasize that 
this assumption is only required here for  the parameters  $\theta_i$ which  are (uniformly) bounded away from $0$. This covers most cases of practical 
interest, where a degenerate kernel appears 
in the case $ \theta_i =0$, but the kernel is non-degenerate, whenever $\theta_i \not = 0$. Roughly speaking, for the problem of testing composite  hypotheses of the form \eqref{hd4} the distinction between the degenerate and non-degenerate case 
is basically not necessary if Assumption (A2) is satisfied (see Section \ref{sec34} for a more detailed discussion  in the context of dependence measures).
Finally,  Assumption $(A3)$ ensures that we can approximate the maximum of dependent normal distributed random variables by the maximum of independent ones,  thereby obtaining a Gumbel limit for the statistic. This assumption is implicit in other works such as \cite{hanetal2017} as it is always true under the classical null of independence.
We already emphasize at this point that this assumption will not  be required 
for the  bootstrap test, which will be developed in Section \ref{sec22} later on.

Our first result shows that the test defined in \eqref{hd5} has asymptotic level $\alpha$
(uniformly over a given class of distributions).
For a  precise statement consider the set
of all distribution functions on $\mathbb{R}^d$ satisfying Assumptions (A1) - (A3),  and define 
\begin{align}
\label{det0aa}
V_0:= \Big \{z=(z_1,\ldots,z_d)^{\top} \in \R^d ~\Big |~ \max_{1 \leq i \leq  d} |z_i|\leq  \Delta  \Big \}
\end{align}
as the parameter space corresponding to the null hypothesis in \eqref{hd4}.  Note that these sets depend on $n$ (through the dimension $d=d(n)$).
We define 
\begin{align}
\label{det0a}
    {\cal H} _{0}(\Delta)  := \big \{F \in {\cal F} ~|~  \theta_F \in V_0 ~,~F \text{ satisfies  Assumptions (A1), (A2), (A3) } 
     \big \}~~~~
\end{align}
as the set of distribution functions satisfying the null hypothesis (and the basic assumptions)
with  existing expectation $\mathbb{E}_F[U]$. Note that $ {\cal H} _{0}(\Delta) $
depends on the constants $\underline{b},$ $D$ and on $n$ (through the dimension $d=d(n)$ and
sequence $(B_n)_{n \in \mathbb{N}}$)
which is not reflected in our notation. 

\begin{theorem}
	\label{alpha}
If 	Assumptions  (A1), (A2), (A3)  are satisfied, $\log d=o(n^\gamma)$ with   $ 0\leq \gamma \leq \frac{1}{2/\beta+1}$ and 
 \begin{align}
 \label{hd7}
	\frac{B_n^2 \big ( \log(nd) \big) ^{4+2/\beta}}{n}=o(1)~, \qquad n\to \infty\,,
	\end{align}
	then, for any $\alpha \in (0,1-e^{-1})$, 
	\begin{align}	
	\label{hd23}
	\limsup_{n \to \infty } \underset{F \in {\cal H} _{0}(\Delta)  }{\sup}
	\mathbb{P} \Big (  \mathcal{T}_{n,\Delta}> \frac{q_{1-\alpha}}{a_d}+b_d  \Big ) \leq \alpha
	\, , 
 \end{align}

 with strict inequality, whenever 
 $\limsup_{n\to \infty}
 \big |   \big \{
 i \in \{ 1,\ldots , d\} : 
| \theta_i| =\Delta  \big \}  \big |/d < 1 $. Moreover,
\begin{align}	
	\label{hd23a}
\lim_{n \to \infty } 
\sup_{F \in {\cal H} _{0}(\Delta)  }
	\mathbb{P} \Big (  \mathcal{T}_{n,\Delta}> \frac{q_{1-\alpha}}{a_d}+b_d  \Big ) = \begin{cases}
  \alpha   ~, ~\text { if  }  | \theta_i| =\Delta \text { for all } ~1 \leq  i \leq d \,,  \\
	 0  ~, ~\text { if  }    ~\sup_{d  \in \mathbb{N}} \max_{i=1}^d |\theta_i|  < \Delta  \,.
	\end{cases}
\end{align}

\end{theorem}

\begin{remark} 
\label{rem1}~~
{\rm 
\begin{itemize}
    \item[(1)] 
 For the proof of Theorem \ref{alpha}  we proceed in two steps:   first we use Gaussian approximation techniques  \citep[see][for example]{UApprox,ImprovedApprox}
 and then compare the resulting Gaussian vector with  a Gaussian vector with i.i.d components under the additional assumption (A3) on the dependence structure of the vector $X_1$.
 The maximum of the latter Gaussian vector then converges to a Gumbel distribution under suitable assumptions on the dependence structure, sample size and dimension. 
    \item[(2)] 
Note that the statement \eqref{hd23} addresses the worst case under the null hypotheses $H_0: \max_{i=1}^d |\theta_i| \leq \Delta $ (uniformly over the class of distributions defined by \eqref{det0a}).

The second part in Theorem \ref{alpha} shows  that there
also exist  vectors  $\theta_F$ 
such that equality holds in \eqref{hd23} and that  
for many parameter constellations in the null hypothesis  the type I error of the test \eqref{hd5} 	will be much smaller than $\alpha$. 
 
  \item[(3)] Assumption (A3) can be dropped if the the statistics  $U_i$ are pairwise positively correlated. Simulations suggest that the test 
  \eqref{hd5}	is very conservative and therefore we will not include this test in the finite sample study presented in Section \ref{sec4}.
  
	\end{itemize}
}
\end{remark}

Next we turn to the consistency of the test \eqref{hd5} and  define 
\begin{align*}
 V(c) = 
\Big \{z \in \R^d ~\Big |~ \max_{1 \leq i \leq  d} |z_i|\geq  \Delta + c B_n((\log d)/n)^{1/2} \Big \}
\end{align*}

as a set of  alternatives (note that for a bounded kernel $h$ the sequence  $B_n$  can be chosen as a constant sequence). 
We will study the power of the  test  \eqref{hd5} against  alternatives in the set 
\begin{align}
\label{det0} 
    {\cal H}_1 (c) = \Big \{F  \in {\cal F} ~| ~\theta_F \in V(c) ~;~F \text{ satisfies  Assumption (A1)}
     \Big \}\, .
\end{align}

\begin{theorem}
	\label{consistency}
	If $\log d=o(n^\gamma)$ with  $0 \leq \gamma \leq \frac{1}{2/\beta+1}$, then  there exists a constant $c>0$, only depending on $\gamma$ and $\beta$, such that
	\begin{align*}
	\lim_{  \nto}   \underset{F\in   {\cal H}_1 (c)}{\inf}\p \Big ( \mathcal{T}_{n,\Delta}>\frac{q_{1-\alpha}}{a_d}+b_d  \Big ) = 1 \,.
	\end{align*}
\end{theorem}

The choice of  the  sequence $(B_n)_{n \in \mathbb{N}}$ 
depends on the tail behavior of the random variables 
$h_i (X_1,\ldots ,X_m)$ and $h_{1,i} (X_k)$ and the condition 
\eqref{hd7} puts a further restriction on the growth rate of the 
dimension.
For example, if the sequence  $(B_n)_{n \in \mathbb{N}}$ is bounded, 
Theorem \ref{alpha} is applicable with an exponentially 
growing dimension $d$, i.e. $\log d=o(n^{1/(4+2/\beta)})$
which results in the rate $\log d=o(n^{1/5})$
 if  $h_1(X_1, \ldots  ,X_m), \ldots , h_d(X_1, \ldots ,X_m)$
 are sub-Gaussian random variables. Note that this property  implies that the random variables $h_{1,1}(X_1), \ldots , h_{1,d}(X_1)$ are Sub-Gaussian as well.
 Under additional assumptions on the kernel $h$
 it can also be proved that 
 the rate $\sqrt{\log (d) /n}$ in Theorem \ref{consistency}  is in fact minimax optimal and 
 cannot be improved by other tests. We discuss this optimality property in the context
 of bivariate dependence measures in Section \ref{sec35}.

\subsection{Bootstrap}
\label{sec22}
	The use of the asymptotic quantiles in the decision rule \eqref{hd5} is attractive from a computational point of view.
	On the other hand the basic statement of  weak convergence \eqref{hd5a}   used to establish its validity 
requires additional  assumptions regarding the dependence structure of the components of the random vectors $X_i$
as formulated in Assumption (A3). Moreover, for testing classical hypotheses,
it is well--known that the rate of convergence in results of this type is typically rather slow \citep[see, for example, Section B.4 in][]{hanetal2017}
and the nominal  level of the test \eqref{hd5} will not be well approximated.

In this section we discuss a bootstrap approach to  solve these problems. As usual in  applications of the bootstrap in testing  hypotheses this requires simulating the distribution of  the statistic $\mathcal{T}_{n,\Delta}$ 
in \eqref{TestStatDefin}
under an appropriate configuration of the null hypothesis $H_0:\max_{1 \leq i \leq d} |\theta_i | \leq \Delta$. While this task is relatively easy in the case of the ``classical'' null hypothesis corresponding to the case $\Delta=0$ it
is significantly more difficult
for the composite hypotheses 
corresponding to $\Delta >0$ as considered in this paper. 
The approach proposed here is based on bootstrap data generated at the ``boundary'' of the hypotheses in \eqref{hd4}, that is 
$\max_{1 \leq i \leq d}  | \theta_i| = \Delta $.

To be precise, let $X_1^*, \ldots , X_n^*$  be drawn with replacement from $X_1, \ldots , X_n$
and define for $i=1, \ldots , d$ by 

\begin{equation}
    \label{hd8}
    U_{i}^*={n \choose m}^{-1}\sum_{1\leq l_1< \ldots <l_m\leq n}h_{i}(X_{l_1}^*, \ldots ,X_{l_m}^*)
\end{equation}
a bootstrap analogue of the statistic introduced in \eqref{UDef}. Note that the conditional expectation of
$U_i^*$ given $X_1, \ldots , X_n$ is
given by the $V$-statistic 
\begin{align}
\label{hd11}
    V_i= \mathbb{E}_F  \big [U_i^* | X_1, \ldots , X_n \big ]
    = \frac{1}{n^m}\sum_{l_1, \ldots ,l_m=1}^{n}h_i(X_{l_1}, \ldots ,X_{l_m})
\end{align}
 \citep[see, for example,][]{UApprox}.
Next we define a truncated version of $V_i$, that is 
\begin{align} 
\label{hd10}
V_{i,\Delta}=
    \begin{cases}
    V_i\,,  \quad \text{ if} \ |V_i|\leq \Delta\\
    \Delta\,, \quad \text{ otherwise}
\end{cases}  \qquad i=1, \ldots , d
\end{align}
and note that $\big|\E[U_i^* - V_i  +V_{i,\Delta}| X_1, \ldots , X_n ]\big| \le \Delta$  a.s. We finally define 
	\begin{align}
	\label{hd10a}
	\mathcal{T}^*_{n}=\underset{1\leq i\leq d}{\max}\ \frac{\left(U^*_i-V_i+V_{i,\Delta}\right)^2  -V_{i,\Delta}^2}{2\,\hat{\sigma}_{i}\, \Delta}
	\end{align} 
	as the bootstrap analogue of the statistic 	$\mathcal{T}_{n, \Delta}$ defined in \eqref{TestStatDefin} and denote by
	$q^*_{1-\alpha}$ the $(1-\alpha)$-quantile of the distribution of $\mathcal{T}^*_{n}$. We propose to reject the null
	hypothesis in \eqref{hd4}, whenever 
\begin{align} \label{boottest}
	\mathcal{T}_{n,\Delta}>q^*_{1-\alpha}\,.
	\end{align}
The next result shows that this procedure defines a (uniformly) consistent and asymptotic level $\alpha$ test for the hypotheses
\eqref{hd4}. We emphasize that we do not require Assumption (A3) for this statement
 and that in this sense the bootstrap  
 test is valid under more  general
 assumptions than the asymptotic test 
 \eqref{hd5}. 
 This comes at the cost of a slight loss of sensitivity as the bootstrap data might have a larger conditional $\psi_\beta$-Orlicz norm than the original data. 
 Additionally, we  need some  conditions on the  entries
 $h_{i}(X_{1}, \ldots ,X_{1},X_{m-k}, \ldots ,X_m)$  for  all $1 \leq k \leq m$,
 which are known as von-Mises conditions in
 the literature  \citep[see][for example]{Bickel1981}.  More precisely, we make the following assumption.

\begin{description}
	\item  \textbf{(A1')}
	Let Assumption (A1) hold and assume that the constant 	$ \beta \in (0, 2]$ and the  sequence  $(B_n)_{n \in \mathbb{N}}$  satisfy
	additionally 
	$  \max_{1 \leq i \leq d}  \norm{h_{i}(X_{j_1}, \ldots ,X_{j_m})}_{\psi_\beta} \leq B_n~,
 $
	for all $j_1, \ldots ,j_m \in \{1, \ldots ,n\}$. 
\end{description}

	\begin{theorem}
	\label{Boot0} 
Let Assumptions  (A1') and  (A2) 
	be  satisfied, assume that $\log d=o(n^\gamma)$ with $0 \leq \gamma \leq \frac{1}{2/\beta+1}$	and that 
	\begin{align}
	\label{pb1}
	    \frac{B_n^2(\log(nd))^{5+2/\beta}}{n}+\frac{B_n^3(\log(nd))^{1+2/\beta}}{\sqrt{n}}=o(1)\,, \qquad \nto\,.
	\end{align}
	\begin{itemize}
	    \item[(1)] 
For any   $\alpha\in (0,1)$ it follows that 
	\begin{align}	
	\limsup_{n \to \infty } \underset{F \in {\cal H} _{0,boot}(\Delta)  }{\sup}
	\mathbb{P} \Big (   \mathcal{T}_{n,\Delta}> q_{1-\alpha}^*  \Big ) \leq \alpha
	~,
	\label{hd9a}
	\end{align} 
	where
\begin{align}
    \label{det0ab}
	    {\cal H} _{0,boot}(\Delta)  := \big \{F \in {\cal F} ~|~  \theta_F \in V_0; 
	    ~F \text{ satisfies  Assumptions (A1'), (A2)}
     \big \} ~~~~~~~
\end{align}
	and $V_0$ is defined in \eqref{det0aa}.
    \item[(2)] For  a sufficiently large constant
$c$, which only depends on $\gamma$ and $\beta$,  it follows that 
	\begin{align}
	\label{hd9}
	\lim_{n \to \infty} 
	    \underset{
	    F \in {\cal H} _1 (c(\log (nd))^{1/\beta})}{\inf}\p\Big (   \mathcal{T}_{n,\Delta}> q_{1-\alpha}^*  \Big )=1~,
	\end{align}
	where the set ${\cal H}_1(c)$ is defined in \eqref{det0}.
Moreover,  if the kernel
	$h$ in \eqref{hd3} is bounded, then the set
	$ {\cal H} _1 (c(\log(nd))^{1/\beta})$ in \eqref{hd9} can be replaced by  ${\cal H} _1 (c)$.
		\end{itemize}
\end{theorem}

\begin{remark} \label{rem2}
~~
{\rm 
\begin{itemize}
    \item[(1)] 
     Note that the sets 
     ${\cal H} _{0} (\Delta )$  and  ${\cal H} _{0,boot}  (\Delta )$
     defined in \eqref{det0a} and \eqref{det0ab}, respectively, satisfy
    ${\cal H} _{0} (\Delta ) \subset {\cal H} _{0,boot}  (\Delta )$ and that 
    ${\cal H} _{0,boot}  (\Delta )$ does not imply assumption (A3).
    This means part (1) of Theorem  \ref{Boot0}
    holds under weaker assumptions than Theorem \ref{alpha}. In particular it also holds in cases, where the statistic  
    $a_d \mathcal{T}_{n,\Delta} - b_d $ 
does not converge weakly (to a
Gumbel distribution).
   
    \item[(2)]
Comparing the statement \eqref{hd9}
for the power of the bootstrap test \eqref{boottest} with Theorem~\ref{consistency} about  the power of the 
asymptotic test \eqref{hd5}, we observe that for 
unbounded kernels there is an additional factor $(\log(nd))^{1/\beta}$
in the definition of the set of alternatives ${\cal H}_1$.
This factor is a consequence of 
 an inflation in the tails of the conditional distribution 
 of the bootstrap data for unbounded kernels.
 As a consequence the bootstrap test can detect local 
 alternatives converging to the null 
 at the rate $
 (\log(nd))^{1/\beta}\sqrt{(\log d)/n}$ 
 and this rate improves to $
 \sqrt{(\log d)/n}$
in the case of bounded kernels. 
 \item[(3)]  

 We emphasize that the test \eqref{boottest} has similar properties as described in Remark~\ref{rem1} for the test \eqref {hd5}, which uses the quantiles of the Gumbel distribution. In particular, under the null hypothesis \eqref {hd1b} the rejection probability is asymptotically $\alpha$ if $|\theta_{i}|=\Delta$ for all $1\le i  \le d$, and, by Theorem \ref{Boot0}, this is  an upper bound for the rejection probability under the null.  Consequently, the type I error can be much  smaller than $\alpha$ if $|\theta_{i}|$ is substantially smaller than $ \Delta$ for many indices $1 \le i  \le d$, where  the extreme case appears if   $\theta_{i} = 0$ for all $1\le i  \le d$. 
 \item[(4)]  
 Under additional assumptions on the kernel $h$ it can also be proved that the test \eqref{boottest} is optimal in the sense that no other test can detect alternatives converging 
 with a faster rate than $B_n\sqrt{\log (d) /{n}}$
 to the null hypotheses.  
 We give more details and illustrate this property in Section \ref{sec35} for the bivariate dependence measures considered  in Example \ref{ex1}.
 \item[(5)]
 Naive algorithms for calculating higher order $U$-statistics result in prohibitive run times of order $n^m$ already when considering the case  $d=2$. Fortunately there are software packages providing optimized algorithms that calculate  rank based $U$-statistics in time $n \log(n)$, see for instance the R package ``independence'' from  \cite{EvenZohar2020}. Similar techniques can be used to shorten the computation times of the quantities $V_i$ and $\hat{\sigma}_i$ for rank based statistics.
 \end{itemize}
 }
\end{remark}

\begin{remark}[an alternative test]
\label{nv} ~~ \\
{ \rm
A careful inspection of the proofs in the online supplement shows that it is possible 
to construct a bootstrap procedure without normalizing the variance of each component.
To be precise we consider the test statistic 
\begin{align}
\label{bootnv}
 \mathcal{T}^{\text{nv}}_{n,\Delta}:=\sqrt{n} \Big ( \underset{1\leq i\leq d}{\max}\ U^2_{i}-\Delta^2  \Big ) ~,
\end{align}
which is obtained from \eqref{TestStatDefin} by omitting the normalizing factors 
$\hat \sigma_i$. This statistic  does not converge weakly to a Gumbel distribution. 
However,   a Gaussian approximation
and corresponding construction of a bootstrap procedure is still possible.

For  this purpose let $X_1^*,\ldots ,X_n^*$ be drawn with replacement from $X_1,\ldots ,X_n$ and 
define $U_i^*$, $V_i$ and $V_{i, \Delta}$ by  \eqref{hd8}, \eqref{hd11} and \eqref{hd10}, respectively. We then obtain a  bootstrap analogue of the statistic \eqref{bootnv} by  
\begin{align}
  \mathcal{T}_n^{*,\text{nv}}:=\sqrt{n}\max_{1 \leq i \leq d}\{ (U_i^*-V_i+V_{i,\Delta})^2-V_{i,\Delta}^2 \}
\end{align}
and denote by $q^{*,\text{nv}}_{1-\alpha}$
the corresponding $(1-\alpha)$-quantile. The null hypothesis in \eqref{hd4} is rejected, whenever
\begin{align} \label{boottestnv}
\mathcal{T}_n^{*,\text{nv}}
>q^{*,\text{nv}}_{1-\alpha}\,.
	\end{align}
For this test an analogue of Theorem \ref{Boot0} can be proved which even allows us to relax condition \eqref{pb1} slightly as we do not need to take into account errors that are incurred by approximating the variances anymore. At the cost of a slightly worse (but still logarithmic) dependence on $p$ we can also substantially weaken Assumption (A2) using recent  Gaussian approximation results from \cite{Chetverikov2020}, which only require the inequality in (A2) to hold for a single coordinate $i$.  The details are omitted for the sake of brevity. 
}
\end{remark}

The numerical results in  Section \ref{sec4} indicate that the bootstrap  tests \eqref{boottest} and \eqref{boottestnv} tend to exceed the desired significance level for some constellations of $(n,p)$. A better finite sample performance can be obtained by
replacing the the statistics $U_i^2$ by their positive square roots  $|U_i|$.
To be precise, we consider 
the non-normalized case and  define the test statistic by 
\begin{align}
\label{pb13}
    \mathcal{T}^{\text{abs}}_{n,\Delta}:=\sqrt{n} \Big ( \underset{1\leq i\leq d}{\max}\ \vert U_{i}\vert -\Delta \Big ) ~.
\end{align}
In this case
one observes  a property similar to \eqref{pb10} that facilitates the application of a Gaussian approximation. 
In particular, whenever  $\theta_i \not = 0 $, we have 
\begin{align}
\label{pb11}
    \vert U_i \vert - \Delta = \text{sign}  (U_i)  U_i -\Delta = \text{sign}  (\theta_i)   (U_i - \theta_i) +   (\text{sign}  (\theta_i)\theta_i -\Delta)
\end{align}
with high probability. Therefore 
a valid bootstrap procedure is obtained as follows.
Let $X_1^*,\ldots ,X_n^*$ be drawn with replacement from $X_1,\ldots ,X_n$ and recall  definitions \eqref{hd8}, \eqref{hd11} and \eqref{hd10}. We then define the bootstrap statistic as
\begin{align}
\label{pb14}
    \mathcal{T}_{n,\Delta}^{*,\text{abs}}:=\sqrt{n}\max_{1 \leq i \leq d}\{ \vert  U_i^*-V_i+V_{i,\Delta}\vert-\vert V_{i,\Delta}\vert \}
\end{align}
and denote by $q^{*,\text{abs}}_{1-\alpha}$ its $(1-\alpha)$ quantile. The null hypothesis \eqref{hd4} is rejected, whenever
\begin{align}
\label{boottestabs}
    \mathcal{T}_n^{*,\text{abs}}
>q^{*,\text{abs}}_{1-\alpha}\,.
\end{align}

  For this test one can obtain the following analogue of Theorem \ref{Boot0}.
\begin{theorem}
	\label{BootAbs} 
Let the assumptions of Theorem \ref{Boot0} be satisfied.
	\label{pb:1}
\begin{itemize}
	    \item[(1)] 
For any   $\alpha\in (0,1)$ it follows that 
	\begin{align}	
	\limsup_{n \to \infty } \underset{F \in {\cal H} _{0,boot}(\Delta)  }{\sup}
	\mathbb{P} \Big (   \mathcal{T}^{\text{abs}}_{n,\Delta}> q_{1-\alpha}^{*,\text{abs}} \Big ) \leq \alpha
	~,
	\label{hd:9a}
	\end{align} 
	where ${\cal H} _{0,boot}(\Delta) $  
	is defined in defined in \eqref{det0ab}.
    \item[(2)] For  a sufficiently large constant
$c$, which only depends on $\gamma$ and $\beta$,  it follows that 
	\begin{align}
	\label{hd:9}
	\lim_{n \to \infty} 
	    \underset{
	    F \in {\cal H} _1 (c(\log (nd))^{1/\beta})}{\inf}\p\Big (    \mathcal{T}^{\text{abs}}_{n,\Delta}> q_{1-\alpha}^{*,\text{abs}}  \Big )=1~,
	\end{align}
	where the set ${\cal H}_1(c)$ is defined in \eqref{det0}.
Moreover,  if the kernel
	$h$ in \eqref{hd3} is bounded, then the set
	$ {\cal H} _1 (c(\log(nd))^{1/\beta})$ in \eqref{hd:9} can be replaced by  ${\cal H} _1 (c)$.
		\end{itemize}
\end{theorem}

\begin{remark}[Testing various thresholds and confidence intervals]
\label{conf}
    {\rm 
    Note that the hypotheses  in \eqref{hd4} are nested.
    Recalling the definition of the statistic $\mathcal{T}^{abs}_{n,\Delta}$ in \eqref{pb13}, it is clear that  the function $\Delta \to \mathcal{T}^{abs}_{n,\Delta}$  is  decreasing. Moreover, if $q^{*,abs}_{1-\alpha,1}$ and $q^{*,abs}_{1-\alpha,2}$ are  the $(1-\alpha)$ quantiles of the statistic $\mathcal{T}^{*,abs}_n$  in \eqref{pb14}
    for  $\Delta=\Delta_1$ and $\Delta=\Delta_2$, respectively, with $\Delta_1<\Delta_2$, it can be shown that 
    \begin{align*}
        q^{*,abs}_{1-\alpha,2}=q^{*,abs}_{1-\alpha,1}+o_\p(1)
    \end{align*}
    as $n\to \infty$.  This yields that the inequality $\mathcal{T}^{abs}_{n,\Delta_2}\geq q^{*,abs}_{1-\alpha,2}$ implies
    \begin{align}
        \mathcal{T}^{abs}_{n,\Delta_1}>\mathcal{T}^{abs}_{n,\Delta_2}\geq q^{*,abs}_{1-\alpha,2} =   q^{*,abs}_{1-\alpha,1}~+o_\p(1),
    \end{align}
  Consequently,  rejecting $H_0$ 
  by the test \eqref{boottestabs}
  for $\Delta= \Delta_0$ also yields (asymptotically) rejection of $H_{0}$ 
  for all  $\Delta>\Delta_0$.
  By the sequential rejection principle, we may simultaneously test the  hypotheses  in \eqref{hd1b} for different $\Delta \geq 0$ 
  starting at $\Delta  = 0$ and 
   increasing  $\Delta $ to 
   find the minimum value of $\Delta $, say 
   $$
   \hat \Delta_\alpha:=\min \big \{\Delta \ge 0 \,| \, \mathcal{T}^{abs}_{n,\Delta}\leq q^{*,abs}_{1-\alpha} \big  \} = \max_{i=1}^d |U_i|-\frac{q^{*,abs}_{1-\alpha}}{n^{1/2}}
   $$ 
   for which 
   $H_0$  in \eqref{hd4} is not rejected.
 In particular, (asymptotically) the null hypothesis
  is accepted  for all  thresholds 
  $ \Delta \geq  \hat \Delta_\alpha $ and rejected for 
  $ \Delta <   \hat \Delta_\alpha $.  Therefore $\hat \Delta_\alpha $ could be interpreted as a measure of evidence against the null hypothesis in \eqref{hd1b}.
Moreover, by a careful inspection of the proofs in the online supplement  we obtain 
that
$$
\lim_{n \to \infty  } \mathbb{P}
\Big ( \sqrt{n} \Big (\max_{i=1}^d   |U_i|
- \max_{i=1}^d  | d_{i} | \Big ) \leq  q^{*,abs}_{1- \alpha}
\Big)  \geq  1- \alpha~. 
$$
Consequently, an 
asymptotic one-sided $(1-\alpha)$-
confidence interval for $\max_{i=1}^d  | d_{i} |$ given by 
\begin{align*}
    [\hat \Delta_\alpha,\infty)=   \Big [\max_{i=1}^d   |U_i|-\frac{q^{*,abs}_{1-\alpha}}{n^{1/2}},\infty    \Big )~.
\end{align*}
In this sense  the question of a reasonable
choice of the threshold $\Delta$ may be postponed until after seeing the data.  We also emphasize that similar arguments can be applied for the tests discussed in Theorem \ref{Boot0} and  Remark \ref{nv}.
}
\end{remark}

\begin{remark}[Reversed  hypotheses] \label{rem11}

{ \rm
As mentioned in the introduction the theory can be extended for testing the  reversed hypotheses 
\begin{align}
\label{hdx1}
   H_0^{\text{int}}:\max_{i=1}^d\vert \theta_i\vert \geq  \Delta \quad \text{versus} \quad H_1^{\text{int}}:\max_{i=1}^d\vert \theta_i\vert <  \Delta \,.
\end{align} 
These hypotheses are of interest if 
one wants to work under the independence assumption. In this case testing the classical hypotheses in \eqref{hd1a} is not helpful, as we cannot control the type II error. However, by  testing the hypotheses 
\eqref{hdx1} with  a  rather small threshold $\Delta$ we can decide  at a controlled type I error that  we are close to mutual independence (measured by the size of $|d_{ij}|$). In this case the threshold can  be determined studying the robustness properties of the subsequent inference  tools which use  the independence assumption.

For the sake of brevity we restrict ourselves to a bootstrap 
test in the spirit of Remark \ref{nv}, which rejects the null hypothesis in \eqref{hdx1}, whenever 
\begin{align} \label{boottestb}
	\mathcal{T}_{n,\Delta} 
	<   q^{**}_{\alpha}\,,
	\end{align}
	where the statistic $\mathcal{T}_{n,\Delta} $  is defined as
 \begin{align}
     \mathcal{T}_{n,\Delta} :=\sqrt{n}\underset{1\leq i\leq d}{\max}\ \vert U_{i} \vert -\Delta
 \end{align}
 and  
	the bootstrap quantile 
 $q^{**}_{\alpha}$
 is obtained as follows.
Let $X_1^*,...,X_n^*$ be drawn with replacement from $X_1,..,X_n$,  recall the definitions \eqref{hd8} and  \eqref{hd11} and 
replace the definition of $V_{i,\Delta}$  in  \eqref{hd10}, 
 by 
\begin{align*} 
V_{i,\Delta}=
    \begin{cases}
    V_i\,,  \quad \text{ if} \ |V_i| >  \Delta\\
    \Delta\,, \quad \text{ otherwise}
\end{cases}  \qquad i=1, \ldots , d~.
\end{align*}

 We then define the statistic 
\begin{align*}
 \mathcal{T}_n^{*,\text{int}}:= \sqrt{n} \left( \vert U_{i_0}^*-V_{i_0}+V_{{i_0},\Delta}\vert-\vert V_{{i_0},\Delta}\vert \right)\,,
\end{align*}
where $i_0$ is an index for which $\max_i |U_i|=|U_{i_0}|$ and we denote by 
 $q^{**}_{\alpha}$ the  $\alpha $-quantile of $\mathcal{T}_n^{*,\text{int}}$.
Using similar arguments as given in the proof of Theorem  \ref{Boot0} 
we can show that the decision rule \eqref{boottestb}
defines a (uniformly) consistent and asymptotic level $\alpha$ test for the hypotheses \eqref{hdx1}.
}
\end{remark}

\begin{remark}[Classical hypotheses]
\label{class hypotheses}
{\rm  With the choice $\Delta =0$ 
the non-normalized bootstrap test
 \eqref{boottestabs} can also be used for testing the classical hypotheses in \eqref{hd1a},
provided that  
the representation \eqref{hd0} for 
the parameter of interest  holds 
 with a $U$-statistic which is non-degenerate 
under the null hypothesis. 
This follows by a careful inspection of  the arguments given in the proofs of Theorem \ref{Boot0} in the online supplement. In such cases this test provides an alternative to the tests constructed by asymptotic arguments, see for instance \cite{hanetal2017, Zhou2019} and \cite{drttonetal2020}. Numerical results, which are available from the authors, indicate some advantages of the test \eqref{boottestabs}  for larger sample sizes. However, for small sample sizes the tests of \cite{hanetal2017} show a better performance. }
\end{remark}

\begin{remark}[One-sided hypotheses]
 \rm As pointed out by a referee,  for signed dependence measures it is also of interest to  consider the hypotheses of at least one  relevant positive (or negative) dependence. In our general formulation of the testing problems this corresponds to the hypotheses 
\begin{align}
\label{rev1}
    H_0^+:\max_{i=1}^d \theta_i \leq \Delta \quad \text{ and } \quad H_0^-:\min_{i=1}^d \theta_i \geq -\Delta~. 
\end{align}
Natural test statistics for $H_0^+$ and $H_0^-$ are given by 
\begin{align}
\label{rev2}
    T_{n}^+=\sqrt{n}\max_{i=1}^d( U_i - \Delta)~
    \text{~~and ~~} T_{n}^- = \sqrt{n}\max_{i=1}^d(- U_i - \Delta), 
\end{align}
respectively, for which quantiles can be obtained using the Gaussian multiplier 
bootstrap proposed in \cite{UApprox}.  Testing each of these hypotheses separately at level $\alpha/2$ and rejecting ${H}_{0}$ in \eqref{hd4} 
if one of them rejects, naturally yields a test that has asymptotic level $\alpha$. By construction, this test is  conservative   and as a consequence less powerful compared to ours. We have confirmed  the superiority of our approach  also for finite samples by means of a small simulation study. These results are not displayed for the sake of brevity.
%

\end{remark}

	\section{Relevant dependencies in high-dimension}
	\label{sec3}
	  \def\theequation{3.\arabic{equation}}	
	\setcounter{equation}{0}

In this section we apply the methodology  in the context of bivariate  dependence
measures between the components of high-dimensional vectors as 
considered in the introduction.  The relation between 
this problem  and 
the general formulation in Section \ref{sec2} is described in Example \ref{ex1}. 
Recall the definition of the  dependence measure in \eqref{hd21} for the kernel  $\tilde h$, the notation $X_k=(X_{k1}, \ldots , X_{kp})^\top $
and write  
 \begin{align}
\label{hd22}
 U_{ij} & = {n \choose m}^{-1}\!\!\sum_{1 \leq l_1 <... < l_m \leq n} \!\!  h_{ij} ( X_{l_1},...,X_{l_m}) \\
 &
\nonumber  = {n \choose m}^{-1}\!\!\sum_{1 \leq l_1 <... < l_m \leq n}\!\! \tilde  h ( X_{l_1i},X_{l_1j},...,X_{l_mi},X_{l_mj}) \, ~~~~~~
\end{align}
for the corresponding $U$-statistic,
where the second equality defines the functions $h_{ij} : \mathbb{R}^{pm} \to \mathbb{R}$ in an obvious manner. We now discuss several dependence measures separately.  For the sake of brevity we restrict ourselves to the bootstrap test introduced in Section \ref{sec22}, which is defined by 
\begin{align} \label{boottesta}
	\mathcal{T}_{n,\Delta}>q^*_{1-\alpha}~,
	\end{align}
where $\mathcal{T}_{n,\Delta} =
	\max_{1\leq i< j \leq p} (U^2_{ij}-\Delta^2)/(2
	\hat{\sigma}_{ij}\Delta)$ and $q^*_{1-\alpha}$ denotes the $(1-\alpha)$-quantile of the corresponding bootstrap distribution.

\subsection{Covariance} \label{sec31}

   The sample covariance matrix
    $$
    \big ( \hat \Sigma_{ij} \big )_{i,j=1,\ldots , p} =\frac{1}{n-1}\sum_{k=1}^{n}(X_k-\bar{X}_n)(X_k-\bar{X}_n)^\top\,, 
    $$
where $\bar{X}_n = {1\over n} \sum_{k=1}^n X_k$ denotes the
   sample mean of $X_1, \ldots , X_n$,
      is the commonly used  unbiased estimate for the  covariance matrix $\Sigma= {\rm Cov}_F (X_1) = \mathbb{E}_F \big  [ (X_1 - 
    \mathbb{E}_F [ X_1]) (X_1 - 
    \mathbb{E}_F [ X_1])^\top \big  ]$. 
  
The covariance is a special case of \eqref{hd22}  choosing  $h(x_1,x_2)=(x_1-x_2)(x_1-x_2)^\top /2$,
and we refer to 
 \cite{baietal2009,chen2017}, among others, who considered
  independence testing of
 the  classical hypotheses in  \eqref{hd1a} for covariances.
 We now consider the problem of testing the relevant 
 hypotheses \eqref{hd1b}, where $d_{ij} = {\rm Cov} 
 (X_{1i}, X_{1j})$, $1  \leq i < j \leq p$.
For the problem of testing relevant hypotheses of the form \eqref{hd1b}
an application of  the results of Section  \ref{sec22} yields the following result.

 \begin{corollary}
If $\log p=o(n^\gamma)$ 
 with $ 0 \leq \gamma \leq  ({5+4/\beta})^{-1} \land ({2+8/\beta}))^{-1}$, then the bootstrap test 
\eqref{boottesta} with $U_{ij} =\hat \Sigma_{ij} $
is (uniformly) consistent and has (uniform) asymptotic level $\alpha$ over the classes of distributions $\mathcal{H}_1(c(\log(nd))^{2/\beta})$ and 
$\mathcal{H}_{0,boot}(\Delta)$  defined in \eqref{det0} and \eqref{det0ab} respectively, where the conditions  (A1') and (A2) have to be replaced (and are implied) by 
\begin{itemize}
     \item[(C1)] There exist
     constants $ \beta \in (0,2] $  and  $C> 0 $ such that for all  $p=p(n), n \in  \mathbb{N}$ 
	\begin{align*}
	\max_{1 \leq i \leq p}
	    \norm{X_i-\E[X_i]}_{\psi_\beta}&\leq C\,.
	    \end{align*} 
	 \item[(C2)] For some constant $\underline{b}>0$ and $c \in (0,\Delta)$ we have
	 \begin{align*}
	     \min_{1 \leq i < j \leq p,|\Sigma_{ij}|>c}\text{Var}_F[(X_{1i}-\E_F [X_{1i}])(X_{1j}-\E_F [X_{1j}])]\geq \underline{b}\,
	 \end{align*}
		 for all $p=p(n), n \in \mathbb{N}$.
 \end{itemize}
 Note that for a normal distribution   Assumption (C2) holds whenever there exists a uniform positive lower bound for the diagonal elements of $\Sigma$.
 \end{corollary}

\subsection{Kendall's $\tau$} \label{sec32}

    A very popular measure of (monotonic) dependence between the $i$th and $j$th component of  the vector $X_1=(X_{11}, \ldots , X_{1p})^\top $
      is  Kendall's  $\tau$ coefficient given by $\tau_{ij}=\E_F[\text{sign}(X_{1i}-X_{2i})\text{sign}(X_{1j}-X_{2j})]$ with empirical version
\begin{align*}
\hat \tau_{ij}= 
\frac{2}{n(n-1)}\sum_{1\leq k<l\leq n}\text{sign}(X_{ki}-X_{li})\, \text{sign}(X_{kj}-X_{lj})~.
\end{align*}
Here the kernel is given by 
$$
h_{ij}(x_1,x_2)=   \tilde h ( x_{1i},x_{1j},x_{2i},x_{2j}) = 
\text{sign}(x_{1i}-x_{2i})\text{sign}(x_{1j}-x_{2j})
$$
and the vector $U$ is defined by
$U={\rm vech} \big ( ( \hat{\tau}_{ij} )_{i,j=1, \ldots , p}  \big ).$
The classical testing problem \eqref{hd1a} 
 with $d_{ij} = \mathbb{E}_F[ \text{sign}(X_{1i}-X_{2i})\text{sign}(X_{1j}-X_{2j})]$
  was considered by  \cite{hanetal2017,Leung2018,Zhou2019} and \cite{Li2021}
 in the high dimensional regime. 
For the problem of testing relevant hypotheses of the form \eqref{hd1b}
an application of  the results of Section  \ref{sec22} yields the following result.

 \begin{corollary}
 If $\log p=o(n^\gamma)$ 
holds with $ 0\leq \gamma \leq \frac{1}{6}$, then the bootstrap test 
\eqref{boottesta} with $U_{ij} =\hat \tau_{ij} $
is (uniformly) consistent and has (uniform) asymptotic level $\alpha$ over the classes of distributions $\mathcal{H}_1(c)$ and 
$\mathcal{H}_{0,boot}(\Delta)$  defined in \eqref{det0} and \eqref{det0ab} respectively, where 
condition (A1')  can be omitted (because the kernel is bounded) and 
condition (A2) is replaced by 
 \begin{itemize}
	 \item[(T1)] 
	 There exist constants  $\underline{b}>0$ and $c \in (0,\Delta)$ such that 
	 \begin{align*}
	    \min\limits_{1 \leq i<j \leq p, |\tau_{ij}|>c}\text{Var}_F\left[ \E_F[\text{sign}(X_{1i}-X_{2i})\text{sign}(X_{1j}-X_{2j})|X_1]\right] \geq \underline{b}\,.
	 \end{align*}
	 for all $p=p(n), n \in \mathbb{N}$.
 \end{itemize}
 \end{corollary}

\subsection{The dominating term of Spearman's $\rho$} \label{sec323}

 Let $Q^i_{nk}$ be the rank of $X_{ki}$ among  $X_{1i},...,X_{ni}$ and  consider
Spearman's rank correlation coefficient
\begin{align*}
   \rho_{ij}=\frac{\sum_{k=1}^n(Q^i_{nk}-(n+1)/2)(Q^j_{nk}-(n+1)/2)}{\sqrt{\sum_{k=1}^n(Q^i_{nk}-(n+1)/2)^2\sum_{k=1}^n(Q^j_{nk}-(n+1)/2)^2}}\, ,
\end{align*}
which defines 
  another popular measure of dependence between the $i$th and $j$th component of the vector $X_1=(X_{11},\ldots ,X_{1p})^\top$. While $\rho_{ij}$ ist not  a $U$-statistic, it
was shown by \cite{Hoeffding1948a} that it can be decomposed as follows 
\begin{align*}
    \rho_{ij}=\frac{n-2}{n+1}\hat{\rho}_{ij}+\frac{3}{n+1} \hat\tau_{ij}\,,
\end{align*}
where the dominating term 
\begin{align*}
    \hat{\rho}_{ij}=\frac{6}{n(n-1)(n-2)}\sum_{1 \leq k_1<k_2<k_3\leq n}\text{sign}(X_{k_1i}-X_{k_2i})\text{sign}(X_{k_1j}-X_{k_3j})
\end{align*}
is a $U$-statistic of degree $3$ with  bounded kernel 
$$
h_{ij}(x_1,x_2,x_3)= 
\tilde h ( x_{1i},x_{1j},x_{2i},x_{2j},x_{3i},x_{3j}) = 
\text{sign}(x_{1i}-x_{2i})\text{sign}(x_{1j}-x_{3j}).
$$
The classical testing problem for this statistic and continuous data was considered by \cite{hanetal2017} and \cite{ Leung2018}. For the problem of testing relevant hypotheses of the form \eqref{hd1b}
an application of  the results of Section  \ref{sec22} yields the following result.

 \begin{corollary}
If $\log p=o(n^\gamma)$ 
holds with $  0\leq \gamma \leq \frac{1}{6}$, then the bootstrap test 
\eqref{boottesta} with $U_{ij} =\hat{\rho}_{ij} $
is (uniformly) consistent and has (uniform) asymptotic level $\alpha$ over the classes of distributions $\mathcal{H}_1(c)$ and 
$\mathcal{H}_{0,boot}(\Delta)$  defined in \eqref{det0} and \eqref{det0ab} respectively, where 
condition (A1')  can be omitted (because the kernel is bounded) and 
condition (A2) is replaced by 
\begin{itemize}
	 \item[(S1)] There exist constants  $\underline{b}>0$ and $c \in (0,\Delta)$ such that 
	 \begin{align*}
	    \min\limits_{1 \leq i<j \leq p, |\rho_{ij}|>c}\text{Var}_F\left[ \E_F[\text{sign}(X_{1i}-X_{2i})\text{sign}(X_{1j}-X_{3j})|X_1]\right] \geq \underline{b}\,.
	 \end{align*}
	 for all $p=p(n), n \in \mathbb{N}$.
 \end{itemize}
 \end{corollary}

\subsection{Dependence measures with degenerate kernel }
\label{sec34}

While Kendall's $\tau$ and Spearman's $\rho$ only capture monotonic dependencies between two random variables there are a number of higher order $U$-statistics that are able to capture any form of dependency between two random vectors.  Exemplary, we mention here Hoeffding's $D$ 
\citep{Hoeffding1948}, Blum-Kiefer-Rosenblatt's $R$
\citep{blum1961} and 
Bergsma–Dassios–Yanagimoto’s $\tau^*$ \citep{bergsma2014}. 
Note that in the case of independence (reflecting  the classical null hypothesis in \eqref{hd1a})
the kernels corresponding to these 
$U$-statistics are degenerate. 
On the other hand, if the components are dependent
(which corresponds to the classical alternative), 
 all three statistics are  non-degenerate  for a large class of distributions.  
In such  cases 
the general theory  developed 
in Section \ref{sec3}  is applicable as well.
Before going into details we emphasize that similar results as presented below can be derived for other types of dependence measures which can be estimated by $U$-statistics with a degenerate kernel under 
 independence  such as the distance correlation introduced by \cite{Szekely2007}, see Theorem 4.1 in \cite{Edelmann2021}.
 
To be precise  we recall the definition of the $U$-statistics considered 
in \cite{Hoeffding1948,blum1961,bergsma2014}.
Let $z_1,...,z_6$ be $p$-dimensional vectors 
of the form $z_i = (z_{i1} , \ldots  , z_{ip})^\top $, define
 \begin{align*}
     \1_{j_1,j_2,j_3}^k &:= \1\{z_{j_1k}\leq z_{j_3k}\}-\1\{z_{j_2k}\leq z_{j_3k}\} ~, \\
     \1_{j_1,j_2}^{j_3,j_4,k} &:= \1\{z_{j_1k}<z_{j_3k}\}\1\{z_{j_1k}<z_{j_4k}\}\1\{z_{j_2k}<z_{j_3k}\}\1\{z_{j_2k}<z_{j_4k}\}~, 
      \end{align*}
     and consider the  kernels 
      \begin{align*}
     h_{ij}^D(z_1,...,z_5) &:= \frac{1}{16}\sum_{1 \leq j_1 \neq ... \neq j_5\leq 5}\1_{j_1,j_2,j_5}^i\1_{j_3,j_4,j_5}^i\1_{j_1,j_2,j_5}^j\1_{j_3,j_4,j_5}^j~, \\
     h_{ij}^R(z_1,...,z_6) &:= \frac{1}{32}\sum_{1 \leq j_1 \neq ...\neq j_6\leq 6}\1_{j_1,j_2,j_5}^i\1_{j_3,j_4,j_5}^i\1_{j_1,j_2,j_6}^j\1_{j_3,j_4,j_6}^j~, \\
        h_{ij}^{\tau^*}(z_1,...,z_4) &:= \frac{1}{16}\sum_{1 \leq j_1\neq ... \neq j_4\leq 4}(\1_{j_1,j_3}^{j_2,j_4,i}+\1_{j_2,j_4}^{j_1,j_3,i}-\1_{j_1,j_4}^{j_2,j_3,i}-\1_{j_2,j_3}^{j_1,j_4,i})
     ~, \\ 
     & \qquad \qquad \qquad \qquad \times (\1_{j_1,j_3}^{j_2,j_4,j}+\1_{j_2,j_4}^{j_1,j_3,j}-\1_{j_1,j_4}^{j_2,j_3,j2}-\1_{j_2,j_3}^{j_1,j_4,j})~.
 \end{align*}
Note that  $h_{ij}^D$, $h_{ij}^R$
and $ h_{ij}^{\tau^*}$ define symmetric kernels of orders $5,6$ and $4$ respectively. The corresponding
 matrices of empirical  dependence measures  calculated from the sample $X_1,...,X_n \in \R^p$  are then given by 
 \begin{align}
 \notag
     \hat  D=(\hat  D_{ij})_{1 \leq i<j\leq p} =\Big({n \choose 5}^{-1}\sum_{1 \leq j_1 < \ldots  < j_5\leq n}h_{ij}^D(X_{j_1},\ldots ,X_{j_5})\Big)_{1 \leq i < j \leq p}\,,\\
    \notag
   \hat  R=(\hat  R_{ij})_{1 \leq i<j\leq p}=\Big({n \choose 6}^{-1}\sum_{1 \leq j_1 < \ldots  < j_6 \leq n }h_{ij}^R(X_{j_1},\ldots ,X_{j_6})\Big)_{1 \leq i < j \leq p}\,.  
 \label{HigherOrder} \\
     \notag
   \hat  \tau^*=(\hat  \tau^*_{ij})_{1 \leq i<j\leq p}=\Big ( {n \choose 4}^{-1} \sum_{1 \leq j_1 < \ldots  < j_4\leq n }h_{ij}^{\tau^*}(X_{j_1},\ldots ,X_{j_4})\Big )_{1 \leq i < j \leq p}\,. \notag
 \end{align}
The classical testing problem \eqref{hd1a}, where 
the dependence measure 
$d_{ij}$ is  either given by  $
D_{ij}=\mathbb{E}_F[ h_{ij}^D(X_{1},\ldots ,X_{5})]$, 
$R_{ij}=\mathbb{E}_F[ h_{ij}^R (X_{1},\ldots ,X_{6})]$ or  $ \tau^*_{ij}=\mathbb{E}_F[ h_{ij}^{\tau^*}(X_{1},\ldots ,X_{4})]$ 
 was considered by \cite{drttonetal2020} in the high dimensional regime.
For the problem of testing relevant hypotheses of the form \eqref{hd1b}
an application of  the results of Section  \ref{sec22} yields the following result.

 \begin{corollary}
If $\log p=o(n^\gamma)$ 
holds with $  0 \leq \gamma \leq \frac{1}{6}$, then the bootstrap test \eqref{boottesta} with $U_{ij}$ given by either
  $ \hat D_{ij}$, $\hat R_{ij}$ or   $\hat \tau^*_{ij}$ is (uniformly) consistent and has (uniform) asymptotic level $\alpha$ over the classes of distributions $\mathcal{H}_1(c)$ and $\mathcal{H}_{0,boot}(\Delta)$  defined in \eqref{det0} and \eqref{det0ab} respectively, where
condition (A1')  can be omitted (because the kernels are bounded) and condition (A2) is replaced 
by
\begin{itemize}
	 \item[(D1)] 
	 There exist constants  $\underline{b}>0$ and $c \in (0,\Delta)$ such that 
	 \begin{align*}
	     \min\limits_{1 \leq i<j \leq p, |D_{ij}|>c}\text{Var}_F\left[ \E_F[h_{ij}^D(X_1,...,X_5)|X_1]\right] \geq \underline{b}\,.
	 \end{align*}
	 for all $p=p(n), n \in \mathbb{N}$. 
	 \end{itemize}
	 in the case of Hoeffding's $D$, by 
\begin{itemize}
 \item[(R1)]
	 There exist constants  $\underline{b}>0$ and $c \in (0,\Delta)$ such that 
	 \begin{align*}
	    \min\limits_{1 \leq i<j \leq p, |R_{ij}|>c} \text{Var}_F\left[ \E_F[h_{ij}^R(X_1,...,X_6)|X_1]\right] \geq \underline{b}\,.
	 \end{align*}
	 for all $p=p(n), n \in \mathbb{N}$.
 \end{itemize}
	 in the case of Blum-Kiefer-Rosenblatt's $R$,  and by 
\begin{itemize}
 \item[(TA1)]
	  There exist constants  $\underline{b}>0$ and $c \in (0,\Delta)$ such that 
	 \begin{align*}
	     \min\limits_{1 \leq i<j \leq p, |\tau^*_{ij}|>c}\text{Var}_F\left[ \E_F[h_{ij}^{\tau^*}(X_1,...,X_4)|X_1]\right] \geq \underline{b}\,.
	 \end{align*}
	 for all $p=p(n), n \in \mathbb{N}$.
	 \end{itemize}
for Bergsma–Dassios–Yanagimoto’s $\tau^*$.
 \end{corollary}

\subsection{Minimax optimality}
\label{sec35}

Recall that, by Theorems \ref{consistency} and \ref{Boot0}, both the asymptotic test and the bootstrap test (under the additional assumption of a bounded kernel)
correctly reject the null hypothesis  in \eqref{hd4}  if at least one entry of the vector $\theta$ is  larger  than $\Delta + C B_n\sqrt{\log (d) /{n}}$. In this section we will show that in many situations, where the sequence  $(B_n)_{n \in \mathbb{N}}$ is bounded
this rate cannot be improved.
These cases include all dependence measures 
discussed in Sections \ref{sec31}  --  \ref{sec34}.
To be precise, we define  
$$
\mathcal{T}_\alpha:= \big \{T_\alpha ~|~
{\sup}_{F \in {\cal H} _0(\Delta)  }
\P(T_\alpha \text{ does not reject } H_0 )\leq \alpha \big \}
$$
as the  set of all  tests
with  (uniform) level $\alpha$. 
   
    We begin with a result for the covariances, that is $d_{ij}={\rm Cov}_F (X_{1i},X_{1j}) $ ($1 \leq i < j \leq p $. 
    For the sake of simplicity, we  assume without loss of generality that $d_{ii}= {\rm Var} (X_{1i})=1$ ($i=1, \ldots , p)$, the general case is obtained by a scaling argument. Note that in this 
    case only values $\Delta  \in (0,1)$ are useful thresholds for the hypotheses \eqref{hd4}.
We then obtain the 
following result.

\begin{theorem}
\label{optimality}
    Assume that
    the dependence measure $d_{ij}$ in \eqref{hd21} is 
    given by $d_{ij}={\rm Cov}_F (X_{1i},X_{1j})$ and $d_{ii}=1$ ($i,j=1, \ldots , p$); so we have $d=p(p-1)/2$.  Further let $c_0, \alpha,\beta $ denote 
     positive constants
such that   $c_0<{1-\Delta}  $
   and  $\alpha+\beta<1$. If  $\log(p)/n\rightarrow 0$ and $\log(p)n/p^2\rightarrow 0$, as $n\to\infty$, then we have for  sufficiently 
   large $n$ and $p$
   \begin{equation}
    \label{hd42}
       \inf_{T_\alpha \in \mathcal{T}_\alpha}\sup_{F \in \mathcal{H}_1(c_0)}\p(
    T_\alpha \text{ does not reject } H_0  
    )\geq 1-\alpha-\beta\,.
  \end{equation}
  
\end{theorem}

 The proof of \eqref{hd42} 
 uses the fact that the supremum of the probabilities 
 with respect to the distributions $F \in \mathcal{H}_1(c_0)$ 
 can be bounded  from below
 by the supremum taken over all centered multivariate normal distributions in $\mathcal{H}_1(c_0)$, where 
 the covariance matrices have the following form. All diagonal elements are $1$, except for two off-diagonal elements all off-diagonal elements are 
 equal to $\Delta$ and the two remaining off-diagonal elements
 are given by $\Delta + \rho $.
 Because this argument 
 does not depend on the specific dependence measure under consideration, a careful  inspection of the proof of  Theorem \ref{optimality}
shows that statements of the form \eqref{hd42} 
are also available for 
dependence measures, 
which, under the assumption of a normal distribution, can be represented as a function of the correlation.  More precisely,
let $d_{ij}(F)= d(X_{1i},X_{1j})$
denote a bivariate dependence measure, such that 
\begin{equation} \label{xx1}
d_{ij}(N_1,N_2)  =    g(\rho)
\end{equation} 
for a normal distributed vector 
$(N_1,N_2)^\top \sim   {\cal N}_2 \big  ( 0 ,
\begin{tiny}
\begin{pmatrix}
1 & \rho \\ \rho &1 
\end{pmatrix} 
\end{tiny}
\big  )$, where   $g: (-1,1) \to \mathbb{R} $ 
is a  differentiable  function with 
non-vanishing derivative at some $\rho\in g^{-1}(\{\Delta\})$.
\begin{corollary}
\label{cor35}
The conclusion of Theorem \ref{optimality} 
remains valid for any bivariate dependence measure
$d_{ij}$, which  satisfies \eqref{xx1}
and for which there exists 
a constant  $ a \in (-1,1)$ such that 
  $|g(a)| =\Delta$  
and $\text{\rm sign}(g'(a))=\text{\rm sign}(g(a))$.
\end{corollary}

\begin{remark}
{\rm We conclude this section with some examples of dependence measures, where Corollary \ref{cor35} is applicable.
Note that   Theorem \ref{optimality}
gives a lower bound for all tests. Thus it
also applicable for dependence measures, which can be estimated by $U$-statistics.

\begin{itemize}
    \item[(1)] A prominent dependence measure that fulfills this assumption is  Kendall's $\tau$ for which it holds that $\tau_{ij}=(2/\pi)\, \text{arcsin}(\rho )$.
A similar result holds  for Spearman's $\rho$, here we have
    $\rho_{ij}=(6/\pi)\,\text{arcsin}(\rho/2 )$.
    Another obvious choice is the Pearson correlation for which $g(\rho)=\rho$ is the identity function.
\item[(2)] 
For  a centered normal 
distribution
Hoeffding's D, Blum-Kiefer-Rosenblatt's R
 and 
Bergsma–Dassios– Yanagimoto’s $\tau^*$s,
which  are considered in Section \ref{sec34}, can be expressed in terms of $\rho$, such that \eqref{xx1} holds. We expect  that 
the assumptions of Corollary~\ref{cor35} are satisfied
as well, but we do not work out the details here for the sake of brevity.
\end{itemize}
}
\end{remark}

	\section{Finite sample properties}
\label{sec4} 
  \def\theequation{4.\arabic{equation}}	
		\setcounter{equation}{0}
		
	In this section we report the results of a small simulation study conducted in order to  investigate the finite sample properties of the proposed tests for the relevant hypotheses \eqref{hd1b}.
	We focus on Kendall's $\tau$
	and  the bootstrap test \eqref{boottest},
its  non-normalized version defined by
	\eqref{boottestnv} and the test 
	\eqref{boottestabs}, which uses the statistics $|U_{ij}|$ instead of their squares $U_{ij}^2$. 
	
	As distributions we consider the  centered $p$-dimensional  normal 
	distribution with covariance matrix $\Sigma$, that is 
	\begin{align}
	    \label{hd25}
	    X_1,\ldots ,X_n \sim \mathcal{N}_p\left(0,\Sigma \right)
	\end{align}
	and the centered  $p$-dimensional  $t$-distribution with $f=3$ degrees of freedom and scale matrix $\Sigma$, that is
	\begin{align}
	    \label{hd26}
	    X_1,\ldots ,X_n \sim t_f\left(0,\Sigma \right)\,, 
	\end{align}	
	with density 
	\begin{align*}
	  g_{f,\Sigma}(x)=\frac{\Gamma((f+p)/2)}{\Gamma(f/2)f^{p/2}\pi^{p/2}|\Sigma|^{1/2}}\Big ( 1+\frac{1}{f}x^\top\Sigma x\Big )^{-(f+p)/2}\,.
	\end{align*}
	We generate data from the models \eqref{hd25} and \eqref{hd26} 
	for sample sizes $n\in\{50, 100\}$   and dimension $p\in \{100,200,400\}$, where we investigate $3$ choices for the covariance matrices $\Sigma$ and $(f/(f-2))\Sigma$ in \eqref{hd25} and \eqref{hd26} respectively, that is 
		\begin{align}
		\label{hd27}\tag{$M1$}
	  &  \text{Diag}_p(1-\rho, \ldots , 1-\rho)+\rho J_p ~, \\
	    		\label{hd28a}\tag{$M2$}
	 &\text{Diag}_p(1, \ldots ,1)+\rho \textstyle{\sum_{1\leq i < j \leq \lfloor p/\sqrt{2} \rfloor }
	 (e_ie_j^\top+ e_je_i^\top)} ~, \\
	 	    	\label{hd28}\tag{$M3$}
	 &\text{Diag}_p(1, \ldots ,1)+\rho (e_ie_j^\top+ e_je_i^\top)~.
	\end{align}
	Here $\text{Diag}_p (a_1, \ldots , a_p) $ denotes a diagonal  $ p \times p$ matrix with diagonal entries 
	$a_1, \ldots , a_p$, $J_p$ denotes the $ p \times p$
	matrix with all entries equal to $1$, $e_j$ is the $j$th standard basis vector and $\rho$ is a constant that varies depending on whether or not on one wants generate data whose Kendall's $\tau$ exceeds the threshold  or not.
	In model \eqref{hd27} we have equal correlation between all components of $X_1$, whereas in model \eqref{hd28} only the $i$th and $j$th components of $X_1$ are correlated. Model  \eqref{hd28a}
	defines an intermediate case  with a block-diagonal correlation matrix, where 
	 the first $\lfloor p/\sqrt{2} \rfloor$ components have the same correlation and the remaining components are uncorrelated.
All numerical results presented in the following discussion 	are based on $1000$ simulation runs and $100$ bootstrap replications.

We investigate  different test for the hypothesis of a relevant deviation from independence 
	between the components of a high-dimensional vector,
	if the dependencies are measured by Kendall's $\tau$, as discussed in Section \ref{sec32}. Thus, the hypotheses are given by 
	\begin{equation}
 \label{kendal0}
  H_0: \max_{1 \le i < j \le p} | \tau _{ij} | \le \Delta  
    \text{\quad versus \quad } H_1 : \max_{1\le i < j \le p} |\tau_{ij}| >  \Delta \, ~,
	\end{equation}
	where we choose the threshold $\Delta=0.1$.
	Note that the distributions in \eqref{hd25}
	and \eqref{hd26}
	are elliptical, which implies the relation 
	 $$\tau_{ij} = \frac{2}{\pi} \arcsin
	 \big ( {\rm Corr} (X_{1i,}, X_{1j})
	 \big  ) $$
	between Kendall's $\tau$ and  the off-diagonal
	elements of the matrices $\Sigma$
	and $(f/(f-2))\Sigma$ in \eqref{hd25} and \eqref{hd26} respectively
	\citep[see][]{Lindskogetal2003}.

	\subsection{Test statistics involving $U_{ij}^2$}
	\label{sec41} 

We begin    studying the type I error of the bootstrap test \eqref{boottest}, which is based on a maximum of normalized 
		statistics involving squares of the $U$-statistics $U_{ij}$.
	As pointed out in Sections \ref{sec2} and \ref{sec3}, the (asymptotic) level of the bootstrap test 
 is substantially smaller than
	the nominal level $\alpha$ if $\max_{1 \leq i < j   \leq p} |\tau_{ij} | < 0.1$. Therefore, we
	concentrate on the case where at least one of the bivariate dependence measures satisfies  $|\tau_{ij} | = 0.1$, 	which corresponds to the choice  $\rho=\sin(\pi/20)$ in  model \eqref{hd27} - \eqref{hd28}. 
	Note that  the matrix in \eqref{hd27} represents the situation, where $|\tau_{ij} | = \Delta = 0.1 $ for all $1 \le i < j \le p$, which corresponds to
	 the ``full boundary'' of the hypotheses \eqref{kendal0}.
	The matrix in \eqref{hd28} represents a case which is closer to the ``interior'' of the null hypothesis (only  two off-diagonal elements have a Kendall's $\tau$ equal to 0.1, but for all other entries  Kendall's $\tau$ is equal to
	$ 0$).
	For the matrix \eqref{hd28a} about $50\%$ of the off-diagonal elements have a Kendall's $\tau$  equal to $0.1$.
 	Therefore, from the discussion in Sections \ref{sec2} and \ref{sec3}, we expect that 
	for model \eqref{hd27} the simulated level should be close to $0.1$, while it should 
	be substantially smaller than $0.1$ in the two other cases.
	Moreover, this effect should be more visible for 	model \eqref{hd28}  than for \eqref{hd28a}.

	\begin{table}[h]
	\centering 
	\begin{tabular}{|c|c|c|c|c|c|c|c|c|c|}
	    \hline
		 $n,p $ & $50,100$ & $50,200$&  $50,400$ & $100,100$ & $100,200$& $100,400$ & $200,100$ & $200,200$& $200,400$ \\
		\hline
	 \ref{hd27}  & 0.081& 0.054 & 0.043 & 0.164 & 0.103 & 0.112 & 0.174 & 0.199 & 0.196 \\
		 \ref{hd28a}  & 0.026& 0.011 & 0.009 & 0.100 & 0.078  & 0.067  &  0.148 & 0.139 & 0.162     \\
		 		 \ref{hd28}  & 0.000& 0.000& 0.000& 0.000 & 0.000 & 0.000 & 0.000 & 0.000 & 0.000  \\
		 \hline
		 \hline
		 \ref{hd27}  & 0.115 & 0.192  & 0.152  & 0.160& 0.275  & 0.337 & 0.265 & 0.274 & 0.302 \\
		 	 \ref{hd28a}  &0.030 & 0.037 & 0.035 & 0.127& 0.124  &   0.123  & 0.148 & 0.139 & 0.162    \\
		 \ref{hd28}  & 0.000&0.000 &0.000 & 0.000 & 0.000 & 0.000 & 0.000 & 0.000 & 0.000  \\
		 \hline
	\end{tabular}
\smallskip 

	\caption{ \it Simulated rejection probabilities of the test \eqref{boottest} under the null hypothesis
	in \eqref{kendal0} (nominal level $\alpha=0.1$). Upper part: multivariate normal distribution;
	lower part: multivariate $t$-distribution with $3$ degrees of freedom.}
	\label{tab1}
	\end{table}

	The corresponding rejection probabilities under the null hypothesis of the test \eqref{boottest} are shown in Table \ref{tab1}. 
For model   \eqref{hd28} 
(only  two off-diagonal elements have a Kendall's $\tau$ equal to $0.1$, but for all other entries  Kendall's $\tau$ is 	$ 0$)
we observe that the 
type I error is approximately $0$. In model 
\eqref{hd28a}  (about $50\%$ of the off-diagonal elements have a Kendall's $\tau$ equal to $0.1$, but for all other entries  Kendall's $\tau$ is 	$ 0$) the type one error
is larger than  for  \eqref {hd28} and provides a reasonable approximation of  the nominal level $\alpha=0.1$ for sample sizes $n=50,100$
(for the $t$-distribution it is slightly too large).
For model \eqref{hd27}
(all of the off-diagonal elements have a Kendall's $\tau$ equal to $0.1$) we only observe a reasonable approximation for the normal distribution and sample sizes $n=50, 100$ (with the exception $n=p=100$),
and the approximation of the nominal level is worse for the $t$-distribution.
A  similar problem was also observed for some tests of classical hypotheses in high dimension \citep[see the discussion in Section 5.2 of 
][]{hanetal2017}.
\\
Note that model \eqref{hd27} represents the ``worst case'' under the null hypothesis (corresponding to the situation $\tau_{ij} = \Delta = 0.1$ for all $1\le i <  j \le p)$ and for many other cases the test will keep its
 nominal level. Moreover, the deviations become smaller
 (but they are still visible) if one is testing the hypotheses \eqref{hd1b} with a smaller threshold  such as $\Delta =0.05$ (these results are not displayed for the sake of brevity).
A  potential explanation of the  observed exceedance in these cases is that the normalization by the variance estimators \eqref{varDef} may yield some instabilities, especially for more heavy tailed data.

 Therefore, we  next investigate the approximation of the nominal level by the non-normalized version of the  test \eqref{Boot0}, which is defined in equation  \eqref{boottestnv} in Remark \ref{nv}. The corresponding empirical type I error rates are displayed in Table \ref{tab3}. Compared to the test \eqref{Boot0} we observe an improvement of  the approximation of the nominal level in  scenario \eqref{hd27}. While this is mostly satisfactory in the case of a normal distribution, the 
 rejection probabilities are still a little too large  for $t$-distributed data if the sample sizes are $n\geq100$
 (again the  deviations become smaller if the threshold $\Delta =0.05$ is used in the hypotheses \eqref{hd1b}). However,  
 the test \eqref{boottestnv} keeps 
the nominal level well
for the two other 
models \eqref{hd28a} and \eqref{hd28} and  all combination of $n$ and $p$.

\begin{table}[t]
	\centering 
	\begin{tabular}{|c|c|c|c|c|c|c|c|c|c|}
	    \hline
	$n,p $ & $50,100$ & $50,200$ & $50,400$ & $100,100$ & $100,200$& $100,400$ & $200,100$ & $200,200$ & $200,400$ \\
		\hline
	 \ref{hd27}  & 0.062 & 0.042 & 0.050 & 0.130  & 0.113  & 0.098 & 0.136 & 0.132 & 0.150  \\ 
	 		 \ref{hd28a}  & 0.016 & 0.007 & 0.009 & 0.074  & 0.045 & 0.034 & 0.108 & 0.128 & 0.116 \\ 
		 \ref{hd28}  & 0.000& 0.000 & 0.000& 0.000 & 0.000 & 0.000  & 0.000 & 0.000& 0.000 \\
		 \hline
		 \ref{hd27}  & 0.102 & 0.084 &0.075  & 0.154 & 0.180 & 0.149  & 0.171 & 0.175 & 0.162\\
		 		 \ref{hd28a}  & 0.026 & 0.018 & 0.008 & 0.061 &  0.059 &  0.060 & 0.122 & 0.136 & 0.119      \\
		 \ref{hd28}  & 0.000&0.000 &0.000 & 0.000 & 0.000 & 0.000  & 0.000 & 0.000& 0.000 \\
		 \hline
	\end{tabular}
	\smallskip
	\caption{ \it Simulated rejection probabilities of the test \eqref{boottestnv} under the null hypothesis
	in \eqref{kendal0} (nominal level $\alpha=0.1$). Upper part: multivariate normal distribution;
	lower part: multivariate $t$-distribution with $3$ degrees of freedom.}
	\label{tab3}
	\end{table}

The power curves of the tests \eqref{boottest} and \eqref{boottestnv} are displayed in
 Figures \ref{fig:1} and \ref{fig:5}, respectively, where  we show the rejection probabilities of the test \eqref{boottest} as a function of Kendall's $\tau=\frac{2}{\pi}\arcsin{\rho}$
 for sample
size and dimension  given by $(n,p) =(50,100)$ and 
$(n,p) =(100,100)$.
\begin{figure}[H]
    \centering
    \includegraphics[width=7cm, height=4.5cm]{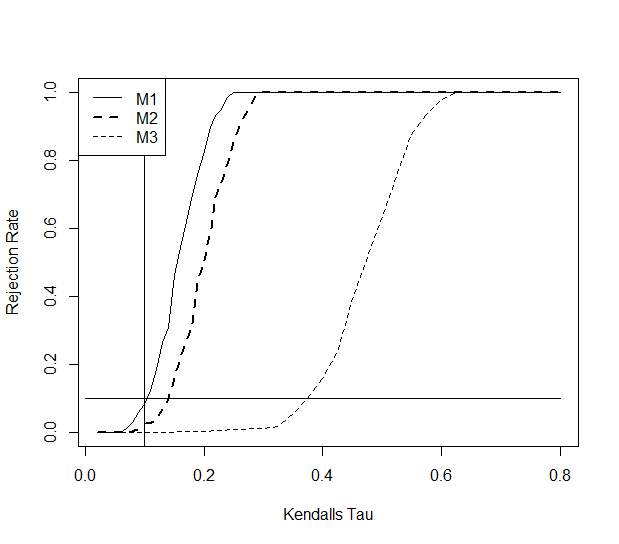}    \includegraphics[width=7cm, height=4.5cm]{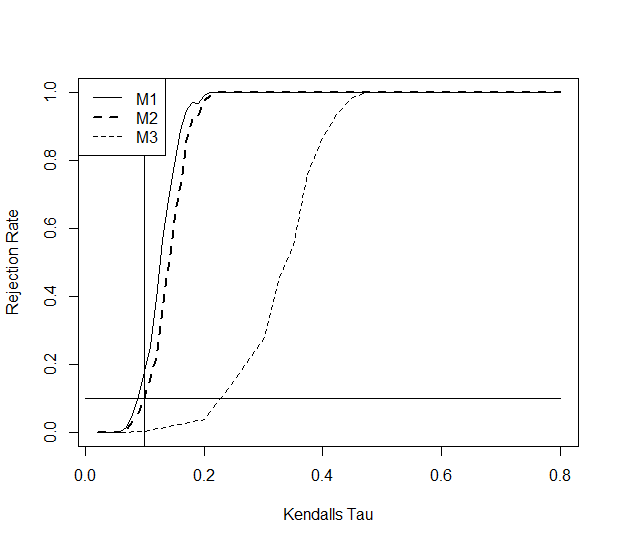}
     \includegraphics[width=7cm, height=4.5cm]{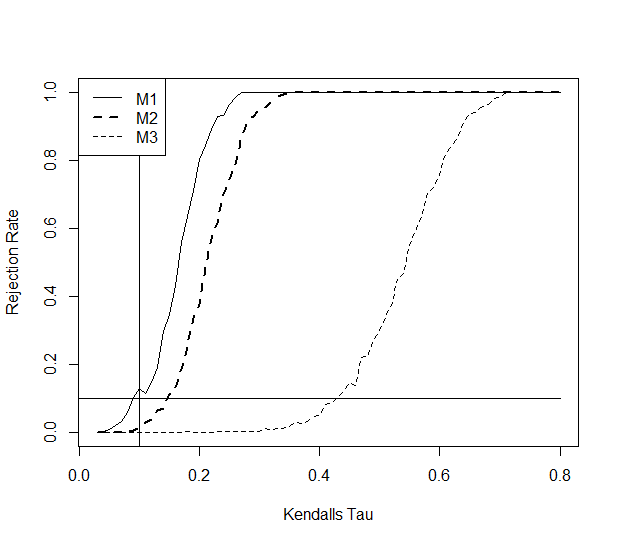}    \includegraphics[width=7cm, height=4.5cm]{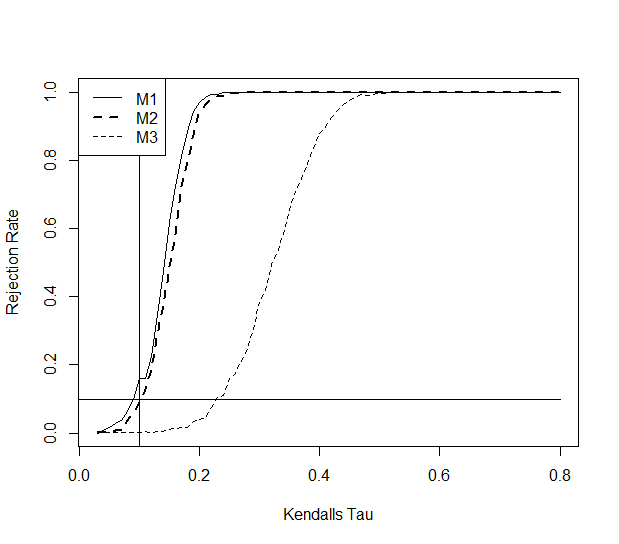}

\vspace{-.5cm}
    \caption{ \it Simulated rejection probabilities of the test  \eqref{boottest}
for the hypotheses \eqref{hd1b} with $\Delta=0.1$. The dimension is $p=100$, and the sample sizes are  $n=50$ (left panels) and $n=100$ (right panels). Upper part:  normal distributed data; Lower part:  $t_3$-distributed data.  }

    \label{fig:1}
    
\end{figure}

\begin{figure}[H]
    \centering
    \includegraphics[width=7cm, height=4.5cm]{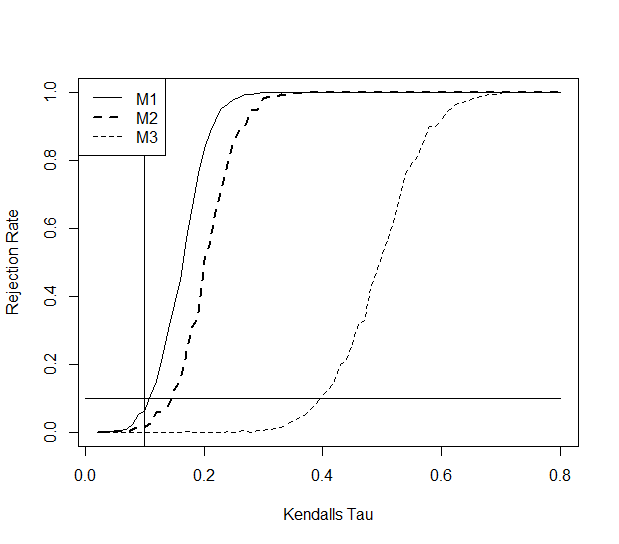}
    \includegraphics[width=7cm, height=4.5cm]{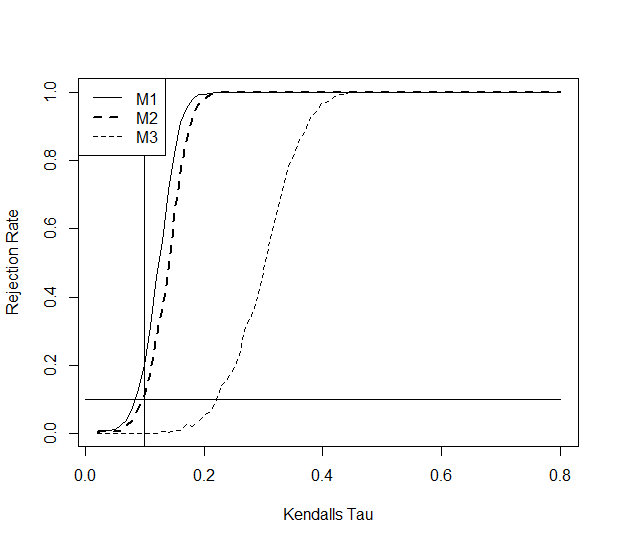}
    \includegraphics[width=7cm, height=4.5cm]{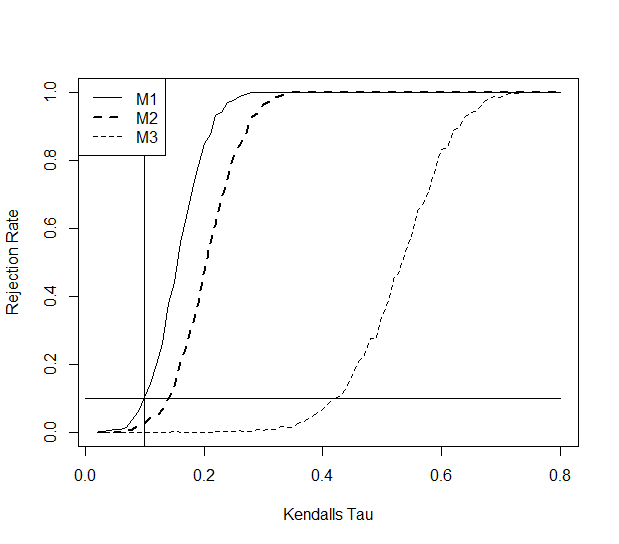}
    \includegraphics[width=7cm, height=4.5cm]{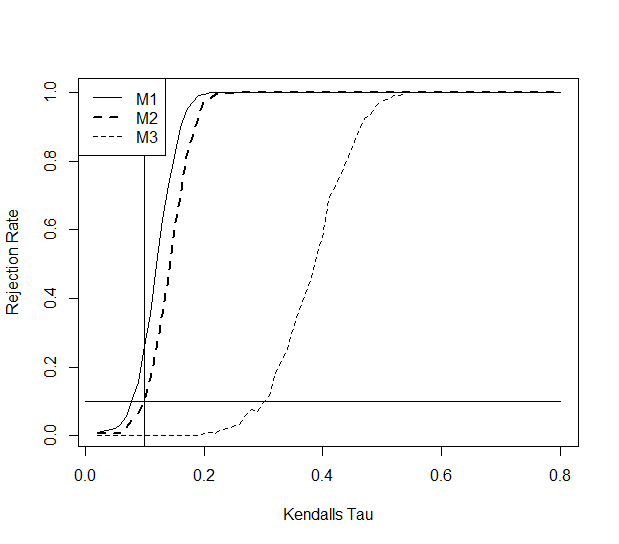}
\vspace{-.5cm} 

    \caption{ \it Simulated rejection probabilities of the test  \eqref{boottestnv}
for the hypotheses \eqref{hd1b} with $\Delta=0.1$. The dimension is $p=100$, and the sample sizes are  $n=50$ (left panels) and $n=100$ (right panels). Upper part:  normal distributed data; Lower part:  $t_3$-distributed data.  }
    \label{fig:5}
\end{figure}

The results reflect our theoretical findings.
The rejection rates increase with the distance to the null-hypothesis and the sample size for all three covariance structures. Moreover, the largest power is obtained for the covariance matrix \eqref{hd27} followed by \eqref{hd28a} and \eqref{hd28}. 
A comparison of the upper and lower parts in the figures shows 
that the tests have  lower power for $t$-distributed data.
Comparing Figures \ref{fig:1} and \ref{fig:5}  we  observe that the results of the 
tests \eqref{boottest} and \eqref{boottestnv}
under the alternative are comparable in most cases (with slight advantages of the test \eqref{boottest}).
Only  for model \eqref{hd28} with $(p,n) =(100,100)$ we observe
that the test \eqref{boottest} has a substantially larger power.

In a further simulation study, which is not presented for the sake of brevity, we have also 
investigated the performance of the asymptotic test \eqref{hd5}. This test is extremely conservative  
(under the null hypothesis the empirical rejection rate is always $0$  in all three scenarios), and it is less powerful than the  bootstrap
 tests.

\subsection{Test statistics involving $|U_{ij}|$}
\label{sec42}

In order to improve the approximation of the nominal level we investigate in this section the bootstrap test \eqref{boottestabs} that  uses the absolute value $|U_{ij}|$
instead of $U_{ij}^2$ in the definition of the test statistic (see  Remark \ref{nv}(b)).
For the sake of comparison we consider the  same scenarios as in Section \ref{sec41}
and study the properties of the test 
 for the hypotheses \eqref{hd1b} with $\Delta=0.1$.
The empirical type I error rates  are shown in Table \ref{tab2} and we observe that 
the test \eqref{boottestabs} keeps the nominal level
in all cases under consideration (in particular  also in  the ``worst case'' scenario \eqref{hd27}, where all (pairwise) Kendall's taus satisfy $\tau_{ij}=0.1$, and the data is  heavy tailed). Again we observe  in the two other scenarios \eqref{hd28a} and \eqref{hd28}  a smaller  type I error rate than  for the scenario \eqref{hd27}, which agrees with our theoretical findings in  Section \ref{sec2} and \ref{sec3}. In scenario M1 the
empirical rejection probabilities are still smaller than the  nominal level $\alpha=0.1$ predicted by the theory  (at least asymptotically). We conjecture that this is a common phenomenon of tests for hypotheses of this type using the  bootstrap  in high dimension. For example, we have conducted a small simulation study for testing the one-sided hypotheses $H_0^+$ in \eqref{rev1} using the statistic $T_n^+$  in \eqref{rev2} and the Gaussian multiplier bootstrap proposed in \cite{UApprox}. In setting M1 with a normal distribution and   $(n,p)=(50,200) $,
  $(n,p)=(50,400) $, $(n,p)=(100,200) $ and  
  $(n,p)=(100,400) $
  we obtained the empirical level  $1.5\%$, $0.4\%$, $6.1\%$ and 
  $3.2\%$, respectively. These results are very similar to the results in Table \ref{tab2} for the test test \eqref{boottestabs}.


\begin{table}[h]
	\centering 
	\begin{tabular}{|c|c|c|c|c|c|c|c|c|c|}
	    \hline
	$n,p $ & $50,100$ & $50,200$ & $50,400$ & $100,100$ & $100,200$& $100,400$  & $200,100$ & $200,200$& $200,400$\\
		\hline
	 \ref{hd27}  & 0.026 &  0.013& 0.015 & 0.048 & 0.047 & 0.029 & 0.076 & 0.067 & 0.077 \\
		 \ref{hd28a}  & 0.005& 0.008& 0.003 & 0.017 & 0.017 & 0.009  & 0.026 & 0.02 & 0.019  \\     
		 	 \ref{hd28}  & 0.000& 0.000 & 0.000& 0.000 & 0.000 & 0.000 & 0.000 & 0.000 & 0.000  \\
		 \hline
		 \hline
		 \ref{hd27}  & 0.052&  0.025& 0.014 & 0.081& 0.066 & 0.044 & 0.106 & 0.102 & 0.093 \\
		 \ref{hd28a}  & 0.014& 0.010 & 0.009 & 0.027 &  0.019 & 0.023 & 0.030 & 0.020 & 0.041        \\
		 	 \ref{hd28}  & 0.000&0.000 &0.000 & 0.000 & 0.000 & 0.000  & 0.000 & 0.000 & 0.000 \\
		 \hline
	\end{tabular}
	\smallskip
	
	\caption{ \it Simulated rejection probabilities of the test \eqref{boottestabs} under the null hypothesis
	in \eqref{kendal0} (nominal level $\alpha=0.1$). Upper part: multivariate normal distribution;
	lower part: multivariate $t$-distribution with $3$ degrees of freedom. }
	\label{tab2}
\end{table}
	
 In Figure \ref{fig:3} we display the empirical rejection probabilities as a function of Kendall's $\tau=\frac{2}{\pi}\arcsin{\rho}$ where the sample size and dimension are given by $(n,p)=(50,100)$ and $(n,p)=(100,100)$.
We consider again the covariance structures \eqref{hd27} - \eqref{hd28} and a multivariate normal  and $t_3$-distribution.

\begin{figure}[H]
    \centering
    \includegraphics[width=7cm, height=5cm]{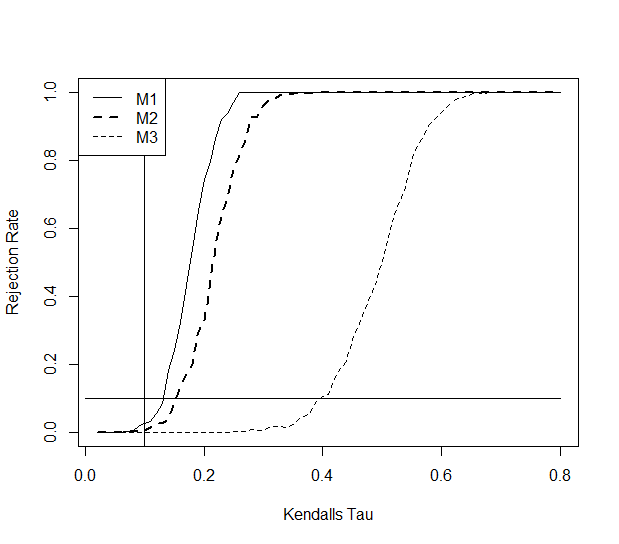}
    \includegraphics[width=7cm, height=5cm]{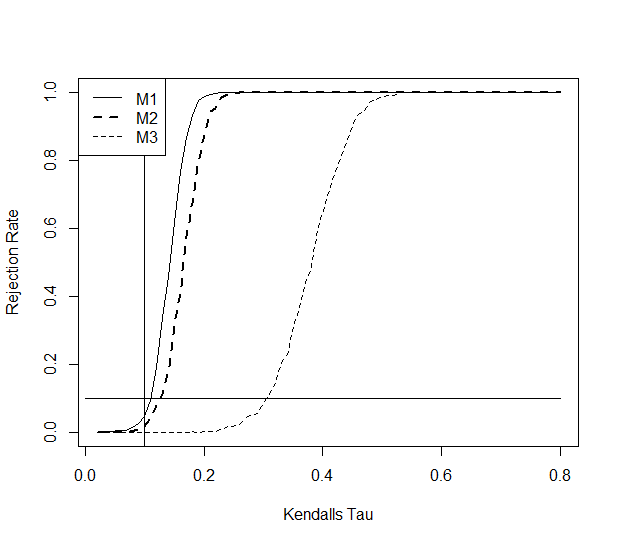}
    \includegraphics[width=7cm, height=5cm]{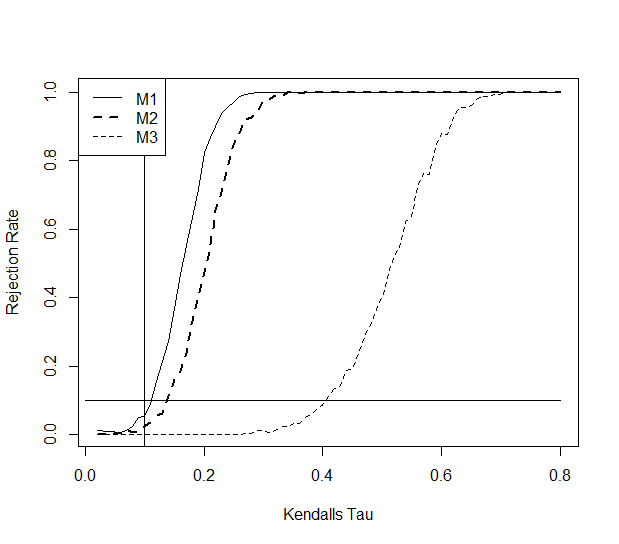}
    \includegraphics[width=7cm, height=5cm]{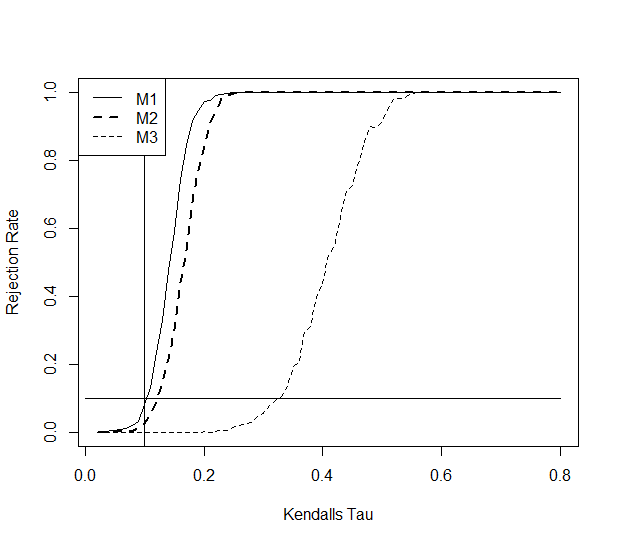}
\vspace{-.5cm} 

    \caption{ \it 
Simulated rejection probabilities of the test  \eqref{boottestabs}
for the hypotheses \eqref{hd1b} with $\Delta=0.1$. The dimension is $p=100$, and the sample sizes are  $n=50$ (left panels) and $n=100$ (right panels). Upper part:  normal distributed data; Lower part:  $t_3$-distributed data.  
}
    \label{fig:3}
\end{figure}	

Once again, the results are in line with our theoretical findings.
The rejection rates increase with the distance to the null-hypothesis and the sample size for all three covariance structures. Moreover, the largest power is obtained for the covariance matrix \eqref{hd27} followed by \eqref{hd28a} and \eqref{hd28}. Comparing the upper and the lower parts  we observe a loss in power for $t$-distributed data. It is also of interest to compare these results 
with the non-normalized test \eqref{boottestnv} in Figure \ref{fig:5}. While the differences 
are small in the case $(p,n)=(100,100)$, they are more visible 
for $(p,n)=(50,100)$. In other words:  the test 
\eqref{boottestabs} 
keeps the  nominal level in all cases under consideration, but
compared to the test \eqref{boottestnv}
this advantage comes with the price of a slight loss in power in the 
case $(p,n)=(50,100)$.

Finally,  we display in Table \ref{tab4} computation times for one run  of the test \eqref{boottestnv} with $100$ bootstrap replications.
 The other tests have similar computation times with exception of the test mentioned in Remark \ref{rem11}, which is substantially faster by construction. All computations were performed on single core of a 2.9 GHz Intel core i5 processor. 
\begin{table}[h]
	\centering 
	\begin{tabular}{|c|c|c|c|c|c|c|c|c|c|}
	    \hline
	$p,n $ & $50,50$ & $50,100$ & $50,200$ & $100,50$ & $100,100$& $100,200$  & $200,50$ & $200,100$& $200,200$\\
		\hline    
	  $t$   & 8.73 & 9.72 & 11.712 &  34.90  & 40.72 & 48.37 & 146.73   & 159.95  & 194.09   \\  
	  \hline        
		
	\end{tabular}
	\smallskip
	\caption{ \it Empirical computation time (in seconds) for  one run of the  bootsrap test \eqref{boottestb} with  $100$ bootstrap replications.}
	\label{tab4}
\end{table}

\subsection{Reversed Hypotheses}

In this section we follow up on Remark \ref{rem11} and provide 
a short study of the empirical power of the test \eqref{boottestb} for the hypotheses
\eqref{hdx1}.  As pointed out by a referee, testing these hypotheses with reasonable power requires a large sample size  as one needs to control the probability of deciding that all $p(p-1)/2 $ dependence measues  are smaller than the threshold, when in fact there is at least one dependence measure  larger or equal than $\Delta$.  This intuitive argument is confirmed by the results
 in Table \ref{pbt2}, which shows the 
 rejection probabilities of the test
 \eqref{boottestb}
for a $p$-dimensional  standard normal distribution (independence scenario), where the threshold is chosen  as $\Delta \in \{0.2,0.3\}$. 
We observe that the test only has (large) power, if the sample size is sufficiently large.
\begin{table}[h]
	\centering 
	\begin{tabular}{|c|c|c|c|c|c|c|c|c|c|}
	    \hline
	$n,p $ & $100,100$ & $100,200$ & $100,400$ & $200,100$ & $200,200$& $200,400$  & $400,100$ & $400,200$& $400,400$\\
		\hline    
	  $\Delta=0.2$  & 0.000 & 0.000 & 0.000 & 0.000 & 0.000 & 0.000& 1.000 & 0.994 & 0.882    \\      
	  \hline
	$\Delta=0.3$	& 0.016 & 0.000 & 0.000 & 0.998 & 0.996 & 0.962 & 1.000 & 1.000 & 1.000 \\
		 \hline
	\end{tabular}
	\smallskip
	\caption{ \it Simulated rejection probabilities of the test \eqref{boottestb} for the hypotheses \eqref{hdx1}  with thresholds $\Delta=0.2$, $0.3$.
	\label{pbt2}}
\end{table}


\vspace*{-0.5cm}
\subsection{Real Data Application Example}
\label{sec46}
In this section we will employ the proposed methods to re-analyze the  prostate cancer data set from protein mass spectroscopy discussed in  \cite{Adam}. The data set consists of samples from a healthy group and a group with prostate cancer, each consisting of a blood serum sample $i$ for which the intensities $X_{i,j}$ of a large number of time-of-flight values $t_j$ is recorded. The time-of-flight is related to the mass over charge ratio m/z of the constituent proteins in the blood. 
\\
The data set consists of $157$ healthy and $167$ prostate cancer patients with a total of 48538 m/z sites. Following previous researchers \citep[see, fore example,][]{Tibshirani} we ignore the m/z sites below $2000$ due to possible chemical artifacts. We also average the intensity levels in consecutive blocks of $20$, and following \cite{Levina} we further average again, this time in consecutive blocks of $10$, obtaining observations with a dimension of $218$. Similar to the cited studies we apply both a test for relevant dependence and a test for relevant bandedness to the data, conducting the tests for each group separately.
Here relevant bandedness refers to the hypotheses
\begin{align}
    H_0:\max_{|i-j|\geq m}|\rho_{i,j}|\leq \Delta
\end{align} for apropriately chosen $m$.
For our study we choose the test \eqref{boottestabs} with Spearman's $\rho$ to measure the dependence between the different intensities. We test at level $\alpha=0.05$ and choose $\Delta=0.1$, i.e. according to the thresholds from \cite{cohen1988} we are testing whether the correlations are (uniformly) small/negligible. We sample $10.000$ times to obtain the bootstrap quantiles. For $m=0$ we reject $H_0$ for both groups, while for $m=125$ we reject the null hypotheses only for the prostate cancer group, but not for the healthy patient group (here the $p$-values are $0$ and $0.193$ respectively.).
\\
It is of interest to compare our approach with the standard one of testing classical hypotheses (i.e. $\Delta=0$). For this purpose  we consider the tests for independence and bandedness proposed in  \cite{hanetal2017} and \cite{yaoetal2017}. The test from \cite{hanetal2017} is also based on Spearman's $\rho$ while the  test from \cite{yaoetal2017}  uses the  the distance covariance as dependence measure. Both tests reject independence and also bandedness  for both patient groups. The statement of rejection is put into perspective in \cite{yaoetal2017} by the authors' remark  that the dependencies for the healthy group seems  weaker for bandwidths  between $125$ and $150$. Testing the relevant hypotheses in  \eqref{hd1b} with a threshold, which defines negligibility of the dependencies provides an interesting alternative to this discussion.

\bigskip

\textbf{ Acknowledgements.} 
This work  was partially supported by the  
 DFG Research unit 5381 {\it Mathematical Statistics in the Information Age}, project number 460867398.  
 The authors would like to thank two referees and the associate editor for  constructive comments, which led to a substantial improvement of  an earlier version of this paper.

\newpage
\appendix

		\section{Online supplement:  proofs}
\label{sec5} 
  \def\theequation{A.\arabic{equation}}	
	\setcounter{equation}{0}	
	
In this section we provide proofs of our theoretical results. These are rather involved and we proceed in several steps.
In Section \ref{sec51}, we begin with the analysis of the variance estimators 
$\hat \sigma_i^2$ defined in \eqref{varDef}.  These results are used in the proofs of Theorems \ref{alpha} and \ref{consistency}, which are provided in  Section \ref{sec52}. The proof of the consistency of the bootstrap test can be found in Section \ref{sec53}.  Several arguments 
given in this section 
rely on  sophisticated technical results, which will be provided in Section \ref{sec6}. \smallskip

{\bf Notation:} Throughout this section we use the symbol
 $a_n \lesssim b_n$ to denote $a_n \leq C\, b_n$ for some generic positive constant $C$ not depending on $n$ whose 
 concrete value   may change from line to line.  We also introduce $1-o_K(1)$ as a shorthand for any term of the form
 \begin{align*}
    1-C_1/(nd)-C_2(\log(nd))^{1/2+1/\beta}/\sqrt{n}-C_3(\log(nd))^{\beta}n^{-\gamma/\beta}\,,
 \end{align*}
 where the non-negative constants $C_1,C_2$ and $C_3$ may only depend  on $\gamma$ and $\beta$. We remark that in many cases some of the factors in the summands will be 0. 
 
 Moreover $\norm{x}_\infty=  \max_{i=1}^d|x_i|$ denotes the maximum norm of a $d$-dimensional vector, where the dimension of $x$ will always be clear from the context. We also note that many bounds could be stated with $\log(d)$ in place of $\log(nd)$ at the cost of slight changes to terms involving $1-o_K(1)$. The only places where we pay close attention to the difference between the two is when we inspect the consistency properties of the two tests. Also note that we write $\E$ instead of $\E_F$ for the sake of notational convenience.

\subsection{Variance Estimation} 
\label{sec51}

From \eqref{varDef}, recall the definition of the variance estimator  $\hat{\sigma}^2_i$. The  following theorem 
characterizes the uniform  convergence rate of the differences  $ \{ n \hat{\sigma}^2_i -m^2\zeta_{1,i}  ~|~i=1, \ldots , d \} $ with $\zeta_{1,i}$ defined in \eqref{x1b}. 

\begin{theorem}
\label{VarConv}
If Assumption $(A1)$ is satisfied and   $\log d=o(n^\gamma)$ for $\gamma\leq\frac{1}{4/\beta+1}$, we have
\begin{align*}
\max_{1 \leq i \leq d} |n\hat{\sigma}^2_i-m^2\zeta_{1,i}| \lesssim B_n^2\sqrt{\frac{\log(nd)}{n}}
\end{align*}
 with probability at least $1-o_K(1)$, where the hidden constant in the inequality depends only  on $\beta$. 
\end{theorem}

\begin{proof}
	
	We will use similar arguments as given in the proof of Lemma A.1  in \citeSM{Zhou2019}. 
	Some difficulties arise as in contrast to this work  we consider 
	$U$-statistics with unbounded kernels.
	First, we define a centralized version of the $U$-statistics in  \eqref{UDef} and the leave one out estimator below \eqref{varDef},
	\begin{align*}
	\bar{U}_i:=U_i-\theta_i\quad \text{ and } \quad &\bar{q}_{k,i}:={n-1 \choose m-1}^{-1} \!\!\! \sum_{1 \leq l_1< \ldots <l_{m-1}\leq n, l_j \neq k}g_i(X_k,X_{l_1}, \ldots ,X_{l_{m-1}})\,,
	\end{align*}
	respectively, where $g_i(X_{l_1},\ldots ,X_{l_m})=h_i(X_{l_1}, \ldots ,X_{l_m})-\theta_i$ and $1\le i\le d$.	A simple calculation shows that 
	\begin{align*}
	\hat{\sigma}^2_{i}=\frac{m^2(n-1)}{n(n-m)^2}\sum_{k=1}^{n}(q_{k,i}-U_i)^2=\frac{m^2(n-1)}{n(n-m)^2}\sum_{k=1}^{n}(\bar{q}_{k,i}-\bar{U}_i)^2.
	\end{align*}
	Setting $g_{1,i}(x) =\E[g_i(X_1, \ldots ,X_m)|X_1=x]$, $\bar{g}_{1,i}=\frac{1}{n}\sum_{j=1}^ng_{1,i}(X_j)$  and using  the triangle inequality
	then yields 
	\begin{align*}
	|n\hat{\sigma}^2_i-m^2\zeta_{1,i}|&\leq \left|\frac{m^2(n-1)}{(n-m)^2}\sum_{k=1}^{n}\Big[(\bar{q}_{k,i}-\bar{U}_i)^2-(g_{1,i}(X_k)-\bar{g}_{1,i})^2\Big]\right|\\
	&+\left|\frac{m^2(n-1)}{(n-m)^2}\sum_{k=1}^{n}(g_{1,i}(X_k)-\bar{g}_{1,i})^2-m^2\zeta_{1,i}\right|=: M_i^{(1)}+M_i^{(2)}\,
	\end{align*} 
	for $1\le i\le d$. Therefore, the claim of Theorem \ref{VarConv} is a consequence of the following two Lemmas \ref{lem:4.1} and \ref{lem:4.2}.
\end{proof}

	\begin{Lemma}\label{lem:4.1}
	Under the conditions of Theorem \ref{VarConv} we have with probability at least $1-o_K(1)$ that
	\begin{align*}
	\max_{1 \leq i \leq d} M_i^{(1)}\lesssim B_n^2\sqrt{\frac{\log (nd)}{n}}\,.
	\end{align*}	

	\end{Lemma}
	\begin{proof}
	Recalling that 	$\sum_{k=1}^{n}(g_{1,i}(X_k)-\bar{g}_{1,i})^2=\sum_{k=1}^{n}g_{1,i}^2(X_k)-n\bar{g}_{1,i}^2$ and using the triangle inequality yields
	\begin{align}
	\label{x1a}
	M_i^{(1)}\lesssim \left|\bar{U}^2_i-\bar{g}^2_{1,i}\right|+\left|\frac{m^2(n-1)}{(n-2)^2}\sum_{k=1}^{n}\Big(\bar{q}_{k,i}^2-g_{1,i}^2(X_k)\Big)\right|. 
	\end{align}	
	For the first term we use Lemma \ref{UStatConc}
	from the online supplement, as we will use it repeatedly throughout the remaining proofs we will explain its application in detail one time. We apply it separately to the $U$-Statistics $\bar{U}_i$ and $\bar{g}_{1,i}$ which fulfill the required conditions by the first equation in assumption (A1) and the assumption that $\gamma \leq \frac{1}{4/\beta+1}$, note that we will always use the version of the bound containing $\log(nd)$ except when considering consistency properties.
	\begin{align*}
	\underset{1 \leq i \leq d}{\max}|\bar{U}^2_i-\bar{g}^2_{1,i}| =	\underset{1 \leq i \leq d}{\max} |(\bar{U}_i-\bar{g}_{1,i})(\bar{U}_i+\bar{g}_{1,i})| \lesssim B_n^2 \frac{\log(nd)}{n} 
	\end{align*}

	with probability at least $1-o_K(1)$.
	For the second term in \eqref{x1a} a more sophisticated analysis is necessary which we facilitate by decomposing
	\begin{align*}
	\sum_{1 \leq l_1< \ldots <l_{m-1} \leq n, l_j \neq k}g_i(X_k,X_{l_1}, \ldots ,X_{l_{m-1}})=A_{n,m}g_{1,i}(X_k)+B_{n,m}S_i+\Gamma_{k,i}\,,
	\end{align*}
	where $A_{n,m}={n-1 \choose m-1}- {n-2 \choose m-2}$, $B_{n,m}={n-2 \choose m-2}$, $S_i=\sum_{l=1}^{n}g_{1,i}(X_l)$ and
	\begin{align*}
	\Gamma_{k,i}=\sum_{1 \leq l_1 <  \ldots  < l_{m-1}\leq n, l_j \neq k}\Big(g_i(X_k,X_{l_1} \ldots ,X_{l_{m-1}})-g_{1,i}(X_k) -\sum_{j=1}^{m-1}g_{1,i}(X_{l_j})\Big)~.
	\end{align*}
	By the definition of $\bar{q}_{k,i}$, we then have
	\begin{align*}
	\bar{q}_{k,i}=\frac{A_{n,m}g_{1,i}(X_k)+B_{n,m}S_i+\Gamma_{k,i}}{{n-1 \choose m-1}}, 
	\end{align*}
	which leaves us with the task to bound
	\begin{align*}	
	J:=	\underset{1 \leq i \leq d}{\max}\left|\frac{m^2(n-1)}{(n-2)^2}\sum_{k=1}^{n}\left[\bar{q}_{k,i}^2-g_{1,i}^2(X_k)\right]\right|~.
	\end{align*}
	Setting
	$V^2_i=\sum_{k=1}^{n}g_{1,i}^2(X_k)$, $D_{n,m}={n-1 \choose m-1}$ and  $\Lambda^2_i= \sum_{k=1}^{n}\Gamma^2_{k,i}$, we have
	\begin{align*}
	\sum_{k=1}^{n}\bar{q}_{k,i}^2 &=\frac{1}{D^2_{n,m}}
	\Big \{ A^2_{n,m}V_i^2+\Lambda_i^2+(nB^2_{n,m}+2A_{n,m}B_{n,m})S_i^2\\
	&\quad +2A_{n,m}\sum_{k=1}^{n}g_{1,i}(X_k)\Gamma_{k,i}+2B_{n,m}S_i\sum_{k=1}^{n}\Gamma_{k,i}\Big \} 
\end{align*}	
	which together with the Cauchy-Schwarz inequality (for $J_4$ and $J_5$) yields 
	\begin{align*}
	    J \lesssim J_1+J_2+J_3+J_4+J_5~,
	\end{align*} 
	where
	\begin{align*}
	&J_1=\underset{1 \leq i \leq d}{\max}\left|\frac{(A^2_{n,m}-D^2_{n,m})V_i^2}{n D^2_{n,m} }\right|\lesssim\underset{1 \leq i \leq d}{\max} \frac{V_i^2}{n^2}\,,\\
	&J_2=\underset{1 \leq i \leq d}{\max}\frac{\Lambda_i^2}{nD^2_{n,m}}\,,\\
	&J_3=\underset{1 \leq i \leq d}{\max}\frac{(nB^2_{n,m}+2A_{n,m}B_{n,m})S_i^2}{nD^2_{n,m}} \lesssim \underset{1 \leq i \leq d}{\max}\frac{S_i^2}{n^2}\,,\\
	&J_4=\underset{1 \leq i \leq d}{\max}\frac{2A_{n,m}V_i\Lambda_i}{nD^2_{n,m}} \lesssim \underset{1 \leq i \leq d}{\max}\frac{V_i\Lambda_i}{n^2}\,,\\
	&J_5=\underset{1 \leq i \leq d}{\max}\frac{2B_{n,m}|S_i|\sqrt{n}\Lambda_i}{nD^2_{n,m}} \lesssim \underset{1 \leq i \leq d}{\max}\frac{V_i\Lambda_i}{n^2}\,.
	\end{align*}
	In the remainder of this proof, we will bound the terms $J_1,\ldots, J_5$ separately.
	For $J_1$ we have by Lemma \ref{ProductNorm} that $\norm{g_{1,i}^2(X_k)}_{\psi_{\beta/2}}\leq B_n^2$ so that  Lemma \ref{UStatConc} in the online supplement yields that  
	\begin{align}
	\label{x2a}
	J_1 \lesssim \underset{1 \leq i \leq d}{\max}\frac{V_i^2}{n^2}=\underset{1 \leq i \leq d}{\max}\frac{\sum_{k=1}^{n}g_{1,i}^2(X_k)-\zeta_{1,i}}{n^2}+\frac{\zeta_{1,i}}{n^2} \lesssim B_n^2\sqrt{\frac{\log(nd)}{n^3}}
	\end{align}
	with probability at least $1-3/(nd)-C(\log (nd))^{1/2+1/\beta}/\sqrt{n}$. \\
	Regarding the term $J_2$, we define the set 
\begin{align*}
A_n & =\Big \{\underset{\substack{1 \leq i \leq d \\ 1 \leq k, l_1, 
	\ldots ,l_{m-1} \leq n}}{\max}\Big(g_i(X_k,X_{l_1} \ldots ,X_{l_{m-1}})-g_{1,i}(X_k) \\
	& ~~~~~~~~~~~~~~~~~~~~~~~~~~~
	-\sum_{j=1}^{m-1}g_{1,i}(X_{l_j})\Big)\leq C_\beta B_n (\log(nd))^{1/\beta}\Big \}\,,
\end{align*}
	where the constant $C_\beta$ is chosen such that $\p(A_n)\geq 1-\frac{1}{nd}$.
	Indeed, using the union bound and Lemma \ref{WeibullChar}  in the online supplement it is easy to see that by choosing $C_\beta$ appropriately (this can be done universally with only dependence on $\beta$) we obtain $\p(A_n)\geq 1-\frac{1}{nd}$.
	
	Conditional on $X_k$ we now apply Lemma \ref{UStatConc} in the online supplement to $\Gamma_{k,i}$ on the set $A_n$ with $K=B_n(\log(nd))^{1/\beta}$ to obtain 
	\begin{align*}
&	\P \Big( \Big\{\underset{1 \leq i \leq d, 1 \leq k \leq n}{\max}\Gamma_{k,i}/D_{n,m}  \lesssim B_n (\log(nd))^{1/\beta}\sqrt{\tfrac{\log(nd)}{n}}\Big\} \cap A_n \,\Big) \\
\nonumber & \qquad
 =\sum_{k=1}^n \P \Big( \Big\{\underset{1 \leq i \leq d}{\max}\Gamma_{k,i}/D_{n,m}  \lesssim B_n (\log(nd))^{1/\beta}\sqrt{\tfrac{\log(nd)}{n}}\Big\} \cap A_n \,\Big\vert \tilde{B}_k \Big) \frac{1}{n}    \\
\nonumber & \qquad
=\sum_{k=1}^n \E\left[\P \Big( \Big\{\underset{1 \leq i \leq d}{\max}\Gamma_{k,i}/D_{n,m}  \lesssim B_n (\log(nd))^{1/\beta}\sqrt{\tfrac{\log(nd)}{n}}\Big\} \cap A_n \,\Big|X_k, \tilde{B}_k\Big)\right] \frac{1}{n}   \\
\nonumber & \qquad
\ge   1-o_K(1)
	\end{align*}
	where we used that $\Gamma_{k,i}$ and $\Gamma_{k,j}$ have the same distribution to obtain the second line and define $\tilde{B}_l$ as the event that $\underset{1 \leq i \leq d, 1 \leq k \leq n}{\max}\Gamma_{k,i}/D_{n,m}=\underset{1 \leq i \leq d}{\max}\Gamma_{l,i}/D_{n,m}$.
	Using the definition of $\Lambda_i$ and recalling that $\gamma \leq \frac{1}{4/\beta+1}$, this yields
	\begin{align}
	\label{x2b}
	J_2=\underset{1 \leq i \leq d}{\max}\frac{\Lambda_i^2}{nD^2} \lesssim  B_n^2\frac{\log(nd)}{n}(\log(nd))^{2/\beta}\lesssim B_n^2\sqrt{\frac{\log(nd)}{n}}
	\end{align}
	with probability at least $1-\frac{4}{nd}-C(\log(d))^{1/2+1/\beta}/\sqrt{n}$.
	
	For $J_3$ we have by Lemma \ref{WeibConc} in the online supplement that
	\begin{align}
	\label{x2c}
	J_3\lesssim \underset{1 \leq i \leq d}{\max}\frac{S_i^2}{n^2}\lesssim B_n^2\frac{\log(nd)}{n}
	\end{align}
	with probability at least $1-\frac{3}{nd}$.
Finally, regarding $J_4$ and $J_5$ we have by the calculations for $J_2$
	\begin{align}
	\label{x2d}
	J_l\lesssim \underset{1 \leq i \leq d}{\max}\frac{V_i\Lambda_i}{n^2}=\underset{1 \leq i \leq d}{\max}\sqrt{\frac{V_i^2\Lambda_i^2}{n^4}} \lesssim B_n^2\sqrt{\frac{\log(nd)}{n}}\,, \qquad l=4,5,
	\end{align}
	with probability at least $1-\frac{4}{nd}-C(\log(d))^{1/2+1/\beta}/\sqrt{n}$ provided that $\gamma \leq \frac{1}{4/\beta+1}$.
	Combining \eqref{x2a}, \eqref{x2b}, \eqref{x2c} and \eqref{x2d} shows that $J\lesssim B_n^2\sqrt{(\log(nd))/n}$ with probability at least $1-o_K(1)$. 
	\end{proof}

	\begin{Lemma}\label{lem:4.2}
	Under the conditions of Theorem \ref{VarConv}, we have with probability at least $1-o_K(1)$  that
	\begin{align*}
	\max_{1\leq i \leq d}M_i^{(2)}\lesssim B^2_n \sqrt{\frac{\log(nd)}{n}}\,.
	\end{align*}
	\end{Lemma}
	\begin{proof}
Recalling that $\sum_{k=1}^{n}(g_{1,i}(X_k)-\bar{g}_{1,i})^2=\sum_{k=1}^{n}g_{1,i}(X_k)^2-n\bar{g}_{1,i}^2$ as well as $\zeta_{1,i}=\E[g_{1,i}^2(X_1)]$ yields that
	\begin{align*}
	\left|\frac{m^2(n-1)}{(n-m)^2}\sum_{k=1}^{n}(g_{1,i}(X_k)-\bar{g}_{1,i})^2-m^2\zeta_{1,i}\right|\lesssim \left|\frac{1}{n}\sum_{k=1}^{n}g_{1,i}(X_k)^2-\zeta_{1,i}\right|+\bar{g}_{1,i}^2\,.
	\end{align*}
	
	Note that $\gamma \leq \frac{1}{4/\beta-1}$. We then apply Lemma \ref{WeibConc}, Lemma \ref{ProductNorm} and Lemma \ref{WeibCond} 
	in the online supplement
	to obtain that with probability at least $1-C/(nd)$,
	\begin{align*}
	\underset{1 \leq i \leq d}{\max}\left|\frac{1}{n}\sum_{k=1}^{n}\big(g_{1,i}^2(X_k)-\zeta_{1,i}\big) \right| \lesssim B_n^2\sqrt{\frac{\log(nd)}{n}}
	\end{align*}
	and
	\begin{align*}
	\underset{1 \leq i \leq d}{\max} \bar{g}^2_{1,i} \lesssim  B_n^2\frac{\log (nd)}{n}\,.
	\end{align*}
	\end{proof}

\subsection{ Proof of the results in Section \ref{sec21}}
\label{sec52}

\subsubsection{Preliminaries} \label{sec521}
The main step in the proofs of Theorem \ref{alpha} and 
\ref{consistency} is a weak convergence result
for the statistic 
\begin{align} \label{det1}
\mathcal{T}_n:=\underset{1\leq i\leq d}{\max}\ \frac{U^2_{i}-\theta^2_{i}}{2\hat{\sigma}_{i}|\theta_{i}|}	\end{align}
in the case where $|\theta|_{\min}:=\min_{1\le i\le d} |\theta_i| >c$ for some constant $c >0 $.
To prepare its  proof we first replace the variance estimates $\hat{\sigma}_i^2$  by the population variances using Lemma \ref{RealVar}
and then apply the Gaussian approximation in Lemma \ref{GaussApproximation} to the linearized statistic $\mathcal{T}_n$  assuming that $\log d$ and the constants $B_n$ in Assumption (A1) do not grow too fast. To this end, we 
recall the notation \eqref{x1b} and define
\begin{align}
\label{x3}
	S_n=\sqrt{n}\underset{1\leq i\leq d}{\max}\frac{\frac{1}{n}\sum_{k=1}^{n}h_{1,i}(X_k)-\theta_{i}}{\sqrt{\zeta_{1,i}}}\text{sign}(\theta_i).
\end{align}

\begin{Lemma}
	\label{RealVar}
	If Assumptions (A1) and (A2) are satisfied and  $|\theta|_{\text{min}}>c$
	for some positive constant $c$, then it holds
	\begin{align*}
	\left|\underset{1\leq i\leq d}{\max}\ \frac{U^2_{i}-\theta^2_{i}}{2\hat{\sigma}_{i}|\theta_{i}|}-S_n\right| \lesssim B_n^3\frac{\log(nd)}{\sqrt{n}}+B_n\frac{\log(nd)}{n^{1/2-\gamma/\beta}}
	\end{align*}
	with probability at least $1-o_K(1)$.
	Here the constants hidden in $\lesssim$ only depend on 
	the quantities $c,\gamma,\beta,\underline{b}$, and therefore the estimate is uniform for the subsets of the classes $\mathcal{H}_0(\Delta)$ and  $\mathcal{H}_1$  defined in \eqref{det0a} and \eqref{det0}, respectively, for which $|\theta|_{\text{min}}>c$.  
\end{Lemma}
\begin{proof}
By Theorem \ref{VarConv}, we have
\begin{align*}
	\max_{1\le 1\le d}\left|\sqrt{n}\hat{\sigma}_i-m\sqrt{\zeta_{1,i}}\right|=\max_{1\le 1\le d}\frac{\left|n\hat{\sigma}_i^2-m^2\zeta_{1,i}\right|}{\sqrt{n}\hat{\sigma}_i+m\sqrt{\zeta_{1,i}}}\lesssim B_n^2\sqrt{\frac{\log(nd)}{n}}
	\end{align*}
	up to a constant depending only on $\beta$ and $\underline{b}$,
	and therefore,
	\begin{align*}
	&\left| \sqrt{n}\underset{1\leq i\leq d}{\max}\ \frac{U^2_{i}-\theta^2_{i}}{2\sqrt{n}\hat{\sigma}_{i}|\theta_{i}|}-\sqrt{n}\underset{1\leq i\leq d}{\max}\ \frac{U^2_{i}-\theta^2_{i}}{2m\sqrt{\zeta_{1,i}}|\theta_{i}|}\right|\\ 
	&\leq \underset{1 \leq i \leq d}{\max}\left|\sqrt{n}\hat{\sigma}_i-m\sqrt{\zeta_{1,i}}\right|\sqrt{n}\underset{1\leq i\leq d}{\max} \left|\frac{U^2_{i}-\theta^2_{i}}{2m\sqrt{\zeta_{1,i}}\sqrt{n}\hat{\sigma}_i|\theta_{i}|}\right|\\
	&\lesssim B_n^2\sqrt{\frac{\log(nd)}{n}}~ \underset{1\leq i\leq d}{\max} \left|\frac{U^2_{i}-\theta^2_{i}}{2m\sqrt{\zeta_{1,i}}\hat{\sigma}_i|\theta_{i}|}\right|\,.
	\end{align*}
By the same  arguments, the triangle inequality and writing 
\begin{align*}
	\hat{T}_{n,1}&=\underset{1\leq i\leq d}{\max}\left|\theta_i\frac{U_{i}-\theta_{i}}{m\sqrt{\zeta_{1,i}}\hat{\sigma}_i|\theta_{i}|}\right|\quad \text{ and } \quad
	\hat{T}_{n,2}=\underset{1\leq i\leq d}{\max}\left|\frac{(U_{i}-\theta_{i})^2}{2m\sqrt{\zeta_{1,i}}\hat{\sigma}_i|\theta_{i}|}\right|\,,
	\end{align*}
	we obtain that
	\begin{align}
		\label{UBound}
	\underset{1\leq i\leq d}{\max} \left|\frac{U^2_{i}-\theta^2_{i}}{2m\sqrt{\zeta_{1,i}}\hat{\sigma}_i|\theta_{i}|}\right| \leq \hat{T}_{n,1}+\hat{T}_{n,2} \lesssim B_n\sqrt{\log(nd)},
	\end{align}
	with probability at least $1-o_K(1)$, where the last inequality in \eqref{UBound} follows from Lemma \ref{UStatConc} in the online supplement. Therefore, the  constant 
	in this inequality only depends on the constants $\gamma, c$ and $\beta$. Combining the two estimates we conclude  
	\begin{align}
\label{approx1}
	\left|\sqrt{n}\underset{1\leq i\leq d}{\max}\ \frac{U^2_{i}-\theta^2_{i}}{2\sqrt{n}\hat{\sigma}_{i}|\theta_{i}|}-\sqrt{n}\underset{1\leq i\leq d}{\max}\ \frac{U^2_{i}-\theta^2_{i}}{2m\sqrt{\zeta_{1,i}}|\theta_{i}|}\right| \lesssim B_n^3\frac{\log(nd)}{\sqrt{n}}\,.
	\end{align}
	
We then observe that 
\begin{equation}
\begin{split}
\label{approx2}
\sqrt{n}\underset{1\leq i\leq d}{\max}\theta_i\frac{U_{i}-\theta_{i}}{m\sqrt{\zeta_{1,i}}|\theta_{i}|} & \leq \sqrt{n}\underset{1\leq i\leq d}{\max}\frac{U_{i}^2-\theta_{i}^2}{2m\sqrt{\zeta_{1,i}}|\theta_{i}|}  \\
& \leq \sqrt{n}\underset{1\leq i\leq d}{\max}\theta_i\frac{U_{i}-\theta_{i}}{m\sqrt{\zeta_{1,i}}|\theta_{i}|}+\sqrt{n}\underset{1\leq i\leq d}{\max}\frac{(U_{i}-\theta_{i})^2}{2m\sqrt{\zeta_{1,i}}|\theta_{i}|}\\
& \lesssim \sqrt{n}\underset{1\leq i\leq d}{\max}\theta_i\frac{U_{i}-\theta_{i}}{m\sqrt{\zeta_{1,i}}|\theta_{i}|}+B_n^2\frac{\log(nd)}{\sqrt{n}}\,, 
\end{split}
\end{equation}
where the last inequality follows by Lemma \ref{UStatConc} with probability at least $1-o_K(1)$ and the hidden constant depends only on $\gamma$ and $\beta$. Using the estimate \eqref{DegenConc}  (with $t=B_n(\log(nd))/(n^{1-\gamma/\beta})$) 
in the proof of Lemma \ref{UStatConc},  we get
	\begin{align}
\nonumber
		\left|\sqrt{n}\underset{1\leq i\leq d}{\max}\theta_i\frac{U_{i}-\theta_{i}}{m\sqrt{\zeta_{1,i}}|\theta_{i}|}-S_n\right|
		& \le \sqrt{n}\underset{1 \leq i \leq d}{\max}\left|\frac{\frac{1}{n}\sum_{k=1}^{n}(h_{1,i}(X_k)-\theta_i)-(U_i-\theta_i)}{\sqrt{\zeta_{1,i}}}\right| \\
			\label{Approx3}
		& \lesssim B_n\frac{\log(nd)}{n^{1/2-\gamma/\beta}}
	\end{align}
with probability at least $1-C(\log(nd))^{\beta}n^{-\gamma/\beta}$, where the constants in both inequalities 
depend only on $\beta$ and $\gamma$. Combining \eqref{approx1}, \eqref{approx2} and \eqref{Approx3} yields the desired result. 
\end{proof}

We will now provide  a Gaussian approximation
for $S_n$, which is a consequence of Lemma \ref{GaussApprox}
in the online supplement. Note that the conditions of Lemma \ref{GaussApprox} 
are satisfied because of  
Assumption (A1), (A2) and Jensen's inequality.

\begin{Lemma}
\label{GaussApproximation}
	Under the assumptions of Theorem \ref{alpha} and $|\theta|_{\min}>c$ we have, up to some constant $C$ depending only on $\beta, D, \underline{b},c$, that
	\begin{align*}
		\underset{x \in \R}{\sup}\left| \p(S_n\leq x)-\p(S^G_n\leq x)\right|\leq C\left(\frac{B_n^2(\log(nd))^{4+2/\beta}}{n}\right)^{1/4}\,,	
	\end{align*}
where $S_n^G$ is defined as in \eqref{x3} with the difference that the vectors $X_1,\ldots ,X_k$ are replaced by independent centered Gaussian vectors with covariance matrix $\Gamma =(\Gamma_{ij})_{1 \leq i,j \leq d}$  defined by 
\begin{align}
\label{GammaDefin}
\Gamma_{ij}= {\rm Corr} (h_{1,i}(X_1),h_{1,j}(X_1))\text{sign}(\theta_i \theta_j)~.
\end{align} 

Note that the sole dependence on $\beta, D, \underline{b},c$ of the bound implies that it is valid uniformly for the subsets of the classes $\mathcal{H}_0, \mathcal{H}_1$ in \eqref{det0a} and \eqref{det0} for which $|\theta|_{\min}>c$.
\end{Lemma}
By the Schur product theorem, $\Gamma$ is a positive semidefinite matrix as it is the Hadamard product of a correlation matrix and the rank one matrix $(\text{sign}(\theta_i \theta_j))_{1\le i,j\le d}$ which are both positive semidefinite.


\subsubsection{Proof of Theorem \ref{alpha}.}

We recall $|\theta_i|\le \Delta$ for all $1\le i\le d$.  Fix some $0<c_0<\Delta$ and consider the following decomposition of $\{1,\ldots, d\}$:
\begin{align*}
	&I_1=\{1 \leq i \leq d \ :\quad  |\theta_i|> c_0\}\,,\\
	&I_2=\{1 \leq i \leq d \ :\quad  0<|\theta_i|\leq c_0\}\,,\\
	&I_3=\{1 \leq i \leq d \ :\quad  |\theta_i|=0\}\,.
\end{align*}

Using the definition of $\mathcal{T}_{n,\Delta}$ in \eqref{TestStatDefin}, we note that 
\begin{align*}
	\mathcal{T}_{n,\Delta}=\max\Big \{ \underset{i\in I_1}{\max}\frac{U_i^2-\Delta^2}{2\hat{\sigma}_i\Delta},\underset{i\in I_2}{\max}\frac{U_i^2-\Delta^2}{2\hat{\sigma}_i\Delta},\underset{i\in I_3}{\max}\frac{U_i^2-\Delta^2}{2\hat{\sigma}_i\Delta}	\Big \} \,.
\end{align*}

First, we show that the second and third terms in the right-hand maximum are negligible for our purposes. 
For the third term we use Lemma \ref{UStatConc} and $\theta_i=0$ for all $i \in I_3$ to obtain 
\begin{align}
	\label{RhoBounded}
	&\underset{i\in I_3}{\max}\frac{U_i^2-\Delta^2}{2\hat{\sigma}_i\Delta}
	\lesssim  \max_{i\in I_3}\frac{ B^2_n\log(nd)/\sqrt{n}-\sqrt{n}\Delta^2}{2\hat{\sigma}_i\sqrt{n}} 
\end{align}
with probability at least $1-o_K(1)$. Due to the assumption $\zeta_{1,i}\leq D$ and the fact that $B^2_n\log(nd)/\sqrt{n}\ll \sqrt{n}\Delta^2$ this diverges to $-\infty$ with rate at least $\sqrt{n}$. Note that all constants in these inequalities depend only on $\gamma, \beta$ and  $D$.

For the second term we use that 
\begin{align}
	\underset{i\in I_2}{\max}\frac{U_i^2-\Delta^2}{2\hat{\sigma}_i\Delta} &= \underset{i\in I_2}{\max}\frac{U_i^2-\theta_i^2 +(\theta_i^2-\Delta^2)}{2\hat{\sigma}_i\Delta} \notag\\
	& \lesssim \max_{i\in I_2}\frac{ B_n\log(d)/\sqrt{n}+\sqrt{n}(\theta_i^2-\Delta^2)}{2\hat{\sigma}_i\sqrt{n}} \,, \label{hd1}
\end{align}
where the last inequality holds with probability at least $1-o_K(1)$ by  the same calculation as in \eqref{RhoBounded} and the decomposition of $U_i^2-\theta_i^2$ into a linear and quadratic part as in \eqref{approx2}.
For the same reasons as in \eqref{RhoBounded} we conclude that the right-hand side of \eqref{hd1} converges to $-\infty$ with rate at least $\sqrt{n}$ (note that  $|\theta_i|\leq c_0 < \Delta$ for all $i \in I_2$).  We again stress the fact that all  constants
in these inequalities depend only on $\gamma, \beta, D$ and $c_0$. 

We hence obtain	uniformly for all distributions in $\mathcal{H}_0(\Delta)$
\begin{align*}
	\p\left(a_d\left(\mathcal{T}_{n,\Delta}-b_d\right)>q_{1-\alpha} \right) = \p\left(a_d \left(\sqrt{n}\,\underset{i\in I_1}{\max}\frac{U_i^2-\Delta^2}{2\hat{\sigma}_i\Delta}-b_d\right)>q_{1-\alpha} \right) +o(1)~,
\end{align*}

and it remains to show that the probability  on the right hand side is asymptotically bounded by $\alpha$
uniformly  in $\mathcal{H}_0(\Delta)$. As $|\theta_i|\leq \Delta$ we obtain the bound

\begin{align}
\label{pb100}
	\mathcal{T}_n(I_1):=\underset{i\in I_1}{\max}\frac{U_i^2-\Delta^2}{2\hat{\sigma}_i\Delta} \leq 
	\underset{i\in I_1}{\max}\frac{U_i^2-\theta_i^2}{2\hat{\sigma}_i\Delta}~.
\end{align}

Let  $C_1$ denote the constants hidden in $\lesssim$ in Lemma~\ref{RealVar} and let $C_2$ denote the hidden constants in Lemma ~\ref{GaussApproximation} (these constants depend only on $\gamma, \beta, c_0, \underline{b}, D$). Defining  \begin{align*}
    c_{\gamma,\beta}^{(1)}&:=C_1\frac{B_n^3(\log(nd))}{\sqrt{n}}+\frac{B_n(\log(nd))}{n^{1/2-\gamma/\beta}}, \\
    c_{\gamma,\beta}^{(2)}&:=C_2\left(\frac{B_n^2(\log(nd))^{4+2/\beta}}{n}\right)^{1/4},\\
    c_{\gamma,\beta}^{(3)}&:= c_{\gamma,\beta}^{(1)}\sqrt{\log(nd)}+c_{\gamma,\beta}^{(2)} 
\end{align*} 
and using  Lemma \ref{RealVar}, Lemma \ref{GaussApproximation} and Nazarovs Inequality  \citepSM[see][]{Nazarov}  we  obtain that
\begin{align*}
 \p\big ( a_d\big ( \mathcal{T}_n(I_1) 
 -b_d\big)
 >q_{1-\alpha}
 \big )
 &\leq    \p \big( \max_{i \in I_1}(S_n)_i>q_{1-\alpha}/a_d-c_{\gamma,\beta}^{(1)}+b_d \big )\\
 & \leq \p \big(\max_{i \in I_1} (S_n^G)_i> q_{1-\alpha}/a_d-c_{\gamma,\beta}^{(1)}+b_d \big )+c_{\gamma,\beta}^{(2)}\\
 & \leq  \p
 \big( \max_{i \in I_1} (S_n^G)_i > q_{1-\alpha}/a_d+b_d \big )+ c_{\gamma,\beta}^{(3)}\\
 & \leq \p\big( \max_{i \in I_1} Z_i > q_{1-\alpha}/a_d+b_d \big )+ c_{\gamma,\beta}^{(3)}+\gamma_n\\
 & \leq \p\big( \max_{1 \leq i \leq d} Z_i > q_{1-\alpha}/a_d+b_d \big )+ c_{\gamma,\beta}^{(3)}+\gamma_n
\end{align*}
where  $Z$ is a $d$-dimensional random vector with independent standard normal components, 
$S_n^G$ is defined in Lemma \ref{GaussApproximation},
 $\gamma_n\to 0$ is the sequence in Assumption (A3) and the second to last line is obtained by the normal comparison Lemma from \citeSM{Lindgren} (Theorem 4.2.1) . 
Note that these estimates are  uniform with respect 
the distribution in ${\cal H}_0 (\Delta)$
(as the  constants $C_1,C_2$  depend only on $\gamma, \beta, c, \underline{b}, D, \gamma_n$) and that last probability does not depend on ${\cal H}_0 (\Delta)$. Therefore, taking the $\limsup$ and the supremum with respect to
${\cal H}_0 (\Delta)$ yields
\begin{align*}
\limsup_{n \to \infty }  \sup_{F \in {\cal H}_0
	(\Delta)} 
	\p\left(a_d\left(\mathcal{T}_{n,\Delta}-b_d\right)>q_{1-\alpha} \right) &= 
	\limsup_{n \to \infty } \sup_{F \in {\cal H}_0
	(\Delta)} 
	\p
	\big (a_d\big (\mathcal{T}_n(I_1) 
	 -b_d\big)
	>q_{1-\alpha}\big ) \\
	& 
	\leq 	\lim_{n \to \infty }  \p\big( \norm{Z}_\infty > q_{1-\alpha}/a_d+b_d \big ) = \alpha~,
\end{align*}
which proves the first assertion of Theorem \ref{alpha}. For the second assertion we note that equality is achieved when $|\theta_i|=\Delta$ for all $i$, because in that case the inequality \eqref{pb100} becomes an equality and all of the following arguments  can  be reversed by the same concentration and anti-concentration bounds used to establish them. Again using \eqref{pb100} it is easy to see that $\sup_{d \in \mathbb{N}}\max_{i=1}^d|\theta_i|<\Delta$ implies that the test has an asymptotic rejection rate of 0, uniformly in $\mathcal{H}_0(\Delta)$.


\subsubsection{Proof of Theorem \ref{consistency}}
Let $i_0$ be an index such that $|\theta_{i_0}|=\max_{1\le i\le d} |\theta_i|>\Delta$; note that $i_0$ can depend on $n$, which is not reflected by our notation. Then we have
\begin{align*}
\mathcal{T}_{n,\Delta} \geq \frac{U_{i_0}^2-\theta_{i_0}^2}{2\hat{\sigma}_{i_0}\Delta}+\frac{\theta_{i_0}^2-\Delta^2}{2\hat{\sigma}_{i_0}\Delta} \,.
\end{align*}

By the same arguments as for \eqref{UBound} we obtain that with probability $1-o_K(1)$ 
\begin{align*}
\left|\frac{U_{i_0}^2-\theta_{i_0}^2}{2\hat{\sigma}_{i_0}\Delta}\right| \lesssim B_n\sqrt{\log(d)}\frac{1}{\sqrt{n}\hat{\sigma}_{i_0}}
\end{align*} while the second term converges to $\infty$ at rate $\frac{\sqrt{n}\xi_n}{\sqrt{n}\hat{\sigma}_{i_0}}$ with $\xi_n=\theta_{i_0}^2-\Delta^2$ by the same arguments as in \eqref{RhoBounded}. These bounds depend only on the constants $\gamma, \beta$ and $B_n$ in the Assumption (A1) and therefore hold uniformly  over the class $\mathcal{H}_1$ defined in \eqref{det0}.  This yields the desired conclusion whenever $\xi_n\geq CB_n\sqrt{\frac{\log(d)}{n}}$ for some large enough constant C as $(\frac{q_{1-\alpha}}{a_d}+b_d)\Delta \simeq \sqrt{\log d}$ and $\sqrt{n}\hat{\sigma}_{i_0}\lesssim B_n$ by Lemma~\ref{VarConv}. Here $a \simeq b$ denotes $c_1 a \leq b \leq c_2 a$ for some constants $c_1,c_2$ that do not depend on $n$.

\subsection{ Proof of the results in Section \ref{sec22}}
\label{sec53}

Let $\xi_k=(\xi_{k1},\dots,\xi_{kn})^\top, 1\le k\le n,$ be independent identically multinomial $\mathcal{M}  (1;\frac{1}{n},\dots ,\frac{1}{n})$ distributed random vectors independent of  $X_1,\dots,X_n$, that is $\p(\xi_{k1}=y_1,\dots,\xi_{kn}=y_n)=1/n$ for $(y_1,\ldots, y_n)\in \{0,1\}^n$ such that $y_1+\cdots+y_n=1$.
Then a sample $X_1^*,\dots, X_n^*$ drawn with replacement from $X_1,\dots,X_n$ can be represented as
$$
X^*_k = X \xi_k={\sum_{j=1}^{n}}  \xi_{kj} X_j \,,
$$
where $X = (X_{1},\dots,X_n) \in \mathbb{R} ^{p \times n}$.
We denote by $\p^*$ and $\E^*$ the probabilities and expectations conditional on $X_1, \ldots ,X_n$.
We also  recall the definition of the statistic $U_i^*$ in \eqref{hd8} and note that $\E^*[U_i^*]=V_i$ (see \eqref{hd11}), so that conditional on $X_1, \ldots ,X_n$ the quantity $U^*-V=(U_1^*-V_1, \ldots ,U_d^*-V_d)^{\top}$ is a $U$-Statistic of the random variables $\xi_1, \ldots ,\xi_n$. 
We start with several auxiliary results, which are required for the proof of 
Theorem~\ref{Boot0} in Section~\ref{sec431}.

\subsubsection{Some preparations}
\label{sec431}

We first observe that the conditional mean of the Bootstrap statistic is close to the mean of the original statistic, this will be used multiple times in some of the following approximations when terms involving $\norm{V-V_\Delta}_\infty$ appear, where we used the definition $V_\Delta=(V_{1,\Delta},...,V_{d,\Delta})$.
\begin{Lemma}
	\label{BootstrapExpectationConv}
Under the assumptions of Theorem \ref{Boot0} we have that
\begin{align}
\label{x10}
\norm{V-\theta}_\infty \lesssim B_n \sqrt{\frac{\log(nd)}{n}}\,,
\end{align}	
with $\p$-probability at least $1-o_K(1)$ where all constants involved depend only on $\gamma$ and $\beta$ which implies that the bound holds uniformly for the classes $\mathcal{H}_0(\Delta)$ and ${\cal H} _{0,boot}(\Delta)$ defined in \eqref{det0a} and \eqref{det0ab} , respectively.
\end{Lemma}
\begin{proof}
    We first decompose $V$ (see \eqref{hd11} for its definition) into its diagonal and non-diagonal parts 
    \begin{equation} \label{hx12b}
    \begin{split}
	V&=\frac{1}{n^m}\sum_{k=1}^{n}h(X_k, \ldots ,X_k)+\frac{1}{n^m}\sum_{1 \leq l_1\neq l_2= \ldots =l_m \leq n} h(X_{l_1}, \ldots ,X_{l_m})  \\ 
	&\quad  +\cdots + \frac{1}{n^m}\sum_{1 \leq l_1\neq \ldots  \neq l_m \leq n} h(X_{l_1}, \ldots ,X_{l_m})
	\end{split}
	\end{equation}
	Applying Lemma \ref{WeibConc} to the diagonal part yields, up to some constant depending only on $\beta$ and $\gamma$,
	\begin{align*}
	\norm{\frac{1}{n^m}\sum_{k=1}^{n}h(X_k, \ldots , X_k)}_\infty \lesssim B_n \sqrt{\frac{\log(nd)}{n^{2m-1}}}
	\end{align*}
	with $\p$-probability at least $1-o_K(1)$,	where we also used Lemma \ref{WeibullChar} to uniformly bound the mean of $h(X_k, \ldots ,X_k)$ by a multiple of $B_n$ that depends only on $\beta$. Next, we will exemplary inspect the term 
	\begin{align*}
	\frac{1}{n^m}\sum_{1 \leq l_1\neq l_2 \leq n} h(X_{l_1}, X_{l_2}, \ldots ,X_{l_2})    
	\end{align*}
	in detail, all other terms (except the very last) in the decomposition of $V_i$
	can be treated analogously. Note that $H(x_1,x_2)=h(x_1,x_2, \ldots ,x_2)$ defines a non-symmetric kernel of order two whose associated $U$-statistic is given by the preceding equation, which can be symmetrized  without changing the value of the associated $U$-statistic. Applying Lemma \ref{UStatConc} then yields
	\begin{align*}
	\norm{\frac{1}{n^m}\sum_{1 \leq l_1\neq l_2 \leq n} h(X_{l_1}, \ldots ,X_{l_2})    }_\infty \lesssim B_n \sqrt{\frac{\log(nd)}{n^{2m-3}}} +n^{-(m-2)}
	\end{align*}
	with probability at least $1-o_K(1)$ for some constant $C$ that depends only on $\beta$ (note that the mean is negligible by the same arguments as for the first term). The same arguments show that all terms in \eqref{hx12b} (except the last one) are of smaller order than $B_n\sqrt{(\log(nd)) / n}$.
	Finally, for the remaining term in \eqref{hx12b}, we have by Lemma \ref{UStatConc} that
	\begin{align*}
	 \norm{\frac{1}{n^m}\sum_{1 \leq l_1\neq \ldots  \neq l_m \leq n} h(X_{l_1}, \ldots ,X_{l_m})-\theta}_\infty\lesssim B_n\sqrt{\frac{\log(nd)}{n}},
	\end{align*} 
	with $\p$-probability at least $1-o_K(1)$, which proves the assertion of the lemma. 
\end{proof}

Next, we set
\begin{align}
\label{hx13} 
S_n^*=\sqrt{n}\underset{1\leq i\leq d}{\max}\theta_i\frac{U^*_{i}-V_{i}}{m\sqrt{\zeta_{1,i}}\Delta}\,,
\end{align}
which is a linearized version of  $\mathcal{T}_n^*$ (see \eqref{hd10a} for its definition).
We will show that $\mathcal{T}_n^*$ is well approximated by $S_n^*$. This will allow us to apply Gaussian approximation results to approximate the distribution of  $\mathcal{T}_n^*$.

\begin{Lemma}
\label{BootstrapApprox}
If the assumptions of Theorem \ref{Boot0} are satisfied, $\min_{1\le i\le d} \zeta_{1,i}\ge \underline{b}>0$ and  $\max_{i=1}^d|\theta_i|\leq \Delta$, we have that 
\begin{align} \label{hd11a}
    \left|\mathcal{T}^*_{n}-S_n^*\right|&\lesssim B_n^2\frac{(\log(nd))^{1+2/\beta}}{\sqrt{n}}+B_n^3\frac{(\log(nd))^{1+1/\beta}}{\sqrt{n}}
\end{align}
holds with $\p^*$ probability at least $1-o_K(1)$ on a set of $\p$-probability at least $1-o_K(1)$.
Here  the constant in inequality \eqref{hd11a} depends only on $\beta, \gamma, \underline{b}$.
This implies that \eqref{hd11a} holds uniformly for the subset of the class ${\cal H} _{0,boot}(\Delta)$ in \eqref{det0ab} for which $\min_{1\le i\le d} \zeta_{1,i}\ge \underline{b}>0$.

\end{Lemma}
\begin{proof}

We start by noting that an analogue of Lemma \ref{MaxBound} 
in the online supplement (which considers the maximum with respect to two indices)  and Assumption (A1') show that up to some universal constant
\begin{align*}
\underset{1 \leq i \leq d, 1 \leq j_1 <\ldots < j_m \leq n}{\max}\norm{h_i(X\xi_{j_1}, \ldots ,X\xi_{j_m})}_{\infty}\lesssim B_n(\log(dn))^{1/\beta} 
\end{align*}
with $\p$-probability at least $1-o_K(1)$.
Part $iii)$ of Lemma \ref{WeibullChar} then yields 
\begin{align}
\label{ConditionalKernelNorm}
\underset{1 \leq i \leq d, 1 \leq j_1 < \ldots  < j_m \leq n}{\max}\norm{h_i(X\xi_{j_1}, \ldots ,X\xi_{j_m})-V_i}_{\psi_2}^*\lesssim B_n(\log(dn))^{1/\beta} 
\end{align}
up to some universal constant, where 
	$ \norm{Z}_{\psi_2}^*:=\inf\{\nu>0 : \E^* [\psi_\beta(|Z|/\nu)]\leq 1\} $
denotes the Orlicz-Norm (of  a real-valued random variable $Z$) with respect to the conditional expectation $\E^*$.

Next we observe by the triangle inequality that
\begin{align}
\label{Triangle}
    \left|\mathcal{T}_{n}^*-S_n^*\right|&\lesssim
    \left| \underset{1\leq i\leq d}{\max}V_{i,\Delta}\frac{U^*_{i}-V_{i}}{\hat{\sigma}_i\Delta}-\sqrt{n}\underset{1\leq i\leq d}{\max}V_i\frac{U^*_{i}-V_{i}}{m\sqrt{\zeta_{1,i}}\Delta}\right|\\ &\qquad+
    \left|\mathcal{T}_{n}^* - \underset{1\leq i\leq d}{\max}V_{i,\Delta}\frac{U^*_{i}-V_{i}}{\hat{\sigma}_i\Delta}\right|  +\bigg | \sqrt{n}\underset{1\leq i\leq d}{\max}V_i\frac{U^*_{i}-V_{i}}{m\sqrt{\zeta_{1,i}}\Delta}
-S_n^* \bigg |\,. \notag
\end{align}
For the second summand we have
\begin{align}
\label{BootstrapDecomp}
   0\leq \mathcal{T}_{n}^* - \underset{1\leq i\leq d}{\max}V_{i,\Delta}\frac{U^*_{i}-V_{i}}{\hat{\sigma}_i\Delta}\le\underset{1\leq i\leq d}{\max}\frac{(U^*_{i}-V_{i})^2}{2\hat{\sigma}_i\Delta}
\end{align}
and, provided that $|\theta_i|\leq \Delta$ for all $i$, we claim
\begin{align}
\label{BootstrapSummandsVariance}
\left| \underset{1\leq i\leq d}{\max}V_{i,\Delta}\frac{U^*_{i}-V_{i}}{\hat{\sigma}_i\Delta}-\sqrt{n}\underset{1\leq i\leq d}{\max}V_i\frac{U^*_{i}-V_{i}}{m\sqrt{\zeta_{1,i}}\Delta}\right|&\lesssim B_n^3\frac{(\log(nd))^{1+1/\beta}}{\sqrt{n}}\,,\\
 \left|\underset{1\leq i\leq d}{\max}\frac{(U^*_{i}-V_{i})^2}{2\hat{\sigma}_i\Delta}-\sqrt{n}\underset{1\leq i\leq d}{\max}\frac{(U^*_{i}-V_{i})^2}{2m\sqrt{\zeta_{1,i}}\Delta}\right|&\lesssim  B_n^4\frac{(\log(nd))^{3/2+2/\beta}}{n} 
 \label{BootstrapSummandsVariance1}
\end{align}
with $\p$-probability at least $1-o_K(1)$, where the constants in the inequalities depend only on $\beta, \gamma$ and $\underline{b}$. This will help bounding the right hand term in \eqref{BootstrapDecomp} while simultaneously taking care of the first summand in \eqref{Triangle}. In view of \eqref{ConditionalKernelNorm}, an application of Lemma \ref{UStatConc} yields for the vector $U^*=(U_1^*, \ldots ,U_d^*)^{\top}$, that
\begin{align}
\label{hx13aa}
\norm{U^*-V}_\infty \lesssim B_n\frac{(\log(nd))^{1/2+1/\beta}}{\sqrt{n}}
\end{align}
 with $\p^*$ probability at least $1-o_K(1)$ with the bound depending only on $\beta$ and $\gamma$. Since by Assumption  $\min\limits_{1 \leq i \leq d, |\theta_i|>c}\zeta_{1,i}\ge \underline{b}$ this gives 
\begin{align}
\label{BootstrapQuadraticDecay}
    \sqrt{n}\underset{1\leq i\leq d}{\max}\frac{(U^*_{i}-V_{i})^2}{2m\sqrt{\zeta_{1,i}}\Delta}\lesssim B_n^2\frac{(\log(nd))^{1+2/\beta}}{\sqrt{n}}
\end{align}
with $\p^*$ probability at least $1-o_K(1)$. Combining \eqref{BootstrapSummandsVariance1} and \eqref{BootstrapQuadraticDecay} yields
\begin{align}
\label{b1}
\left|\mathcal{T}_{n}^* - \underset{1\leq i\leq d}{\max}V_{i,\Delta}\frac{U^*_{i}-V_{i}}{\hat{\sigma}_i\Delta}\right|  &\lesssim B_n^2\frac{(\log(nd))^{1+2/\beta}}{\sqrt{n}}+B_n^4\frac{(\log(nd))^{3/2+2/\beta}}{n} \notag \\
&\lesssim B_n^2\frac{(\log(nd))^{1+2/\beta}}{\sqrt{n}}~.
\end{align}\\

The estimate \eqref{BootstrapSummandsVariance} 
is  obtained as follows.  
First, we use the  inequality
\begin{align*}
   \sqrt{n} &\left| \underset{1\leq i\leq d}{\max}V_{i,\Delta}\frac{U^*_{i}-V_{i}}{\sqrt{n}\hat{\sigma}_i\Delta}- \underset{1\leq i\leq d}{\max}V_i\frac{U^*_{i}-V_{i}}{m\sqrt{\zeta_{1,i}}\Delta}\right|\\ 
   &\lesssim \sqrt{n}\norm{V-V_{\Delta}}_\infty \norm{U^*-V}_\infty 
 \max_{1\leq i\leq d} \frac{\big| \sqrt{n}\hat{\sigma}_i-m\zeta_{1,i}\big|}{\sqrt{n}\hat{\sigma}_i} ~.
\end{align*}
Secondly, we use 
\eqref{hx13aa} and Theorem \ref{VarConv} to bound the terms involving $U^*$ and $\hat{\sigma}_i$. Recalling \eqref{hd10}, we get
	\begin{align*}
	|V_i-V_{i,\Delta}|=\begin{cases}
	0 &\text{if } |V_i| \leq \Delta \\
	|V_i-\Delta| &\text{otherwise}\,.
	\end{cases}	
	\end{align*}
	As long as $|\theta_i|\leq \Delta$ and when \eqref{hx13aa} holds we can bound the latter quantity uniformly by $\norm{V-\theta}_\infty$  so that  Lemma \ref{BootstrapExpectationConv} is applicable to derive \eqref{BootstrapSummandsVariance} with the hidden constants depending only on $\beta, \gamma$ and $\underline{b}$. The bound \eqref{BootstrapSummandsVariance1} is obtained similarly. 
	
	For the last term on the right-hand side of \eqref{Triangle} we observe that 
	\begin{align}
	    \label{b2}
	    \bigg | \sqrt{n}\underset{1\leq i\leq d}{\max}V_i\frac{U^*_{i}-V_{i}}{m\sqrt{\zeta_{1,i}}\Delta}
-S_n^* \bigg |&\lesssim \sqrt{n}\norm{V-\theta}_\infty \norm{U^*-V}_\infty \notag\\
&\lesssim B_n^2\frac{(\log(nd))^{1+1/\beta}}{\sqrt{n}}
	\end{align}
	with $\p$-probability at least $1-o_K(1)$ by virtue of Lemma \ref{BootstrapExpectationConv}.\\
Combining \eqref{BootstrapSummandsVariance}, \eqref{b1} and \eqref{b2} yields, up to some constant depending only on $\gamma, \beta, \underline{b}$, that
\begin{align*}
     \left|\mathcal{T}_{n}^*-\sqrt{n}\underset{1\leq i\leq d}{\max}V_i\frac{U^*_{i}-V_{i}}{m\sqrt{\zeta_{1,i}}\Delta}\right|&\lesssim
     B_n^2\frac{(\log(nd))^{1+2/\beta}}{\sqrt{n}}+B_n^3\frac{(\log(nd))^{1+1/\beta}}{\sqrt{n}}
\end{align*}
with $\p^*$ probability at least $1-o_K(1)$ on a set of $\p$-probability at least $1-o_K(1)$.

\end{proof}

In the next step we  decompose the statistic $U^*$  into a linear and a non-linear part. The linear part of the Hoeffding decomposition (for more details see \citeSM{Hoeffding1948}) of $U^*$ conditional on $X_1, \ldots ,X_n$ is given by
\begin{align*}
h_1^X(\xi_1)=\E^*[h(X\xi_1, \ldots , X\xi_m)|\xi_1]=\frac{1}{n^{m-1}}\sum_{l_1, \ldots ,l_{m-1}=1}^{n}h(X\xi_1,X_{l_1}, \ldots ,X_{l_{m-1}})\,.
\end{align*}
To proceed we need the notation
\begin{align*}
S_{n,1}^*=\sqrt{n}\underset{1 \leq i \leq d}{\max}\theta_i \frac{\frac{1}{n}\sum_{j=1}^{n}h_{1,i}^X(\xi_j)-V_i}{\sqrt{\zeta_{1,i}}\Delta}\,.
\end{align*}
\begin{Lemma}
\label{BootstrapLinearization}
Under the assumptions of Theorem \ref{Boot0} we have for the statistic $S_n^*$ in \eqref{hx13} that
\begin{align*}
	|S_n^*-S_{n,1}^*|\lesssim B_n\frac{(\log(nd))^{1+1/\beta}}{n^{1/2-\gamma/\beta}}
\end{align*}	
 with $\p^*$-probability at least $1-n^{-\gamma/\beta}$ whenever \eqref{ConditionalKernelNorm} holds. Here the constant in the inequality depends only on $\beta$, and therefore the inequality holds uniformly over  the classes ${\cal H} _{0,boot}(\Delta)$ and $\mathcal{H}_1$ defined in \eqref{det0ab} and \eqref{det0}.
\end{Lemma}
\begin{proof}
By Theorem 5.1 in \citeSM{songetal2019} and Markov's inequality 
the non-linear part of the Hoeffding decomposition is bounded by some multiple of
 $B_n\frac{(\log(nd))^{1+1/\beta}}{n^{1-\gamma/\beta}}$ that depends only on $\beta$ with $\p^*$-probability at least $1-n^{-\gamma/\beta}$ whenever \eqref{ConditionalKernelNorm} holds. 
\end{proof}
The final result of this section provides a Gaussian approximation for the  statistic 
$S_{n,1}^*$. Note that
 $h_1^X(\xi_i)$ is not the bootstrap version of $h_1(X_i)$  and therefore  
 Lemma \ref{GaussApprox} is not applicable. Instead we will utilize a Gaussian approximation together with a bound on the distance of two Gaussian random vectors by the difference of their covariance matrices and their dimension.

Recalling the definition of $\Gamma$ from \eqref{GammaDefin}, we define the $d\times d$ diagonal matrix $B=\text{Diag}(\zeta_{1,1}^{-1/2},...,\zeta_{1,d}^{-1/2})$  and put 
\begin{align*}
\hat{\Gamma} &:=B\text{Cov}^*(h_1^X(\xi_1))B \\
& =B\Big (\frac{1}{n^{2m-1}}\sum_{l,l_1, \ldots , l_{2m-2}}^{n}h(X_l,X_{l_1}, \ldots ,X_{l_{m-1}})h(X_l,X_{l_m}, \ldots ,X_{l_{2m-2}})^\top-VV^\top\Big )B\,,
\end{align*}
where $\text{Cov}^*$ is the covariance operator with respect to the conditional expectation $\E^*$.
$\hat{\Gamma}$ is a rescaled version of the (conditional) covariance matrix  
of the vector $h_1^X(\xi_1)$. Further,
we introduce the matrices $\hat{\Lambda}$ and $\Lambda$ with entries
\begin{align*}
\hat{\Lambda}_{ij}=\hat{\Gamma}_{ij}\theta_i\theta_j \quad \text{ and } \quad
\Lambda_{ij}=\Gamma_{ij}\theta_i\theta_j\,,\qquad i,j=1,\ldots,d\,.
\end{align*} 
In the following discussion the symbol $a \leq b$ for vectors $a,b \in \mathbb{R}^d$ means coordinate-wise inequality.

\begin{Lemma}
	\label{LinearGaussApprox}
	Let $Z\sim N(0,\Lambda)$ and $Z^X \sim N(0,\hat{\Lambda})$ conditional on $X_1, \ldots ,X_n$. Suppose that the assumptions of Theorem \ref{Boot0} hold and that $|\theta|_{\min}>c>0$ for some constant $c$. Then we have  
	\begin{align*}
	\underset{\mathbf{x} \in \R^d}{\sup}\left|\p(Z \leq \mathbf{x} )-\p^*(Z^X \leq \mathbf{x} )\right|\lesssim \left(\frac{B_n^2\,(\log(nd))^5}{n}\right)^{1/6}
	\end{align*}
	with $\p$-probability at least $1-o_K(1)$. 
	Additionally, whenever \eqref{ConditionalKernelNorm} holds,  we have
    \begin{align}\label{eq:sfgsedg}
	\underset{x \in \R}{\sup}\left|\p^*(Z^X\leq (x,...,x)^\top )-\p^*(\Delta\, S_{n,1}^*\leq x)\right|\lesssim \left(\frac{B_n^2\,(\log(nd))^{5+\frac{2}{\beta}}}{n}\right)^{1/4}\,.
	\end{align}
	The  constants in both inequalities depend only on $\beta$ and $\gamma$. Therefore, both inequalities hold uniformly in the  subsets
	of the classes ${\cal H} _{0,boot}(\Delta)$ and $\mathcal{H}_1 $,  defined in \eqref{det0ab} and \eqref{det0}, for which $|\theta|_{\min}>c>0$.
	
\end{Lemma}
\begin{proof}
   We employ a 
    decomposition into $U$-statistics of orders up to $2m-1$ 
    \begin{align*}
        &\frac{1}{n^{2m-1}}\sum_{l,l_1, \ldots , l_{2m-2}}^{n}h(X_l,X_{l_1}, \ldots ,X_{l_{m-1}})h(X_l,X_{l_m}, \ldots ,X_{l_{2m-2}})^\top \\
        &=\frac{1}{n^{2m-1}}\sum_{l\neq l_1\neq  \ldots . \neq l_{2m-2}}^{n}h(X_l,X_{l_1}, \ldots ,X_{l_{m-1}})h(X_l,X_{l_m}, \ldots ,X_{l_{2m-2}})^\top + {R_n}   ~,
    \end{align*}
  where the term $R_n$ contains all 
   sums, where at least  two of the indices $l_i$  and $l_j$ ($i\not = j$) coincide
   (compare with the  the proof of Lemma \ref{BootstrapExpectationConv}).
   We then apply Lemma~\ref{UStatConc} to each $U$-statistic appearing in the above decomposition
   %
    to obtain, up to some constant depending only on $\gamma$ and $\beta$, that 
    \begin{align*}
        \max_{1\le i,j\le d}\big|\hat{\Gamma}_{ij}-\Gamma_{ij} \big| \lesssim B_n\sqrt{\frac{\log(nd)}{n}}
    \end{align*}
    with probability at least  $1-o_K(1)$. Finally, we use the Gaussian to Gaussian comparison from Lemma C.1 from \citeSM{UApprox} to establish the desired result. 
    
	The second bound \eqref{eq:sfgsedg} is an immediate consequence of Lemma \ref{GaussApprox}
	in the online supplement. Note that conditions (A) and (W), which are required for Lemma \ref{GaussApprox},  are satisfied with $B_n(\log(nd))^{1/\beta}$ instead of $B_n$ with $\p$-probability at least $1-o_K(1)$, which  follows from similar arguments as for the first bound and the fact that $\norm{h^X_{1,i}-V_i}_\infty$ is bounded by $B_n(\log(nd))^{1/\beta}$ with $\p$-probability at least $1-1/(nd)$.

\end{proof}

\subsubsection{Proof of Theorem \ref{Boot0}}
\label{sec432}
We start with the proof of \eqref{hd9a}. First assume that $|\theta|_{\min}>c>0$.
A combination of Lemmas \ref{BootstrapApprox} and  \ref{BootstrapLinearization} yields that under the null hypothesis
	\begin{align*}
		|S_{n,1}^*-\mathcal{T}^*_n|\lesssim
		c_{n,\epsilon}:= B_n\frac{(\log(nd))^{1+1/\beta}}{n^{1/2-\gamma/\beta}}+\frac{B_n^3(\log(nd))^{1+1/\beta}+B_n^2(\log(nd))^{1+2/\beta}}{\sqrt{n}}
	\end{align*}
with $\p^*$-probability at least $1-o_K(1)$  on a set of $\p$-probability at least $1-o_K(1)$, where all involved constants depend
only on $\beta,\gamma,c$ and $\underline{b}$. We hence obtain 
\begin{align}
\label{hx14}
\p^*(S^*_{n,1}>t
+c_{n,\epsilon})-o_K(1) & \leq \p^*(\mathcal{T}^*_n>t)  \leq \p^*(S^*_{n,1}>t-c_{n,\epsilon})+o_K(1)\,. 
\end{align}
Nazarov's inequality \citepSM[see for example][]{Nazarov} combined with the second part of Lemma \ref{LinearGaussApprox} then yields, up to some constant depending only on $\beta, \gamma ,\underline{b},c$, that
\begin{align*}
\underset{t \in \R}{\sup}	\left|\p^*(S^*_{n,1}>t\pm c_{n,\epsilon})-\p^*(S^*_{n,1}>t)\right|\lesssim \left(\frac{B_n^2(\log(nd))^{5+\frac{2}{\beta}}}{n}\right)^{1/4}+ c_{n,\epsilon}\sqrt{\log d}\,.
\end{align*}
In conjunction with  \eqref{hx14} and   the first part of  Lemma \ref{LinearGaussApprox}, we obtain  
\begin{align*}
	\underset{t \in \R}{\sup}\left|\p^*(\mathcal{T}_n^*\leq t)-\p(Z/\Delta \leq (t,\ldots,t)^{\top})\right|&\lesssim 
	d^{(1)}_{n,\epsilon}
\end{align*}
with $\p$-probability at least $1-o_K(1)$,
where $Z\sim N(0,\Lambda)$ and 
$$
	d^{(1)}_{n,\epsilon}:=
\left(\frac{B_n^2(\log(nd))^{5+\frac{2}{\beta}}}{n}\right)^{1/4}+\left(\frac{(\log(nd))^5B_n^2}{n}\right)^{1/6} +c_{n,\epsilon}\sqrt{\log d}  .
$$

We now derive a similar Gaussian approximation for the quantity 
\begin{align*}
T_{n,\Delta}=\underset{1\leq i\leq d}{\max}\ \frac{U^2_{i}-\theta_i^2}{2\hat{\sigma}_{i}\Delta}\,.
\end{align*}
Using Lemma \ref{RealVar} as well as the same arguments as above 
(with Lemma \ref{AntiConcentrationSampleMean} replacing  Nazarov's inequality), we get
\begin{align*}
	\underset{t \in \R}{\sup}\left|\p(T_{n,\Delta}\leq t)-\p(Z/\Delta \leq (t,...,t)^{\top}\right|\lesssim
		d_{n,\epsilon}^{(2)}\,,
\end{align*}
where
$$
d_{n,\epsilon}^{(2)}=
\left(\frac{B_n^2(\log(nd))^{4+\frac{2}{\beta}}}{n}\right)^{1/4}+
\frac{B_n^3(\log(nd))^{3/2}}{\sqrt{n}}+\frac{B_n(\log(nd))^{3/2}}{n^{1/2-\gamma/\beta}}.
$$
where all constants involved depend only on $\beta, \gamma ,\underline{b}$ and $c$.
Since $d^{(2)}_{n,\epsilon} \lesssim d^{(1)}_{n,\epsilon}$, we deduce
\begin{align*}
\underset{t \in \R}{\sup}\left|\p(T_{n,\Delta}\leq t)-\p^*(\mathcal{T}_n^*\leq t)\right|\lesssim
	d^{(1)}_{n,\epsilon}
\end{align*}
with $\p$-probability at least $1-o_K(1)$.
Because $\mathcal{T}_{n,\Delta}\leq T_{n,\Delta}$ this yields \eqref{hd9a} in the case $| \theta |_{\min}  >c$.

We conclude the proof considering the case where $|\theta|_{\min}$ is not bounded away from zero.  First, we assume that there exists a sufficiently small constant 
such that the set $I:=\{1 \leq i \leq d \colon  |\theta_i|> c\}$ is non-empty (in other words, we have $\theta_i \not =0$ for at least one  index $i\in \{ 1, \ldots , d\}$.
 By the arguments in the proof of Theorem \ref{alpha} we observe that 
\begin{align*}
\mathcal{T}_{n,\Delta}= \mathcal{T}_{n,\Delta}^I := 
\underset{i \in I}{\max}\frac{U_i^2-\Delta^2}{2 \hat{\sigma}_i\Delta}
\end{align*}
with high probability uniformly with respect to the class $\mathcal{H}_0$.
 Let $\mathcal{T}_{n}^{**}$  denote the analogue of the statistic $\mathcal{T}_{n}^{*}$
defined in \eqref{hd10a}, where the maximum
 is only taken over the set $I$, and denote by  $\hat{q}_{1-\alpha}^{**}$ the corresponding $(1-\alpha)$-quantile.
Observing that $q^*_{1-\alpha} \geq q^{**}_{1-\alpha}$, we have by the arguments  given in the above paragraph and the first part of this proof that 
\begin{align*}
 	\limsup_{n \to \infty } \underset{F \in {\cal H} _0(\Delta)  }{\sup}
	\mathbb{P}     (\mathcal{T}_{n,\Delta} \geq \hat{q}_{1-\alpha}^{*}) &= 
	 	\limsup_{n \to \infty } \underset{F \in {\cal H} _0(\Delta)  }{\sup}
	\mathbb{P}     (\mathcal{T}_{n,\Delta}^I \geq \hat{q}_{1-\alpha}^{*}) \\
 & \leq 	\limsup_{n \to \infty } \underset{F \in {\cal H} _0(\Delta)  }{\sup}
	\mathbb{P}     (\mathcal{T}_{n,\Delta}^I \geq \hat{q}_{1-\alpha}^{**})\leq \alpha ~,
\end{align*}
which yields \eqref{hd9a}.  Next we  consider the case where $\theta_i= 0$ for all $i=1, \ldots , d$  (in other words $I = \emptyset$ for all $c>0$). From the considerations in the proof of Theorem \ref{alpha} we know that in this case $\mathcal{T}_{n,\Delta}\lesssim -\sqrt{n}$. It hence suffices to show that $\mathcal{T}_n^*$ is lower bounded by a constant with probability at least $1-o_K(1)$. To this end note that 
\begin{align*}
    (U_i^*-V_i+V_{i,R})^2-V_{i,R}^2=(U_i^*-V_i)^2+2V_{i,R}(U_i^*-V_i)~.
\end{align*}
The first term on the right-hand side is positive, while the second term is bounded in absolute value by $B_n^2\frac{(\log(nd))^{1+\beta}}{\sqrt{n}}$ due to Lemmas \ref{BootstrapExpectationConv} and \ref{UStatConc}. This establishes the desired conclusion in the case $I = \emptyset$  and completes the proof under the null hypothesis.

Finally, we turn to the consistency part of Theorem \ref{Boot0}.
We have already seen in the proof of Theorem \ref{consistency} that there exists some constant $C>0$ such that for 
 $\xi= {\max_{1 \leq i  \leq d}~}\theta_{i}^2-\Delta^2=\theta_{i_0}^2-\Delta^2$
\begin{align*}
\mathcal{T}_{n,\Delta}\geq O_\p\left(B_n\sqrt{\log(d)}\frac{1}{\sqrt{n}\hat{\sigma}_{i_0}}\right)
+\frac{C\xi \sqrt{n}}{\sqrt{n}\hat{\sigma}_{i_0}}~.
\end{align*}
uniformly over $\mathcal{H}_1$. Note that for $\xi\downarrow 0$ we have $\xi \simeq \max_{1 \leq i\leq d }|\theta_i|-\Delta$. 
On the other hand, the arguments used in the proof of Lemma \ref{BootstrapApprox} show that
\begin{align*}
    \mathcal{T}^*_n \lesssim B_n(\log d)^{1/2}(\log(nd)^{1/\beta}
\end{align*}
with $\p^*$-probability at least $1-o_K(1)$ on a set of $\p$-probability at least $1-o_K(1)$ which implies that any fixed quantile of $T_n^*$ is eventually bounded (up to some constant that does not change with $n$) by $B_n(\log d)^{1/2}(\log(nd)^{1/\beta}$ with $\p$-probability at least $1-o_K(1)$. 
 Moreover, 
if the kernel	$h$ in \eqref{hd3} is bounded we can obtain \eqref{hd11a} without the additional
 factor
 $(\log(nd))^{1/\beta}$, which yields
\begin{align*}
    \mathcal{T}_n^*\lesssim B_n(\log d)^{1/2}
\end{align*}
and hence establishes the improved rate  in Theorem \ref{Boot0} for bounded kernels.

\subsubsection{Proof of Theorem \ref{optimality}}

Let $I_p$ be the $p$-dimensional identity matrix and $J_{a,b}$ the $a\times b$ matrix filled with ones and $J_p:=J_{p,p} $. Let $U_k$ be the $p \times 2$ matrix with entries $U_{k,11}=U_{k,k2}=1$ and $U_{k,ij}=0$ otherwise
 and write $e_1,\ldots,e_p$ for the canonical basis vectors of $\R^p$. We then define $\Sigma_{p,a}=(1-a)I_p+a J_p $ and $C=J_2-I_2$. Set $M_0=\Sigma_{p,\Delta}$ and $$M_k=(1-\Delta)I_p+\Delta J_p+\rho e_1 e_k^\top +\rho e_k e_1^\top =\Sigma_{p,\Delta}+\rho\, U_kCU_k^\top\,,\quad 2 \leq k \leq p\,,$$
 where $\rho=c_0(\log(p)/n)^{1/2}$ for some small constant $c_0=c_0(\Delta)$, which  will be specified later. Note that for sufficiently small $\rho$, the matrices $M_k$ are correlation matrices. 

Let $\mu_p$ be the uniform measure on the set $\mathcal{F}(p)=\{M_2  , \ldots , M_p \}$. We denote by $\p_\Sigma = 
\mathcal{N}_p(0,\Sigma)  \otimes \ldots \otimes  \mathcal{N}_p(0,\Sigma)$  the  product  probability measure induced by $n$ i.i.d.\ $p$-dimensional random vectors $Z_1, \ldots , Z_n \sim \mathcal{N}_p(0,\Sigma)$ and define $\p_{\mu_{p}}=\int \p_\Sigma d\mu_p(\Sigma)$. Let $\p_0$ denote the $n$-fold product probability measure of $\mathcal{N}_p(0,M_0)$. By the same arguments as in the proof of Theorem~5 in \citeSM{hanetal2017}, we obtain
\begin{align}
\label{hd45}
    \inf_{T_\alpha \in \mathcal{T}_\alpha}\sup_{\Sigma \in \mathcal{F}(p)}\p_\Sigma(T_\alpha \text{ does not reject } H_0)\geq 1-\alpha-{1 \over 2}
    \big (\E_{\p_0}[\mathcal{L}^2_{\mu_p}(Y)]-1\big )^{1/2}\,,
\end{align}
where
\begin{align*}
    \mathcal{L}_{\mu_p}(y)={\mathrm{d}\p_{\mu_p} \over \mathrm{d}\p_0} (y)=\frac{1}{p-1}\sum_{k=2}^p\Big [\prod_{i=1}^n\frac{\left|M_0\right|^{1/2}}{\left|M_k\right|^{1/2}} \exp\big (-\frac{1}{2}y_i^\top (M_k^{-1}-M^{-1}_0)y_i\big)\Big]
\end{align*}
with $|M_k|$ being the determinant of $M_k$.
Squaring and taking expectations yields
\begin{align*}
\E_{\mathbb{P}_0}\big [ \mathcal{L}^2_{\mu_p}(Y)\big ] = \frac{1}{(p-1)^2}\sum_{k,l=2}^p \E_{\mathbb{P}_0} &\Big [\prod_{i=1}^n\frac{\left|M_0\right|^{1/2}}{\left|M_k\right|^{1/2}}\frac{\left|M_0\right|^{1/2}}{\left|M_l\right|^{1/2}}\\ &\times \exp\big(-\frac{1}{2}Y_i^\top (M_k^{-1}+M_l^{-1}-2M_0^{-1})Y_i\big  )\Big]\,,
\end{align*}
where $Y=(Y_1,...,Y_n)$ 
and the random vectors $Y_1, \ldots , Y_n$ are independent with distribution $\p_0 = {\cal N}_p (0,M_0)$. 
By definition, the matrix $M_k$ is a rank two perturbation of $M_0$ and thus we can obtain its inverse by  the Woodbury matrix identity. Lengthy but straightforward calculations then yield
\begin{align*}
    M_0 ^{-1}-M_k^{-1}
    &=\frac{1}{a(a+2b)}\left[(u-v)U_k+vJ_{p,2}\right]\left((a+2b)I_2-bJ_2\right)\left[(u-v)U_k^\top +vJ_{2,p}\right],
\end{align*}
where 
\begin{align*}
u&:=\frac{1}{1-\Delta}\Big(1-\frac{\Delta}{1+(p-1)\Delta}\Big)~ ,\quad ~v:=\frac{-\Delta}{(1-\Delta)(1+(p-1)\Delta)}\,,\\
a & :=\frac{1}{1-\Delta}-\frac{1}{\rho}~ ,\quad
b:=\frac{1}{\rho}-\frac{\Delta}{(1-\Delta)(1+(p-1)\Delta)}\,.
\end{align*}
Denoting $T^{kl}=M_k^{-1}+M_l^{-1}-2M_0^{-1}$ we have by standard results on the moment generating function of a Gaussian quadratic form that
\begin{align*}
    \E\left[\exp\left(-\frac{1}{2}Y_i^\top T^{kl}Y_i\right)\right]=\left|I_d+T^{kl}\Sigma_{p,\Delta}\right|^{-1/2}\,.
\end{align*}
We will show below that these determinants 
attain only two values depending on whether $k=l$ or $k \not =l$. Hence, observing that $|M_k| = |M_2|$ for $k=2, \ldots , d$ we  obtain 
\begin{align}
\label{hd43}
\E_{\p_0}  \big [ \mathcal{L}^2_{\mu_p}(Y) \big] &= \frac{1}{p-1}\prod_{i=1}^n\frac{ |M_0 |}{ |M_2 |}\left|I_p+T^{22}\Sigma_{p,\Delta}\right|^{-1/2} \notag \\ & \quad + \frac{p-2}{p-1}
\prod_{i=1}^n\frac{|M_0|}{|M_2|}\left|I_p+T^{23}\Sigma_{p,\Delta}\right|^{-1/2} =:A_{11}+A_{22}\,. 
\end{align}

We now investigate the different terms separately. {First we consider the ratio $|M_0 | / |M_2 |$
which appears in   both terms in \eqref{hd43}}. Using the fact that the eigenvalues of an equicorrelation matrix $\Sigma_{p,a}$ are $1-a$ with multiplicity $p-1$ and $1+(p-1)a$ with multiplicity 1 we have
\begin{align*}
\left| {M_0} \right|&=(1-\Delta)^{p-1}(1+(p-1)\Delta)\,.
\end{align*}
For $M_2$ we have the block decomposition
$$M_2 = 
     \begin{pmatrix}
     \Sigma_{2,\Delta+\rho} & \Delta J_{2,p-2}\\
     \Delta J_{p-2,2} & \Sigma_{p-2,\Delta}
     \end{pmatrix}$$
     from which we deduce that 
\begin{align*}
|M_2| &= \big|\Sigma_{2,\Delta+\rho} \big| ~ \big |\Sigma_{p-2,\Delta}-\Delta^2 J_{2,p-2} \Sigma_{2,\Delta+\rho}^{-1} J_{p-2,2} \big |\\
&= \big|\Sigma_{2,\Delta+\rho} \big| ~\Big| \frac{\Delta(1-\Delta +\rho)}{1+\Delta+\rho}J_{p-2} +(1-\Delta) I_{p-2} \Big|\\
  \nonumber 
&=(1-\Delta-\rho)(1+\Delta+\rho)(1-\Delta)^{p-3}\Big [(p-2)\frac{\Delta(1-\Delta+\rho)}{1+\Delta+\rho}+1-\Delta\Big ]\,.
\end{align*}
Hence, we get
\begin{align}
 {
\frac{\left|M_0\right|}{\left|M_k\right|}}
& =\frac{(1-\Delta)^2(1+(p-1)\Delta}{(1-\Delta-\rho)(1+\Delta+\rho)\big [(p-2)\frac{\Delta(1-\Delta+\rho)}{1+\Delta+\rho}+1-\Delta\big ]}\,. 
\label{DetQuotient}
\end{align}

{Next we consider the determinant involving the matrix $T^{kk}$}. We start by observing that 
\begin{align*}
    T^{kk}\Sigma_{p,\Delta}=2[(u-v)U_k+vJ_{p,2}]MU_k^\top \,,
\end{align*}
where $M:=\frac{-1}{a(a+2b)}\left((a+2b)I_2-bJ_2\right)$. An application of the Weinstein–Aronszajn identity \citeSM{Akritas1996} then yields
\begin{align} |I_p+T^{kk}\Sigma_{p,\Delta}|&=|I_2+2MU_k^\top [(u-v)U_k+vJ_{p,2}]| \nonumber \\
    &=|I_2+2M[(u-v)I_2+vJ_2]| \nonumber \\
    &=\Big|I_2-\frac{2}{a(a+2b)}[(a+2b)(u-v)I_2+((a+2b)v-b(u+v))J_2]\Big| \nonumber \\
    &=|q_1I_2+q_2J_2|=(2q_2+q_1)q_1\,,\label{kkDeterminant} 
\end{align}
where 
\begin{align*}
q_1&=1-\frac{2}{a}(u-v)=\frac{1+\rho-\Delta}{1-\rho-\Delta}\,,\\
q_2&=\frac{-2v}{a}+\frac{2b(u+v)}{a(a+2b)}=\frac{\rho(1-\Delta)2((p-1)\Delta+1)}{(\rho-1+\Delta)((-p+1)\Delta^2+((p-3)\rho+d-2)\Delta+\rho+1)}.
\end{align*}

Combining \eqref{DetQuotient} and \eqref{kkDeterminant} then yields 
\begin{align*}
\log A_{11} &= \log \Big [
    \frac{1}{p-1}\prod_{i=1}^n\frac{\left|M_0\right|}{\left|M_2\right|}\left|I_p+T^{22}\Sigma_{p,\Delta}\right|^{-1/2} \Big ] \\
    \nonumber 
  &=-\log(p-1)+ {n \over 2} \Big [ 4\log(1-\Delta)+2\log(1+(p-1)\Delta)
 \\
    \nonumber
    & ~~~~- \log((1-\Delta)^2-\rho^2)-\log((1+(p-2)\Delta(1-\Delta)-\Delta^2)^2-((p-3)\Delta+1)^2\rho^2)
    \Big ] \\
  &=-\log(p-1)+ {n \over 2} \Big [\frac{-C}{p}+o(p^{-1})+\frac{2\rho^2}{(1-\Delta)^2}
  \Big ]\,, \qquad n\to \infty\,,
     \nonumber
\end{align*}
where we used a Taylor expansion  for $\log(1+x)$ 
in the last step (assuming that $\rho\rightarrow 0$) and $C$ is some positive constant. Therefore we obtain  
\begin{align}
\label{kkterm}
A_{11} = o(1) 
\end{align}
if we choose $\rho^2=c_0^2\log(p)/n$, where the constant $c_0$ satisfies  $c_0<1-\Delta$.

{For the determinant involving $T^{kl}$ in the $A_{22}$ term in \eqref{hd43} }
we obtain by straightforward calculations that
\begin{align*}
    |I_p+T^{kl}\Sigma_{d,\Delta}|&=|I_p+T^{23}\Sigma_{p,\Delta}|\\
    &=\left|I_3-\frac{1}{a(a+2b)}\begin{pmatrix}
    2[u(a+b)-vb] & v(a+b)-ub & v(a+b)-ub\\
    v(a+b)-ub & u(a+b)-vb & av\\
    v(a+b)-ub & av & u(a+b)-vb
    \end{pmatrix}\right|\,.
    \end{align*}

Tedious but straightforward calculations  yield
\begin{align*}
\Big ( \frac{|M_0|}{|M_2|} \Big )^2|I_p+T^{kl}\Sigma_{p,\Delta}|^{-1} =\frac{f}{g}~,
\end{align*}
where
{\small\begin{align*}
    f&=((-p+1)\Delta^2+((\rho+1)p-3\rho-2)\Delta+\rho+1)(1+(p-1)\Delta)^2(-1+\Delta)^4\,,\\
    g&=\left(1+(p-1)^2\Delta^4+(-2p^2+6p-4)\Delta^3+(\rho^2p+p^2-3\rho^2-6p+6)\Delta^2+(\rho^2+2p-4)\Delta\right)\\
    & \quad \quad \times (1-\Delta+\rho)\left(1+(p-1)\Delta^3+(-2p+2\rho+3)\Delta^2+(\rho^2+p-2\rho-3)\Delta\right)\,.
\end{align*}}

Once again assuming $\rho\rightarrow 0$, taking the logarithm of $f/g$ and using the Taylor expansion of $\log(1+x)$ yields that $\log(f/g)=O(p^{-2}+\rho p^{-1})$ so that
\begin{align*}
 A_{22}=   \frac{p-2}{p-1}\left[\prod_{i=1}^n\frac{\left|M_0\right|}{\left|M_2\right|}\left|I_p+T^{23}\Sigma_{p,\Delta}\right|^{-1/2} \right]=\exp(o(1))(1+o(1)) =
 1+o(1)\,,
\end{align*}
 where we used $\log(p)n/p^2=o(1)$. 
Observing \eqref{hd43} and \eqref{kkterm} we obtain
$$
\E_{\p_0}  \big [ \mathcal{L}^2_{\mu_p}(Y) \big] =
A_{11} + A_{22} =1 + o(1)
$$
and the assertion of the theorem follows from \eqref{hd45}, completing the proof.

\subsection{Proof of Theorem \ref{BootAbs}}
The proofs are structurally the same as those for Theorems \ref{alpha}, \ref{consistency} and \ref{Boot0}, one only needs to substitute the squares of $U_i, \theta_i, U_i^*$ and $ V_i$ by absolute values and omit the variances and their estimators (this only ever improves any of the bounds appearing in the proofs). We indicate how to adjust the arguments for those cases where substantial differences occur.  \\
\textbf{Differences in the proof of Theorem \ref{alpha}:}\\
In the proof of Lemma \ref{RealVar} the quantity $\hat T_{n,2}$ is instead given by
\begin{align*}
    \1\{U_i\theta_i<0\}\max_{1 \leq i \leq d}\left|\frac{|U_i|-|\theta_i|}{2m}\right|
\end{align*}
This quantity is zero with probability at least $1-o_K(1)$ by Lemma \ref{UStatConc} which yields the required bound \eqref{UBound}. Similarly the decomposition mentioned after \eqref{hd1} of $U_i^2-\theta_i^2$ into a quadratic and linear part is instead the decomposition of $|U_i|-|\theta_i|$ into a part where they have the same sign and a part where they do not.\\
\textbf{Differences in the proof of Theorem \ref{Boot0}}\\
In the proof of Lemma \ref{BootstrapApprox} the  upper bound in the equation \eqref{BootstrapDecomp} is instead given by \begin{align*}   
    \1\{U_i^*V_i<0\}\max_{1 \leq i \leq d}\left|\frac{|U_i^*|-|V_i|}{2m\Delta}\right|
\end{align*}
Which is 0 with $\p^*$ probability at least $1-o_K(1)$ on a set of $\p$ probability at least $1-o_K(1)$ whenever $\min_{1 \leq i \leq d}|\theta_i|>c>0$  for some $c>0$ (this not a restriction since the resulting inequalities are only needed for this case in the proof). This can be seen by  applying Lemma \ref{BootstrapExpectationConv}  in conjunction with Lemma \ref{UStatConc}.

\section{Further technical details}
\label{sec6}
  \def\theequation{B.\arabic{equation}}	
	\setcounter{equation}{0}
	
    \subsection{Randomized Lindeberg Method}
\label{sec61}

In this section we state two important auxiliary results (Lemmas \ref{GaussApprox}  and \ref{AntiConcentrationSampleMean}), which will be used in the proofs of our main results in Section \ref{sec5}. They are a consequence of a general Gaussian approximation result (Theorem \ref{Lindeberg}), which is proved in Section \ref{sec511} via the iterative randomized Lindeberg method.

\subsubsection{A Gaussian approximation  and its consequences}

Let $V_1, \ldots ,V_n, Z_1, \ldots ,Z_n$ denote independent random vectors in $\R^d$, where
 $V_i=(V_{i1}, \ldots, V_{id})^\top$ and $Z_i=(Z_{i1}, \ldots, Z_{id})^\top$ 
 for $i=1, \ldots , n$. We also assume that the following conditions hold for the vectors
 $V_1, \ldots ,V_n, Z_1, \ldots ,Z_n$.
 There exists a sequence of constants $B_n$ such that:

\medskip 

\textbf{Condition V:} There exists a constant $C_v>0$ such that for all $j$
\begin{align*}
	\frac{1}{n}\sum_{i=1}^{n}\E[V_{ij}^2+Z_{ij}^2]\leq C_v, \quad \frac{1}{n}\sum_{i=1}^{n}\E[V_{ij}^4+Z_{ij}^4]\leq B_n^2C_v\,.
\end{align*}
\textbf{Condition P:} There exists a constant $C_p\geq 1$ such that for all $i$
\begin{align*}
	\p\left(\norm{V_i}_\infty \lor \norm{Z_i}_\infty>C_pB_n(\log(dn))^{1/\beta}\right)\leq \frac{1}{n^4}\,.
\end{align*}
\textbf{Condition B:} There exists a constant $C_b>0$ such that for all $i$
\begin{align*}
	\E[\norm{V_i}^8_\infty+\E[\norm{Z_i}^8_\infty]\leq C_bB_n^8(\log(dn))^{8/\beta}\,.
\end{align*}
\textbf{Condition A:} There exists a constant $C_a>0$ such that for all $(y,t)\in \R^d\times \R_+$, we have
\begin{align*}
	\p\left(\frac{1}{\sqrt{n}}\sum_{i=1}^{n}Z_i\leq y+t\right)-\p\left(\frac{1}{\sqrt{n}}\sum_{i=1}^{n}Z_i\leq y\right) \leq C_a t \sqrt{\log d}\,.
\end{align*}

Here $y+t$ means addition of $t$ to every component of $y$.
The following result, which will be proved in Section \ref{sec511}, will be crucial for Lemmas \ref{GaussApprox} - \ref{AntiConcentrationSampleMean}. Its proof uses distributional approximations via the Iterative Randomized Lindeberg Method
and is structurally the same as in \citeSM{ImprovedApprox}. However, we require a weaker decay in the tails at the cost of a weaker bound.

\begin{theorem}[Iterative Randomized Lindeberg Method]
	\label{Lindeberg}
	Suppose that conditions V,P,B and A are satisfied. In addition, suppose that for some positive constant $C_m$
	\begin{align*}
		\underset{1\leq j,k\leq d}{\max}\left|\frac{1}{\sqrt{n}}\sum_{i=1}^{n}(\E[V_{ij}V_{ik}]-\E[Z_{ij}Z_{ik}])\right|\leq C_mB_n(\log(dn))^{1/\beta}\,, \\
		\underset{1\leq j,k,l\leq d}{\max}\left|\frac{1}{\sqrt{n}}\sum_{i=1}^{n}(\E[V_{ij}V_{ik}V_{il}]-\E[Z_{ij}Z_{ik}Z_{il}])\right|\leq C_mB_n^2 (\log(dn))^{2/\beta}~.
	\end{align*}
	Then it holds
	\begin{align*}
		\underset{y \in \R^d}{\sup}\left|\p\left(\frac{1}{\sqrt{n}}\sum_{i=1}^{n}V_i\leq y\right)-\p\left(\frac{1}{\sqrt{n}}\sum_{i=1}^{n}Z_i\leq y\right)\right|\leq C\left(\frac{B_n^2(\log(dn))^{4+2/\beta}}{n}\right)^{1/4}\,,
	\end{align*}
	where $C>0$ is a constant depending only on $C_v, C_p, C_b, C_a, C_m$.
\end{theorem}

Theorem \ref{Lindeberg} has several important consequences, which are now stated in Lemma~\ref{GaussApprox} and Lemma~\ref{AntiConcentrationSampleMean}
and used in the proofs in Section \ref{sec5}. 
For a precise formulation we require the following assumptions.

  Let 
$X_1, \ldots ,X_n \in \R^d$ denote  i.i.d. centred random vectors, 
$ X_i=(X_{i1},\dots,X_{id})^\top$, satisfying the following 
Assumptions:
\begin{description}
	\item \textbf{(A)}: There exists a sequence of constants $(B_n)_{n \in \mathbb{N}}$ such that for $1 \leq j \leq d$ we have 
	$\norm{X_{1j}}_{\psi_\beta}\lesssim B_n$ for some $0<\beta\leq 2$.
	\item \textbf{(W)}: There exist constants $\sigma_{min}>0$ and $D>0$ such that for all $j$ 
	\begin{align*}
		\sigma_{min} \leq \frac{1}{n}\sum_{i=1}^{n}\E[X^2_{ij}]\leq D \quad \text{ and } \quad
		\frac{1}{n}\sum_{i=1}^{n}\E[X^4_{ij}]\leq B_n^2D\,.
	\end{align*}
	
\end{description}
We begin with a result describing the deviation between the empirical moments of the centered vectors 
$$
\Tilde{X}_k=  (\Tilde{X}_{k1},\dots, \Tilde{X}_{kd})^\top := X_k - \bar {X} = (X_{k1}-\bar {X}_{ 1},\dots,X_{kd}-\bar {X}_{d})^\top\,,
$$
where $\bar{X}_j=\frac{1}{n}\sum_{i=1}^nX_{ij}$, and the covariance matrix $\mathbb{E}{[X_k  X^\top_k]}.$ 
\begin{Lemma}
	\label{MomentBound}
	Suppose that assumptions (A)  and (W) hold.  Then there exists a universal constant $c>0$, constants $C,D>0$ and $n_0 \in \mathbb{N}$ depending only on $\beta$, $\sigma_{min}$ and $B_n$ such that for all $n\geq n_0$ the inequality 
	\begin{align*}
		B_n^2(\log(dn))^{4+2/\beta}\leq cn
	\end{align*}
	implies that the inequalities
	\begin{align}
		\label{det41}
		\frac{\sigma_{min}}{2}\leq \frac{1}{n}\sum_{i=1}^{n}\tilde{X}^2_{ij} &\leq D\,, \\
		\label{det42}
		\frac{1}{n}\sum_{i=1}^{n}\tilde{X}^4_{ij} &\leq B_n^2D\,,\\	
		\label{det43}
		\underset{1\leq k,j\leq p}{\max}\left|\frac{1}{\sqrt{n}}\sum_{k=1}^{n}(\tilde{X}_{ik}\tilde{X}_{ij}-E[X_{ik}X_{ij}])\right| &\leq CB_n(\log(dn))^{1/\beta}\,,\\
		\label{det44}
		\underset{1\leq k,j,l\leq p}{\max}\left|\frac{1}{\sqrt{n}}\sum_{i=1}^{n}(\tilde{X}_{ik}\tilde{X}_{ij}\tilde{X}_{il}-E[X_{ik}X_{ij}X_{il}])\right| &\leq CB_n^2(\log(dn))^{2/\beta}
	\end{align}
	hold jointly with probability at least $1-1/n$.
	
\end{Lemma}
\begin{proof}
		Let $A=5L(C_1+C_2)$ for some $C_1, C_2$ to be specified later and denote by $\mathcal{A}$ the event that the inequalities
	\begin{align*}
		&\underset{1\leq k\leq d}{\max}\left|\frac{1}{\sqrt{n}}\sum_{k=1}^{n}X_{ik}\right|\leq A\sqrt{\log(dn)}\,,\\
		&\underset{1\leq k,j\leq d}{\max}\left|\frac{1}{\sqrt{n}}\sum_{k=1}^{n}(X_{ik}X_{ij}-E[X_{ii}X_{ij}])\right|\leq AB_n(\log(dn))^{1/\beta}\,,\\
		&\underset{1\leq k,j,l\leq d}{\max}\left|\frac{1}{\sqrt{n}}\sum_{k=1}^{n}(X_{ik}X_{ij}X_{kl}-E[X_{ii}X_{ij}X_{il}])\right|\leq AB_n^2(\log(dn))^{2/\beta}\,,\\
		&\underset{1\leq k,j,l,r\leq d}{\max}\left|\frac{1}{\sqrt{n}}\sum_{k=1}^{n}(X_{ik}X_{ij}X_{kl}X_{ir}-E[X_{ii}X_{ij}X_{il}X_{ir}])\right|\leq AB_n^3(\log(dn))^{3/\beta}
	\end{align*}
	hold jointly.
	Noting that
	\begin{align*}
		&\underset{1\leq k,j\leq d}{\max}\left|\frac{1}{\sqrt{n}}\sum_{i=1}^{n}(\tilde{X}_{ik}\tilde{X}_{ij}-E[X_{ii}X_{ij}])\right| \\&\leq \underset{1\leq k,j\leq d}{\max}\left|\frac{1}{\sqrt{n}}\sum_{i=1}^{n}(X_{ik}X_{ij}-E[X_{ii}X_{ij}])\right|+\sqrt{n}\underset{1\leq k \leq d}{\max}|\bar{X}_{k}|^2 
	\end{align*}
	and 
	\begin{align*}
		&\underset{1\leq k,j,l\leq d}{\max}\left|\frac{1}{\sqrt{n}}\sum_{k=1}^{n}(\tilde{X}_{ik}\tilde{X}_{ij}\tilde{X}_{il}-E[X_{ik}X_{ij}X_{il}])\right| \\
		&\leq \underset{1\leq k,j,l\leq d}{\max}\left| \frac{1}{\sqrt{n}}\sum_{k=1}^{n}(X_{ik}X_{ij}X_{il}-E[X_{ik}X_{ij}X_{il}])\right|\\
		&+2\sqrt{n}\underset{1 \leq k \leq d}{\max}|\bar{X}_k|^3+\underset{1 \leq k,j,l\leq d}{\max}|\bar{X_l}|\left|\frac{3}{\sqrt{n}}\sum_{i=1}^{n}X_{ik}X_{ij}\right|
	\end{align*}
	yields the   bounds \eqref{det43} and \eqref{det44} on $\mathcal{A}$ .
	Considering 
	\begin{align*}
		\underset{1\leq k \leq d}{\max}\left|\frac{1}{n}\sum_{i=1}^{n}(\tilde{X}^2_{ik}-\E[X_{ik}^2])\right|
		& \leq \underset{1 \leq k,j\leq d}{\max} \left|\frac{1}{n}\sum_{i=1}^{n}(\tilde{X}_{ik}\tilde{X}_{ij}-\E[X_{ik}X_{ij}])\right| \\
		& \leq \frac{CB_n(\log(dn))^{1/\beta}}{\sqrt{n}}\leq \frac{\sigma_{min}}{2}
	\end{align*}
	yields \eqref{det41} on $\mathcal{A}$ , and \eqref{det42} follows by  similar considerations.\\
	
	We now show that we can find $C_1,C_2$ such that $\mathcal{A}$ has probability at least $1-1/n$. Fix $m \in \{1,2,3,4\}$ and let $P=\{1, \ldots ,d\}^m$. We denote 
	$y^h=y_{h_1}\ldots y_{h_m}$ for any $y=(y_1,\ldots ,y_d)^\top \in \R^d$ and $h=(h_1,\ldots ,h_m)^\top \in P$. As $X_{ij}$ have  $\psi_\beta$-norms uniformly bounded by $B_n$ we obtain by standard calculations and (W) that 
	\begin{align*}
		\underset{h \in P}{\max}\frac{1}{n}\sum_{i=1}^{n}\E[(X_i^h-\E[X_i^h])^2] \leq \underset{h \in P}{\max}\frac{1}{n}\sum_{i=1}^{n}\E[(X_i^h)^2]\leq C_1^2B_n^{2(m-1)}\log(dn)^{2(m-2)/\beta\lor 0},
	\end{align*}
	where the constant $C_1$ depends only on  $\beta$ and $D$.
	By Lemma  \ref{ProductNorm} and Lemma  \ref{MaxBound} in Section \ref{Weibull}
	we obtain that 
	\begin{align*}
		\E\left[\underset{1 \leq i \leq n, h \in P}{\max}(X^h_i-\E[X^h_i])^2\right] \leq C_2^2B_n^{2m}(\log(nd))^{2m/\beta}~,
	\end{align*}
	where the constant  $C_2$ depends only on $\beta, m$ and $K$.
Therefore, it follows from  Lemma~\ref{MaxIneqCentSum} that 
	\begin{align*}
		&\E\left[\underset{h \in P}{\max}\left|\frac{1}{\sqrt{n}}\sum_{i=1}^{n}(X^h_i-\E[X^h_i])\right|\right]\\
		&\leq L\left(C_1B_n^{m-1}\log(nd)^{((m-2)/\beta+1/2)\lor 1/2}+\frac{C_2(\log(nd))^{m/\beta+1}}{\sqrt{n}}\right)\\
		&\leq L(C_1+C_2)B_n^{m-1}(\log(nd))^{(m-1)/\beta \lor 1/2}~.
	\end{align*}
	Now applying Lemma \ref{MaxIneqCentSum} with $t=3L(C_1+C_2)B_n^{m-1}(\log(nd))^{(m-1)/\beta \lor 1/2}$, $\nu=1$ and  $\beta=\frac{2m}{\beta}$ we obtain  for $n\geq n_0$ (where $n_0$ depends only on $\beta,m,K$ and $L$) that 
	\begin{align*}
		\underset{h \in P}{\max}\left|\frac{1}{\sqrt{n}}\sum_{i=1}^{n}(X^h_i-\E[X^h_i])\right|>5L(C_1+C_2)B_n^{m-1}(\log(nd))^{(m-1)/\beta \lor 1/2}
	\end{align*}
	with probability at most
	\begin{align*}
	\frac{1}{(dn)^3}
		+3\exp\left(-C\left(\frac{\sqrt{n}B_n^{m-1}(\log(nd))^{(m-1)/\beta \lor 1/2}}{B_n^m(\log(nd))^{2m/\beta}}\right)^{\beta/(2m)}\right)\leq \frac{1}{(nd)^{3}}+\frac{3}{n^3}\leq \frac{1}{n}~.
	\end{align*}
Here $C$ is a constant, which  depends only on $\beta,m,K$ and $L$ and we 
have used that $B_n^2(\log(nd))^{4+2/\beta}\leq n$ for the first inequality.

\end{proof}

We note that, using Lemma \ref{MaxBound}, Assumptions (A) and (W) imply  Conditions A,B, P and V and therefore  Theorem \ref{Lindeberg} is applicable in the following discussion.  We start with a preliminary result regarding the quantities
\begin{equation}
    \label{det31}
   T_n=\frac{1}{\sqrt{n}} {\sum_{i=1}^{n}}  X_i
  \quad \text{ and } \quad
    T^*_n=\frac{1}{\sqrt{n}}\sum_{i=1}^{n}e_k(X_i-\bar{X})\,. 
  \end{equation}
Here $e_1,\ldots ,e_n$, which we will sometimes call multipliers, are  independent random variables (independent of $X=(X_1,\dots, X_n)$) such that $e_i=e_{i,1}+e_{i,2}$,  where $e_{i,1}$ and $e_{i,2}$ are independent,
$e_{i,1} \sim N(0,\sigma^2)$   and  $e_{i,2}$   has a two point distribution
with $\E[e_i]=0$, $\E[e_i^2]=\E[e_i^3]=1$ (see Lemma 7.3 in \citeSM{ImprovedApprox} for more details and note that $\sigma$ can be chosen universally). 

\begin{Lemma}\label{GaussApprox}  
	Suppose that Conditions (A) and (W) hold,
	then, with probability at least $1-2/n$, we have
	\begin{align}
	\label{det32}
		\underset{x \in \R}{\sup}\ \big |  \p(T_n\leq x)-\p(T^*_n\leq x|X) \big |
		\leq C\left(\frac{B_n^2(\log(nd))^{4+2/\beta}}{n}\right)^{1/4}~,
	\end{align}	
	where the constant $C$ only depends on $\sigma_{min}$, $\beta$. \\
	Further, let $T_n^G$ denote the analogue of the statistic $T_n$ in \eqref{det31},
	where the random variables $X_k$ have been replaced 
	by independent zero mean Gaussian vectors with the same covariance structure. Then
	\begin{align}
		\label{det33}
		\underset{x \in \R}{\sup}\big |  \p(T_n\leq x)-\p(T^G_n\leq x) \big |\leq C\left(\frac{B_n^2(\log(nd))^{4+2/\beta}}{n}\right)^{1/4}\,.	
	\end{align}
\end{Lemma}
\begin{proof}
First assume that  $Y_1, \ldots ,Y_n$ are  vectors in $\R^d$ such that 
	\begin{align*}
		\underset{1\leq k \leq d}{\max}\norm{Y_k}_\infty&\leq KB_n\left(5\log(dn)\right)^{1/\beta}\,,\\
		\frac{\sigma_{min}}{2}\leq \frac{1}{n}\sum_{i=1}^{n}Y^2_{ik} &\leq D \,,\\
		\frac{1}{n}\sum_{i=1}^{n}Y^4_{ik} &\leq B_n^2D\,,\\	
		\underset{1\leq k,j\leq d}{\max}\left|\frac{1}{\sqrt{n}}\sum_{i=1}^{n}(Y_{ik}Y_{ij}-E[X_{ik}X_{ij}])\right| &\leq C_mB_n(\log(dn))^{1/\beta}\,,\\
		\underset{1\leq k,j,l\leq d}{\max}\left|\frac{1}{\sqrt{n}}\sum_{i=1}^{n}(Y_{ik}Y_{ij}Y_{il}-E[X_{ik}X_{ij}X_{il}])\right|& \leq C_mB_n^2(\log(dn))^{2/\beta}\,.
	\end{align*}
	Recall the definition of the multipliers 
	in the paragraph following equation
	\eqref{det31}. 
	 We will apply Theorem \ref{Lindeberg} with $Z_i=e_iY_i$ and $V_i=X_i$. The Conditions V, P and B follow immediately from the properties of $Y$ with $C_v, C_p, C_b$ only depending on $\sigma_{min}$ and $B_n$.  Condition $A$ follows from the Gaussianity of $e_{i,1}$ with $C_a$ depending only on $\sigma_{min}$ and $\sigma$ (first condition on $e_{i,2}$ and then use Lemma 8.3 from \citeSM{ImprovedApprox}). The remaining conditions  in Theorem \ref{Lindeberg} follow easily from the properties of $e_i$ and $Y_i$. We hence obtain
	{\small \begin{align*}
		\sup_{y\in \R}\left|\p\left(\frac{1}{\sqrt{n}}\sum_{i=1}^{n}X_i\leq y\right)-\p\left(\frac{1}{\sqrt{n}}\sum_{i=1}^{n}e_iY_i\leq y\right)\right|\leq K_2\left(\frac{B_n^2(\log(nd))^{4+2/\beta}}{n}\right)^{1/4} ~,
	\end{align*}}
where $K_2$ depends only on $\sigma_{min}$. By Lemma \ref{MaxBound} and  Lemma \ref{MomentBound} the random 
	vectors  $X_{i}-\bar{X}$ satisfy 
	the assumptions stated for the vectors $Y_{i}$ 
	(with probability close to $1$), and we  obtain
	\begin{align*}
		\underset{x \in \R}{\sup}|\p(T_n\leq x)-\p(T^*_n\leq x|X)|\leq K_2\left(\frac{B_n^2(\log(nd))^{4+2/\beta}}{n}\right)^{1/4}
	\end{align*}
	 with probability at least $1-2/n$, establishing \eqref{det32}. 
 For the second inequality \eqref{det33} we define
 \begin{align*}
     R^*_n:=\frac{1}{\sqrt{n}}\sum_{i=1}^{n}\tilde{e}_i(X_i-\bar{X})\,,
 \end{align*}
 where the multipliers $\tilde{e}_1, \ldots ,\tilde{e}_n$ are now chosen as in Corollary 5.2 of \citeSM{ImprovedApprox},
 with $v=0, \alpha = 1/2$ and $\beta = 3/2$. More precisely, we sample $\tilde{e}_i$ independently from the distribution that is given by $4\nu-1$ where $\nu\sim \text{Beta}(1/2,3/2)$. Note that $\E[\tilde{e}_i]=0$ and $\E[\tilde{e}_i^2]=1$. We then obtain by similar arguments as above that 
 \begin{align*}
		\underset{x \in \R}{\sup}| \p(T_n\leq x)-\p(R^*_n\leq x|X)|\leq C\left(\frac{B_n^2(\log(nd))^{4+2/\beta}}{n}\right)^{1/4}	
\end{align*}
with probability at least $1-2/n$. We let $A_n$ be the event that the first three inequalities in Lemma \ref{MomentBound} hold. Then $\p(A_n)\geq 1-1/n$. On $A_n$ we may apply first Corollary 5.2 of \citeSM{ImprovedApprox} which gives
\begin{align*}
    \underset{x \in \R}{\sup}|\p(T_n^{\tilde{G}} \leq x|X)-\p(R^*_n\leq x|X)|\leq C\left(\frac{B_n^2(\log(nd))^{5}}{n}\right)^{1/4}\,,	
\end{align*}
where $T_n^{\tilde{G}}$ is defined analogously to $T_n^G$ for a certain Gaussian process $\tilde{G}$. The Gaussian to Gaussian comparison from Corollary 5.1 of \citeSM{ImprovedApprox} then yields that on $A_n$ we further have
\begin{align*}
     \underset{x \in \R}{\sup}| \p(T_n^{\tilde{G}} \leq x|X)-\p(T_n^G \leq x)|\leq C\left(\frac{B_n^2(\log(nd))^{4+2/\beta}}{n}\right)^{1/4}	
\end{align*}
as we can bound $ \| \text{vech}(\Sigma_G) - \text{vech}(\Sigma_{\tilde{G}})\|_{\infty}$ by $B_n\log(dn)^{1/\beta}$ due to \eqref{det43}.
\end{proof}

\begin{Lemma}
	\label{AntiConcentrationSampleMean}
	Suppose that Conditions (A) and (W) hold. Then for any $x\in \R$ and $t>0$ we have
	\begin{align*}
		\p(T_n\leq x+t)-\p(T_n\leq x)\leq C\left(t\sqrt{\log d}+\left(\frac{B_n^2(\log(nd))^{4+2/\beta}}{n}\right)^{1/4}\right)\,.
	\end{align*}
\end{Lemma}

\begin{proof}
 For some constant $C$ only depending on $\sigma_{min}$ and $B_n$ we get
	\begin{align*}
		&\p(T_n\leq x+t)-\p(T_n\leq x)\le \left| \p(T_n\leq x+t)-\p(T_n^G\leq x+t)\right|\\
		&\quad + \left| \p(T_n\leq x)-\p(T_n^G\leq x)\right| +\left| \p(T_n^G\leq x+t)-\p(T_n^G\leq x)\right|\\
		&\le 2C\left(\frac{B_n^2(\log(nd))^{4+2/\beta}}{n}\right)^{1/4}+Ct\sqrt{\log d}\,,
	\end{align*}
where for the last line we used \eqref{det33} and the Gaussian anti-concentration property from \citeSM[Lemma~8.3]{ImprovedApprox}.
\end{proof}

\subsubsection{Proof of Theorem \ref{Lindeberg}} \label{sec511}

We will establish Theorem  \ref{Lindeberg} via the Iterative Randomized Lindeberg Method. The proof is structurally the same as in \citeSM{ImprovedApprox} but asks for weaker decay in the tails at the cost of a weaker bound. We begin by introducing some notation which we will be used in this section.

For $\epsilon \in \{0,1\}^n$ we set
\begin{align*}
	S_{n,\epsilon}^V=\frac{1}{\sqrt{n}}\sum_{i=1}^{n}(\epsilon_iV_i+(1-\epsilon_i)Z_i) \quad \text{ and } \quad S_{n}^Z=\frac{1}{\sqrt{n}}\sum_{i=1}^{n}Z_i\,.
\end{align*}
Let $\epsilon^0=(1, \ldots ,1)$, $D=[\log(n)]+1$ and define random vectors $\epsilon^1, \ldots ,\epsilon^D \in \{0,1\}^n$  such that
i) $\epsilon_i^s=0$ if $\epsilon_i^{s-1}=0$ and ii) for $I_{s-1}=\{i=1, \ldots ,n:\epsilon_i^{s-1}=1\}$, the random variables $\{\epsilon_i^s\}_{i \in I_{s-1}}$ are exchangeable conditional on $\epsilon^{s-1}$ and satisfy 
	\begin{align*}
		\p\left(\sum_{i \in I_{s-1}}\epsilon_i^s=k\, |\, \epsilon^{s-1}  \right)=\frac{1}{1+|I_{s-1}|}\,, \qquad k=0,\ldots,|I_{s-1}|\,.
	\end{align*}
As remarked in \citeSM{ImprovedApprox}, these properties uniquely determine the joint distribution of $\epsilon^1, \ldots ,\epsilon^D$ which we also assume independent of $V_1, \ldots ,V_n, Z_1, \ldots ,Z_n$. For positive constants $B_{n,1,s},B_{n,2,s}$ and 
\begin{align*}
	\mathcal{E}_{i,jk}^V&=\E[V_{ij}V_{ik}]\,, \quad ~ \mathcal{E}_{i,jkl}^V=\E[V_{ij}V_{ik}V_{il}]\,,\\
	\mathcal{E}_{i,jk}^Z&=\E[Z_{ij}Z_{ik}]\,, \quad \mathcal{E}_{i,jkl}^Z=\E[Z_{ij}Z_{ik}Z_{il}]\,,
\end{align*}
we denote by $\mathcal{A}_s$ the event
{\small\begin{align*}
	\left\{\underset{1\leq j,k\leq d}{\max}\Big|\frac{1}{\sqrt{n}}\sum_{i=1}^{n}\epsilon^s_i(\mathcal{E}^V_{i,jk}-\mathcal{E}^Z_{i,jk})\Big|\leq B_{n,1,s}\,, 
	\underset{1\leq j,k,l\leq d}{\max}\Big|\frac{1}{\sqrt{n}}\sum_{i=1}^{n}\epsilon^s_i(\mathcal{E}^V_{i,jkl}-\mathcal{E}^Z_{i,jkl})\Big|\leq B_{n,2,s}\right\}.
\end{align*}}

We also fix a five times continuously differentiable and decreasing function $g_0:\R\rightarrow \R$ such that
$$\text{i) } g_0(t)\geq 0\,, \quad
	\text{ii) } g_0(t)=0 \text{ when } t\geq 1\,, \quad\text{ and }
	\text{iii) } g_0(t)=1 \text{ when } t\leq 0\,.$$

Clearly we can bound the first five derivatives of this function uniformly by some constant $C_g$. The bounds in the following proofs and results depend on the particular choice of $g_0$, but as we may choose some $g_0$ that works universally we suppress that dependence.\\

Next we let $\phi>0, \beta=\phi\log d$, set $g(t)=g_0(\phi t)$ and define for $\omega \in \R^d$ the softmax function 
\begin{align*}
	F(w)=\beta^{-1}\log\left(\sum_{j=1}^{d}\exp(\beta w_j)\right)\,.
\end{align*}
It is easy to check that 
\begin{align*}
	g(t)=\begin{cases}
		1 \quad \text{if} \ t\leq 0\\
		0 \quad \text{if} \ t\geq \phi^{-1}
	\end{cases}
\end{align*}
and	$\max_{1 \leq j \leq d} w_j \leq F(w) \leq \max_{1 \leq j \leq d} w_j+\phi^{-1}$.
For $y\in \R^d$ we now define the function
\begin{align*}
	m^y(w)=g(F(w-y))\,, \qquad w\in \R^d\,,
\end{align*}
and its partial derivatives up to fifth order, for instance we write
\begin{align*}
	m^y_{jklrh}(w)=\frac{\partial^5 m^y(w)}{\partial w_j\partial w_k\partial w_l\partial w_r\partial w_h}\,, \qquad j,k,l,r,h=1,\ldots,d\,.
\end{align*}
From \citeSM{ImprovedApprox} we know that there exist functions $U^y_{jk}, U^y_{jkl}, U^y_{jklr}, U^y_{jklrh}:\R^d\rightarrow\R$ with the following 3 properties.
\begin{description}
	\item \textbf{i)} $|m^y_{I}(w)| \leq U^y_{I}(w)$ where $I$ is any of the index sets $jk, jkl, jklr$ or $jklrh$.
	\item \textbf{ii)} For any $w_1,w_2 \in \R^d$ such that $\beta\norm{w_2}_\infty \leq 1$ we have 
	\begin{align}
		\label{mSumBound}
		U^y_{jklr}(w_1+w_2)\lesssim U^y_{jklr}(w_1), \quad U^y_{jklrh}(w_1+w_2)\lesssim U^y_{jklrh}(w_{1})\,.
	\end{align}
	\item \textbf{iii)} For the same $I$ as in i) we have uniformly in $w$
	\begin{align} 
		\label{x1}
		\sum_{I} U^y_{I}(w) \lesssim \phi^{|I|}
		(\log d)^{|I|-1} \,.
	\end{align}
\end{description}

Lastly we define 
\begin{align*}
	&\mathcal{I}^y:=m^y(S^V_{n,\epsilon^s})-m^y(S^Z_n)\,,\\
	&h^y(Y;x):=\1\{-x < \underset{1 \leq j \leq d}{\max}(Y_j-y_j)\leq x\}\,, \qquad x>0\,,\\
	&\varrho_\epsilon:=\underset{y \in \R^d}{\sup}|\p(S^V_{n,\epsilon}\leq y)-\p(S^Z_n\leq y)|\,.
\end{align*}
We now state and prove three auxiliary results, which be essential for the proof of Theorem \ref{Lindeberg}.

\begin{Lemma}
	\label{LindebergLemma}
	Suppose that conditions V,P,B and A are satisfied. Then for any $d=0, \ldots ,D-1$ and any $\phi>0$ such that 
	\begin{align}
		\label{mSumBound2}
		C_pB_n \phi (\log(dn))^{1+1/\beta}\leq \sqrt{n}
	\end{align}
	on the event $\mathcal{A}_s$, we have
	\begin{align*}
		\varrho_{\epsilon^s}\lesssim &
		\frac{\sqrt{\log d}}{\phi}+\frac{B_n^2\phi^4(\log(dn))^{3+2/\beta}}{n^2}\\&+\left(\frac{\sqrt{\log d}}{\phi}+
		\E[\varrho_{\epsilon^{s+1}}|\epsilon^s]\right) 
		\left(\frac{B_{n,1,s}\phi^2\log d}{\sqrt{n}}+\frac{B_{n,2,s}\phi^3 (\log d)^2}{n}+\frac{B_n^2\phi^4(\log d)^3}{n}\right)
	\end{align*}
	up to a constant depending only on $C_v, C_p, C_b, C_a$.
\end{Lemma}

\begin{proof}
	Fix $s=0, \ldots ,D-1$ and $e^s\in \{0,1\}^n$ such that if $\epsilon^s=e^s$, then $\mathcal{A}_s$ holds. All following arguments will be conditional on $\epsilon^s=e^s$, for the sake of brevity we will make this conditioning implicit and  write $\p(\cdot)$ and $\E[\cdot]$ instead of $\p(\cdot|\epsilon^s=e^s)$ and $\E[\cdot |\epsilon^s=e^s]$.
	We denote 
	\begin{align*}
		W=(W_1, \ldots ,W_d)^\top =\frac{1}{\sqrt{n}}\sum_{i=1}^{n}\left(\epsilon_i^{s+1}V_i+(1-\epsilon_i^{s+1})Z_i\right)\,.
	\end{align*}
	We will split the proof into two steps and three auxiliary calculations where we prove bounds that are used in the first two steps.
In the first step we establish the bound
	\begin{align}
		\label{Bound1}
		\underset{y \in \R^d}{\sup}&|\E[\mathcal{I}^y]|\lesssim  \frac{B_n^2\phi^4(\log(dn))^{3+2/\beta}}{n^2}\\&+\left(\frac{\sqrt{\log d}}{\phi}+
		\E[\varrho_{\epsilon^{s+1}}|\epsilon^s] \right) 
		\left(\frac{B_{n,1,s}\phi^2\log d}{\sqrt{n}}+\frac{B_{2,n,s}\phi^3(\log d)^2}{n}+\frac{B_n^2\phi^4(\log d)^3}{n}\right)
	\end{align}
	and in the second step we show that 
	\begin{align}
		\label{Bound2}
		\varrho_{\epsilon^s}\lesssim \frac{\sqrt{\log d}}{\phi}+\underset{y \in \R^d}{\sup}|\E[\mathcal{I}^y]|
	\end{align}
	which then yields the desired claim.
	
	\textbf{Step 1.} 
	Let $\mathcal{S}_n$ be the set of permutations on $\{1, \ldots ,|I_s|\}$ and let $\sigma$ be a random variable that is distributed uniformly on $\mathcal{S}_n$ and also independent of $V_1, \ldots ,V_n,Z_1, \ldots ,Z_n$ and $\epsilon^{s+1}$.
	Writing
	\begin{align*}
		W^\sigma_i=\frac{1}{\sqrt{n}}\sum_{j=1}^{i-1}V_{\sigma(j)}+\frac{1}{\sqrt{n}}\sum_{j=i+1}^{|I_s|}Z_{\sigma(j)}+\frac{1}{\sqrt{n}}\sum_{j \notin I_s}Z_j, \quad \text{for all }i=1, \ldots ,|I_s|\,,
	\end{align*}
	it follows by Lemma \ref{Permutation} that for any function $m:\R^d\rightarrow\R$ and any $i \in I_s$, 
	\begin{align*}
		\E[m(W)]=\E\left[\frac{\sigma^{-1}(i)}{|I_s|+1}m\left(W^\sigma_{\sigma^{-1}(i)}+\frac{V_i}{\sqrt{n}}\right)+(1-\frac{\sigma^{-1}(i)}{|I_s|+1})m\left(W^\sigma_{\sigma^{-1}(i)}+\frac{Z_i}{\sqrt{n}}\right)\right]
	\end{align*}
	
	Fixing some $y\in \R^d$ we observe that
	\begin{align*}
		\mathcal{I}^y=\sum_{i=1}^{|I_s|} \left(m\left(W^\sigma_{i}+\frac{V_{\sigma(i)}}{\sqrt{n}}\right)  - m\left(W^\sigma_{i}+\frac{Z_{\sigma(i)}}{\sqrt{n}}\right)  \right)
	\end{align*}
	and let 
	\begin{align*}
		f(t)=\sum_{i=1}^{|I_s|}\E \left(m\left(W^\sigma_{i}+\frac{tV_{\sigma(i)}}{\sqrt{n}}\right)  - m\left(W^\sigma_{i}+\frac{tZ_{\sigma(i)}}{\sqrt{n}}\right)  \right), \quad \text{for } t \in [0,1].
	\end{align*}
	Clearly $\E[\mathcal{I}^y]=f(1)$ and by Taylor's expansion
	\begin{align*}
		f(1)=f(0)+f^{(1)}(0)+\frac{f^{(2)}(0)}{2}+\frac{f^{(3)}(0)}{6}+\frac{f^{(4)}(\bar{t})}{24}
	\end{align*}
	for some $\bar{t} \in (0,1)$. Clearly $f(0)=0$ and because of $\E[V_{ij}]=\E[Z_{ij}]=0$ we also obtain $f^{(1)}(0)=0$.
	
	We defer the bounds of $|f^{(2)}(0)|,|f^{(3)}(0)|$ and $|f^{(4)}(\bar{t})|$ to the three auxiliary calculations.\\
	\smallskip

\textbf{Step 2.}
We observe that
\begin{align*}
	\p(S^V_{n,\epsilon^s}\leq y)&\leq \p(F(S^V_{n,\epsilon^s}-y-\phi^{-1})\leq 0)\leq \E[m^{y+\phi^{-1}}(S^V_{n,\epsilon^s})]\\
	&\leq \E[m^{y+\phi^{-1}}(S^Z_n)]+|E[\mathcal{I}^{y+\phi^{-1}}]|\leq \p(S^Z_n\leq y+2\phi^{-1})+|E[\mathcal{I}^{y+\phi^{-1}}]|\\
	&\leq \p(S^Z_n\leq y)+2C_a\phi^{-1}\sqrt{\log d}+|E[\mathcal{I}^{y+\phi^{-1}}]|\,.
\end{align*}
Similarly, we obtain 
\begin{align*}
	\p(S^V_{n,\epsilon^s}\leq y)\geq \p(S^Z_n\leq y)-2C_a\phi^{-1}\sqrt{\log d}-|E[\mathcal{I}^{y+\phi^{-1}}]|\,.
\end{align*}
Combining these bounds yields \eqref{Bound2}.

\textbf{Auxiliary Calculation 1.}
We calculate a bound for $|f^{(2)}(0)|$ by utilizing the representation
\begin{align*}
	f^{(2)}(0)&=\frac{1}{n}\sum_{i=1}^{|I_s|}\sum_{j,k=1}^{d}\E[m^y_{jk}(W^\sigma_i)(V_{\sigma(i)j}V_{\sigma(i)k}-Z_{\sigma(i)j}Z_{\sigma(i)k})]\\
	&=\frac{1}{n}\sum_{i \in I_s}\sum_{j,k=1}^{d}\E[m^y_{jk}(W^\sigma_{\sigma^{-1}(i)})(V_{ij}V_{ik}-Z_{ij}Z_{ik})]\\
	&=\frac{1}{n}\sum_{i \in I_s}\sum_{j,k=1}^{d}\E[m^y_{jk}(W^\sigma_{\sigma^{-1}(i)})](\mathcal{E}^V_{i,jk}-\mathcal{E}^Z_{i,jk})~,
\end{align*}
where we used the independence of $W^\sigma_{\sigma^{-1}(i)}$ and $V_{ij}V_{ik}-Z_{ij}Z_{ik}$ when conditioning on $\sigma$ in the third line. Denoting
{\small\begin{align*}
	R^\sigma_{i,jk}=m^y_{jk}(W^\sigma_{\sigma^{-1}(i)})-\frac{\sigma^{-1}(i)}{|I_s|+1}m^y_{jk}\left(W^\sigma_{\sigma^{-1}(i)}+\frac{V_i}{\sqrt{n}}\right)-(1-\frac{\sigma^{-1}(i)}{|I_s|+1})m^y_{jk}\left(W^\sigma_{\sigma^{-1}(i)}+\frac{Z_i}{\sqrt{n}}\right)
\end{align*}}
we obtain the decomposition $f^{(2)}(0)=\mathcal{I}_{2,1}+\mathcal{I}_{2,2}$,  where
\begin{align*}
	\mathcal{I}_{2,1}&=\frac{1}{n}\sum_{i \in I_s}\sum_{j,k=1}^{d}\E[m^y_{jk}(W)](\mathcal{E}^V_{i,jk}-\mathcal{E}^Z_{i,jk})\,,\\
	\mathcal{I}_{2,2}&=\frac{1}{n}\sum_{i \in I_s}\sum_{j,k=1}^{d}\E[R^\sigma_{i,jk}](\mathcal{E}^V_{i,jk}-\mathcal{E}^Z_{i,jk})\,.
\end{align*}

We first bound $\mathcal{I}_{2,1}$ by
\begin{align*}
	|\mathcal{I}_{2,1}|\leq \sum_{j,k=1}^{d}\E[|m^y_{jk}(W)|]\underset{1 \leq j,k\leq p}{\max}\left|\frac{1}{n}\sum_{i=1}^{n}\epsilon^s_i(\mathcal{E}^V_{i,jk}-\mathcal{E}^Z_{i,jk})\right|\leq \frac{B_{n,1,s}}{\sqrt{n}}\sum_{j,k=1}^{d}\E[|m^y_{jk}(W)|]\,.
\end{align*}
Recalling the definition of $m^y$ and $h^y$, we see that $m^y_{jk}(W)=h^y(W;\phi^{-1})m^y_{jk}(W)$. Thus, since 
\begin{align}
	\label{x2}
	\mathcal{P}:=\p\left(-\phi^{-1}\leq \underset{1\leq j \leq p}{\max}\frac{1}{\sqrt{n}}\sum_{i=1}^{n}(Z_{ij}-y_j)\leq \phi^{-1}\right)\leq \frac{2C_a\sqrt{\log d}}{\phi}
\end{align}
by Condition A, the basic properties of $U_{jk}$ and the definitions of the quantities involved imply (note that $m^y_{jk}\leq U_{jk}$ and 
\eqref{x1}, \eqref{x2})

\begin{align*}
	\sum_{j,k=1}^{d}\E[|m^y_{jk}(W)|]=&\sum_{j,k=1}^{d}\E[|h^y(W;\phi^{-1})m^y_{jk}(W)|] \\
	&\le \sum_{j,k=1}^{d}\E[|h^y(W;\phi^{-1})U_{jk}(W)|]\\
	& \lesssim \phi^2\log d\, \p\left(-\phi^{-1}<\underset{1\leq j \leq d}{\max}(W_j-y_j)\leq \phi^{-1}\right) \\
	& \leq \phi^2\log d \,(2\E[\varrho_{\epsilon^{s+1}}]+\mathcal{P}) \\ &
	\lesssim \phi^2\log d \,\left(\E[\varrho_{\epsilon^{s+1}}]+\frac{\sqrt{\log d}}{\phi}\right)\,,
	\tag{\theequation}\label{m^yBound}
\end{align*}
and therefore,
\begin{align*}
	|\mathcal{I}_{2,1}|\lesssim \frac{B_{n,1,d}\phi^2\log d}{\sqrt{n}}\left(\E[\varrho_{\epsilon^{s+1}}]+\frac{\sqrt{\log d}}{\phi} \right)\,.
\end{align*}

To bound $\mathcal{I}_{2,2}$ we use the same Taylor expansion as above and get
\begin{align*}
	|\E[R^\sigma_{i,jk}]|\leq \sum_{l,r=1}^{d}\E\left[m^y_{jklr}\left(W^\sigma_{\sigma^{-1}(i)}+\frac{\bar{t}V_i}{\sqrt{n}}\right)\frac{V_{il}V_{ir}}{n}\right]\\+ \sum_{l,r=1}^{d}\E\left[m^y_{jklr}\left(W^\sigma_{\sigma^{-1}(i)}+\frac{\bar{t}Z_i}{\sqrt{n}}\right)\frac{Z_{il}Z_{ir}}{n}\right]\,,
\end{align*}
which yields $|\mathcal{I}_{2,2}|\leq \mathcal{I}_{2,2,1}+\mathcal{I}_{2,2,2}$, where
\begin{align*}
	\mathcal{I}_{2,2,1}=\frac{1}{n^2}\sum_{i \in I_s}\sum_{j,k,l,r=1}^{n}\E\left[\left| m^y_{jklr}\left(W^\sigma_{\sigma^{-1}(i)}+\frac{\bar{t}V_i}{\sqrt{n}}\right)V_{il}V_{ir}\right|\right]|\mathcal{E}^V_{i,jk}-\mathcal{E}^Z_{i,jk}|\,,\\
	\mathcal{I}_{2,2,2}=\frac{1}{n^2}\sum_{i \in I_s}\sum_{j,k,l,r=1}^{n}\E\left[\left| m^y_{jklr}\left(W^\sigma_{\sigma^{-1}(i)}+\frac{\bar{t}Z_i}{\sqrt{n}}\right)Z_{il}Z_{ir}\right|\right]|\mathcal{E}^V_{i,jk}-\mathcal{E}^Z_{i,jk}|\,.
\end{align*}
Next, we will bound $\mathcal{I}_{2,2,1}$. Setting $x=C_pB_n(\log(dn))^{1/\beta}/\sqrt{n}+\phi^{-1}$ and $\bar{V}_i=\1\{\norm{V_i}_\infty \leq C_pB_n(\log(dn))^{1/\beta}\}$, we have
\stepcounter{equation}
\begin{align*}
	&\sum_{l,r=1}^{d}\E\left[\bar{V}_i\left|m^y_{jklr}\left(W^\sigma_{\sigma^{-1}(i)}+\frac{\bar{t}V_i}{\sqrt{n}}\right)V_{il}V_{ir}\right|\right]\\
	&=\sum_{l,r=1}^{d}\E\left[\bar{V}_i h^y\left(W^\sigma_{\sigma^{-1}(i)};x\right)\left|m^y_{jklr}\left(W^\sigma_{\sigma^{-1}(i)}+\frac{\bar{t}V_i}{\sqrt{n}}\right)V_{il}V_{ir}\right|\right]\\
	&\leq \sum_{l,r=1}^{d}\E\left[\bar{V}_i h^y\left(W^\sigma_{\sigma^{-1}(i)};x\right)U^y_{jklr}\left(W^\sigma_{\sigma^{-1}(i)}+\frac{\bar{t}V_i}{\sqrt{n}}\right)|V_{il}V_{ir}|\right] \\
	&\lesssim  \sum_{l,r=1}^{d}\E\left[\bar{V}_i h^y\left(W^\sigma_{\sigma^{-1}(i)};x\right)U^y_{jklr}\left(W^\sigma_{\sigma^{-1}(i)}\right)|V_{il}V_{ir}|\right]
	\tag{\theequation}\label{V-h-m-trick} ~,
\end{align*}
where the first equality follows by the definitions of the involved quantities and the second inequality follows by \eqref{mSumBound}. Setting $\bar{Z}_i=\1\{\norm{Z_i}_\infty \leq C_pB_n(\log(dn))^{1/\beta}\}$ we bound the above expectation by
\stepcounter{equation}
\begin{align*}
	&\E\left[\bar{V}_i h^y\left(W^\sigma_{\sigma^{-1}(i)};x\right)U^y_{jklr}\left(W^\sigma_{\sigma^{-1}(i)}\right)|V_{il}V_{ir}|\right] \\
	&\lesssim \E\left[\bar{V}_i\bar{Z}_i h^y\left(W^\sigma_{\sigma^{-1}(i)};x\right)U^y_{jklr}\left(W^\sigma_{\sigma^{-1}(i)}\right)\right]\E[|V_{il}V_{ir}|]\\
	&\lesssim \E\left[h^y\left(W;2x\right)U^y_{jklr}\left(W\right)\right]\E[|V_{il}V_{ir}|]\,,
	\tag{\theequation}\label{2x-trick} 
\end{align*}

where the inequalities follow by Condition P, the definitions of $h^y$, $W$ and $W^\sigma_{\sigma^{-1}(i)}$ and \eqref{mSumBound} as well as \eqref{mSumBound2}. 

Hence we obtain, by the same arguments as for \eqref{m^yBound},
\begin{align*}
	&\frac{1}{n^2}\sum_{i \in I_s}\sum_{j,k,l,r=1}^{d}\E\left[\bar{V}_i\left|m^y_{jklr}\left(W^\sigma_{\sigma^{-1}(i)}+\frac{\bar{t}V_i}{\sqrt{n}}\right)V_{il}V_{ir}\right|\right]|\mathcal{E}^V_{i,jk}-\mathcal{E}^Z_{i,jk}|\\
	&\lesssim \frac{1}{n^2}\sum_{i \in I_s}\sum_{j,k,l,r=1}^{d} \E\left[h^y\left(W;2x\right)U^y_{jklr}\left(W\right)\right]\E[|V_{il}V_{ir}|]|\mathcal{E}^V_{i,jk}-\mathcal{E}^Z_{i,jk}|\\
	&\lesssim \frac{B_n^2}{n}\sum_{j,k,l,r=1}^{d}\E\left[h^y\left(W;2x\right)U^y_{jklr}\left(W\right)\right]\lesssim \frac{B_n^2\phi^4(\log d)^3}{n}\left(\E[\varrho_{\epsilon^{s+1}}]+\frac{\sqrt{\log d}}{\phi}\right)\,,
\end{align*}
where we used that by Condition V
\begin{align*}
	\underset{1 \leq j,j,l,r\leq d}{\max}\sum_{i=1}^{n}\E[|V_{il}V_{ir}|]|\mathcal{E}^V_{i,jk}-\mathcal{E}^Z_{i,jk}|
	&\lesssim \underset{1 \leq j,j,l,r\leq d}{\max}\sum_{i=1}^{n} \E[|V_{il}V_{ir}|^2]|+|\mathcal{E}^V_{i,jk}-\mathcal{E}^Z_{i,jk}|^2\\
	&\lesssim B_n^2n\,.
\end{align*}
Additionally we have
\begin{align*}
	&\frac{1}{n^2}\sum_{i \in I_s}\sum_{j,k,l,r=1}^{d}\E\left[(1-\bar{V}_i)\left|m^y_{jklr}\left(W^\sigma_{\sigma^{-1}(i)}+\frac{\bar{t}V_i}{\sqrt{n}}\right)V_{il}V_{ir}\right|\right]|\mathcal{E}^V_{i,jk}-\mathcal{E}^Z_{i,jk}|\\
	&\lesssim \frac{\phi^4 (\log d)^3}{n}\sum_{i=1}^{n}\E[(1-\bar{V}_i)\norm{V_i}_\infty^2]\\
	&\lesssim \frac{B_n^2\phi^4(\log(nd))^{3+2/\beta}}{n^2}\,,
\end{align*}
where the first inequality follows by $m^y_I \leq U^y_I$ for appropriate index sets $I$ as well as Condition V, and the second inequality follows from Hölder's inequality and Condition P as well as B.

Combining these two inequalities we obtain
\begin{align*}
	\mathcal{I}_{2,2,1} \lesssim \frac{B_n^2\phi^4(\log d)^3}{n}\left(\E[\varrho_{\epsilon^{s+1}}]+\frac{\sqrt{\log d}}{\phi}\right)+\frac{B_n^2\phi^4(\log(nd))^{3+2/\beta}}{n^2} ~.
\end{align*}
Similar arguments as for $\mathcal{I}_{2,2,1}$ also establish the same bound for $\mathcal{I}_{2,2,2}$ and therefore we conclude 
{\small\begin{align*}
	|f^{(2)}(0)|\lesssim \left( \frac{B_{n,1,s}\phi^2\log d}{\sqrt{n}} + \frac{B_n^2\phi^4(\log d)^3}{n}\right) \left(\E[\varrho_{\epsilon^{s+1}}]+\frac{\sqrt{\log d}}{\phi}\right)+\frac{B_n^2\phi^4(\log(nd))^{3+2/\beta}}{n^2}\,.
\end{align*}}

\textbf{Auxiliary Calculation 2.}

Just as in the beginning of the previous calculations we obtain
\begin{align*}
	f^{(3)}(0)=\frac{1}{n^{3/2}}\sum_{i \in I_s}\sum_{j,k,l=1}^{d}\E[m^y_{jkl}(W^\sigma_{\sigma^{-1}(i)})](\mathcal{E}^V_{i,jkl}-\mathcal{E}^Z_{i,jkl})~.
\end{align*}
Writing
\begin{align*}
	R^\sigma_{i,jkl}&=m^y_{jkl}(W^\sigma_{\sigma^{-1}(i)})-\frac{\sigma^{-1}(i)}{|I_s|+1}m^y_{jkl}\left(W^\sigma_{\sigma^{-1}(i)}+\frac{V_i}{\sqrt{n}}\right)\\
	&\quad -\Big(1-\frac{\sigma^{-1}(i)}{|I_s|+1}\Big)m^y_{jkl}\left(W^\sigma_{\sigma^{-1}(i)}+\frac{Z_i}{\sqrt{n}}\right)
\end{align*}
we have (just as above for $f^{(2)}(0)$) that $f^{(3)}(0)=\mathcal{I}_{3,1}+\mathcal{I}_{3,2}$, where
\begin{align*}
	\mathcal{I}_{3,1}=\frac{1}{n^{3/2}}\sum_{i \in I_s}\sum_{j,k,l=1}^{d}\E[m^y_{jkl}(W)](\mathcal{E}^V_{i,jkl}-\mathcal{E}^Z_{i,jkl})\,,\\
	\mathcal{I}_{3,2}=\frac{1}{n^{3/2}}\sum_{i \in I_s}\sum_{j,k,l=1}^{d}\E[R^\sigma_{i,jkl}(W)](\mathcal{E}^V_{i,jkl}-\mathcal{E}^Z_{i,jkl})\,.
\end{align*}
By the same arguments that we used to bound $\mathcal{I}_{2,1}$, we obtain
\begin{align*}
	\mathcal{I}_{3,1} &\lesssim \frac{B_{n,2,s}\phi^3(\log d)^2}{n} \left(\E[\varrho_{\epsilon^{s+1}}]+\frac{\sqrt{\log d}}{\phi}\right)\,.
\end{align*}

We also get by the same arguments as before that $|\mathcal{I}_{3,2}|\leq \mathcal{I}_{3,2,1}+\mathcal{I}_{3,2,2}$ where
\begin{align*}
	\mathcal{I}_{3,2,1}=\frac{1}{n^{5/2}}\sum_{i \in I_s}\sum_{j,k,l,r,h=1}^{d}\E\left[\left| m^y_{jklrh}\left(W^\sigma_{\sigma^{-1}(i)}+\frac{\bar{t}V_i}{\sqrt{n}}\right)V_{ir}V_{ih}\right|\right]|\mathcal{E}^V_{i,jkl}-\mathcal{E}^Z_{i,jkl}|\,,\\
	\mathcal{I}_{3,2,2}=\frac{1}{n^{5/2}}\sum_{i \in I_s}\sum_{j,k,l,r,h=1}^{d}\E\left[\left| m^y_{jklrh}\left(W^\sigma_{\sigma^{-1}(i)}+\frac{\bar{t}Z_i}{\sqrt{n}}\right)Z_{ir}Z_{ih}\right|\right]|\mathcal{E}^V_{i,jkl}-\mathcal{E}^Z_{i,jkl}|\,.
\end{align*}

Moreover, we have
\begin{align*}
	|\mathcal{E}^V_{i,jkl}|&\leq \E[|V_{ij}V_{ik}V_{il}|]=\E[\bar{V}_i|V_{ij}V_{ik}V_{il}|]+\E[(1-\bar{V}_i)|V_{ij}V_{ik}V_{il}|]\\
	&\lesssim  B_n(\log(dn))^{1/\beta}\E[|V_{ij}V_{ik}|]+B_n^3(\log(dn))^{2/\beta}/n^2
	\stepcounter{equation}\tag{\theequation}\label{ThirdMomentBounds1} 
\end{align*}
and similarly 
\begin{align}
	\label{ThirdMomentBounds2}
	|\mathcal{E}^Z_{i,jkl}| \lesssim B_n(\log(dn))^{1/\beta}\E[|V_{ij}V_{ik}|]+B_n^3(\log(dn))^{2/\beta}/n^2~.
\end{align}

Just as in the previous auxiliary calculation we get
\begin{align*}
	&\frac{1}{n^{5/2}}\sum_{i \in I_s}\sum_{j,k,l,r,h=1}^{d}\E\left[\bar{V}_i\left|m^y_{jklr}\left(W^\sigma_{\sigma^{-1}(i)}+\frac{\bar{t}V_i}{\sqrt{n}}\right)V_{il}V_{ir}\right|\right]|\mathcal{E}^V_{i,jkl}-\mathcal{E}^Z_{i,jkl}|\\
	&\lesssim \frac{1}{n^{5/2}}\sum_{i \in I_s}\sum_{j,k,l,r,h=1}^{d} \E\left[h^y\left(W;2x\right)U^y_{jklrh}\left(W\right)\right]\E[|V_{ir}V_{ih}|]|\mathcal{E}^V_{i,jkl}-\mathcal{E}^Z_{i,jkl}|\\
	& \lesssim \left( \frac{B_n^3\phi^5(\log(dn))^{4+1/\beta}}{n^{3/2}}+\frac{B_n^3\phi^5(\log(dn))^{4+2/\beta}}{n^{7/2}}\right)\left(\E[\varrho_{\epsilon^{s+1}}]+\frac{\sqrt{\log d}}{\phi}\right)\\
	& \lesssim \frac{B_n^3\phi^5(\log(dn))^{4+1/\beta}}{n^{3/2}}\left(\E[\varrho_{\epsilon^{s+1}}]+\frac{\sqrt{\log d}}{\phi}\right)
\end{align*}
and
\begin{align*}
	&\frac{1}{n^{5/2}}\sum_{i \in I_s}\sum_{j,k,l,r,h=1}^{d}\E\left[(1-\bar{V}_i)\left|m^y_{jklrh}\left(W^\sigma_{\sigma^{-1}(i)}+\frac{\bar{t}V_i}{\sqrt{n}}\right)V_{ir}V_{ih}\right|\right]|\mathcal{E}^V_{i,jkl}-\mathcal{E}^Z_{i,jkl}|\\
	&\lesssim \frac{B_n\phi^5(\log(dn))^4}{n^{3/2}}\sum_{i=1}^{n}\E\left[(1-\bar{V}_i)\norm{V_i}_\infty^2\right] \lesssim
	\frac{B_n^3\phi^5(\log(dn))^{4+2/\beta}}{n^{5/2}}\,,
\end{align*}
where we used that by Condition V and Hölder's inequality
\begin{align*}
	|\mathcal{E}^V_{i,jkl}| \lesssim B_nn \quad \text{ and } \quad |\mathcal{E}^Z_{i,jkl}| \lesssim B_nn\,.
\end{align*}
Thus,
\begin{align*}
	\mathcal{I}_{3,2,1}\lesssim \frac{B_n^3\phi^5(\log(dn))^{4+1/\beta}}{n^{3/2}}\left(\E[\varrho_{\epsilon^{s+1}}]+\frac{\sqrt{\log d}}{\phi}\right)+\frac{B_n^3\phi^5(\log(dn))^{4+2/\beta}}{n^{5/2}}
\end{align*}
and since the same bound holds for $\mathcal{I}_{3,2,2}$ we have that 
\begin{align*}
	\mathcal{I}_{3,2}\lesssim \frac{B_n^3\phi^5(\log(dn))^{4+1/\beta}}{n^{3/2}}\left(\E[\varrho_{\epsilon^{s+1}}]+\frac{\sqrt{\log d}}{\phi}\right)+\frac{B_n^3\phi^5(\log(dn))^{4+2/\beta}}{n^{5/2}}\,,
\end{align*}
which finally yields
\begin{align*}
	|f^{(3)}(0)|&\lesssim \left(\frac{B_n^3\phi^5(\log(dn))^{4+1/\beta}}{n^{3/2}}+\frac{B_{n,2,s}\phi^3(\log d)^2}{n}\right)\left(\E[\varrho_{\epsilon^{s+1}}]+\frac{\sqrt{\log d}}{\phi}\right)\\
	&\quad +\frac{B_n^3\phi^5(\log(dn))^{4+2/\beta}}{n^{5/2}}\,.
\end{align*}

\textbf{Auxiliary Calculation 3.}
We decompose $f^{(4)}(\bar{t})=\mathcal{I}_{4,1}-\mathcal{I}_{4,2}$, where
\begin{align*}
	\mathcal{I}_{4,1}=\frac{1}{n^2}\sum_{i \in I_s}\sum_{j,k,l,r=1}\E\left[m^y_{jklr}\left(W^\sigma_{\sigma^{-1}(i)}+\frac{\bar{t}V_i}{\sqrt{n}}\right)V_{ij}V_{ik}V_{il}V_{ir}\right]\,,\\
	\mathcal{I}_{4,2}=\frac{1}{n^2}\sum_{i \in I_s}\sum_{j,k,l,r=1}\E\left[m^y_{jklr}\left(W^\sigma_{\sigma^{-1}(i)}+\frac{\bar{t}Z_i}{\sqrt{n}}\right)Z_{ij}Z_{ik}Z_{il}Z_{ir}\right]\,.
\end{align*}
Again denoting $x=C_pB_n(\log(dn))^{1/\beta}/\sqrt{n}+\phi^{-1}$ we have, by the same arguments leading to \eqref{V-h-m-trick},
\begin{align*}
	&\frac{1}{n^2}\sum_{i \in I_s}\sum_{j,k,l,r=1}\E\left[\bar{V}_i\left|m^y_{jklr}\left(W^\sigma_{\sigma^{-1}(i)}+\frac{\bar{t}V_i}{\sqrt{n}}\right)V_{ij}V_{ik}V_{il}V_{ir}\right|\right]\\
	&\lesssim \frac{1}{n^2}\sum_{i \in I_s}\sum_{j,k,l,r=1}\E\left[h^y(W^\sigma_{\sigma^{-1}(i)};x)U^y_{jklr}\left(W^\sigma_{\sigma^{-1}(i)}\right)\right]\E[|V_{ij}V_{ik}V_{il}V_{ir}|] \,.
\end{align*}
We also obtain
\begin{align*}
	\E[h^y(W^\sigma_{\sigma^{-1}(i)};x)U^y_{jklr}(W^\sigma_{\sigma^{-1}(i)})]\lesssim \E[h^y(W;2x)U^y_{jklr}(W)]
\end{align*}
by the same arguments as those leading to \eqref{2x-trick}. 
Hence,
\begin{align*}
	&\frac{1}{n^2}\sum_{i \in I_s}\sum_{j,k,l,r=1}\E\left[\bar{V}_i\left|m^y_{jklr}\left(W^\sigma_{\sigma^{-1}(i)}+\frac{\bar{t}V_i}{\sqrt{n}}\right)V_{ij}V_{ik}V_{il}V_{ir}\right|\right]\\
	&\lesssim \frac{1}{n^2}\sum_{j,k,l,r=1}\E\left[h^y(W;2x)U^y_{jklr}\left(W\right)\right]\sum_{i=1}^{n}\underset{1 \leq j,k,l,r\leq p}{\max}\E[|V_{ij}V_{ik}V_{il}V_{ir}|] \\
	&\lesssim \frac{B_n^2\phi^4(\log d)^3}{n}\left(\E[\varrho_{\epsilon^{s+1}}]+\frac{\sqrt{\log d}}{\phi} \right)\,,
\end{align*}
where the second inequality follows from the properties of $U^y$ and the arguments leading up to  \eqref{m^yBound}.
Moreover, we have
\begin{align*}
	&\frac{1}{n^2}\sum_{i \in I_s}\sum_{j,k,l,r=1}\E\left[(1-\bar{V}_i)\left|m^y_{jklr}\left(W^\sigma_{\sigma^{-1}(i)}+\frac{\bar{t}V_i}{\sqrt{n}}\right)V_{ij}V_{ik}V_{il}V_{ir}\right|\right] \\
	&\lesssim \frac{\phi^4(\log d)^3}{n^2}\sum_{i=1}^{n}\E[(1-\bar{V}_i)\norm{V_i}_\infty^4]\lesssim \frac{B_n^4\phi^4(\log d)^3(\log(nd))^{4/\beta}}{n^3}\\ 
	&\leq \frac{B_n^2\phi^4(\log(nd))^{1+2/\beta}}{n^2}
\end{align*}
by Condition B and $m^y_I\lesssim U^y_I$. Clearly the same bounds also hold for $\mathcal{I}_{4,2}$ which finally establishes 
\begin{align*}
	|f^{(4)}(\bar{t})|\lesssim \frac{B_n^2\phi^4(\log(nd))^{1+2/\beta}}{n^2}+\frac{B_n^2\phi^4(\log d)^3}{n}\left(\E[\varrho_{\epsilon^{s+1}}]+\frac{\sqrt{\log d}}{\phi} \right)\,.
\end{align*}
\end{proof}

\begin{Lemma}
	\label{ExchLemma}
	Suppose that the conditions of Lemma \ref{LindebergLemma} are satisfied. Then there exists a constant $K>0$ depending only on $C_v, C_p, C_b$  such that for all $s=0, \ldots ,D$, if $B_{n,1,s+1}\geq B_{n,1,s}+KB_n(\log(nd))^{1/2}$ and $B_{n,2,s+1}\geq B_{n,2,s}+KB_n^2(\log(dn))^{1/2+2/\beta}$, then for any constant $\phi>0$ satsifying $\eqref{mSumBound2}$ we have
	\begin{align*}
		\E[\varrho_{\epsilon^s}\1\{\mathcal{A}_s\}]\lesssim &
		\frac{\sqrt{\log d}}{\phi}+\frac{B_n^2\phi^4(\log(dn))^{3+2/\beta}}{n^2} +
	\left(\frac{\sqrt{\log d}}{\phi} + 
		\E[\varrho_{\epsilon^{s+1}}\1\{\mathcal{A}_{s+1}\}]\right) 
			\\& \times
		\left(\frac{B_{n,1,s}\phi^2\log d}{\sqrt{n}}+\frac{B_{2,n,s}\phi^3(\log d)^2}{n}+\frac{B_n^2\phi^4(\log d)^3}{n}\right)
	\end{align*}
	up to a constant only depending on $C_v, C_p, C_b, C_a$.
\end{Lemma}
\begin{proof}
	Fix $s=0, \ldots ,D-1$ and $\phi>0$ such that \eqref{mSumBound2} holds. By Lemma \ref{LindebergLemma} we have
	\begin{align*}
		\E[\varrho_{\epsilon^s}\1\{\mathcal{A}_s\}]\lesssim &
		\frac{\sqrt{\log d}}{\phi}+\frac{B_n^2\phi^4(\log(dn))^{3+2/\beta}}{n^2}+\left(\frac{\sqrt{\log d}}{\phi}+
		\E[\varrho_{\epsilon^{s+1}}\1\{\mathcal{A}_{d}\}]\right) 
		\\&  \times 
		\left(\frac{B_{n,1,s}\phi^2\log d}{\sqrt{n}}+\frac{B_{2,n,s}\phi^3(\log d)^2}{n}+\frac{B_n^2\phi^4(\log d)^3}{n}\right)
	\end{align*}
	up to a constant only depending on $C_v,C_p,C_B,C_a$. Hence the claim of the lemma follows if we can show that
	\begin{align*}
		\E[\varrho_{\epsilon^{s+1}}\1\{\mathcal{A}_s\}] \leq 	\E[\varrho_{\epsilon^{s+1}}\1\{\mathcal{A}_{s+1}\}]+\frac{4}{n}\,.
	\end{align*}
	We have
	\begin{align*}
		\E[\varrho_{\epsilon^{s+1}}\1\{\mathcal{A}_s\}]&=\E[\varrho_{\epsilon^{s+1}}\1\{\mathcal{A}_s\}\1\{\mathcal{A}_{s+1}\}]+\E[\varrho_{\epsilon^{s+1}}\1\{\mathcal{A}_s\}(1-\1\{\mathcal{A}_{s+1}\})]\\
		&\leq 	\E[\varrho_{\epsilon^{s+1}}\1\{\mathcal{A}_{s+1}\}]+\E[\1\{\mathcal{A}_s\}(1-\1\{\mathcal{A}_{s+1}\})]\\
		&\leq \E[\varrho_{\epsilon^{s+1}}\1\{\mathcal{A}_{s+1}\}]+1-\p(\mathcal{A}_{s+1}|\mathcal{A}_s)
	\end{align*}
	where we used that $0\leq \varrho_{\epsilon^{s+1}} \leq 1$ for the first inequality.
	Now Lemma \ref{ExchangeableConcentration} yields
	\begin{align*}
		\p\left(\left|\frac{1}{\sqrt{n}}\sum_{i=1}^{n}\epsilon^{s+1}_i(\mathcal{E}^V_{i,jk}-\mathcal{E}^Z_{i,jk})\right|>\left|\frac{1}{\sqrt{n}}\sum_{i=1}^{n}\epsilon^{s}_i (\mathcal{E}^V_{i,jk}-\mathcal{E}^Z_{i,jk})\right|+t \bigg| \epsilon^s \right)\\
		\leq 2\exp\left(-\frac{nt^2}{32\sum_{i=1}^{n}(\mathcal{E}^V_{i,jk}-\mathcal{E}^Z_{i,jk})^2}\right)\leq 2\exp\left(-\frac{t^2}{128B_n^2C_v}\right)\,,
	\end{align*}
	where the last inequality is due to Condition V. Setting $t=8B_n\sqrt{6C_v\log(dn)}$ and recalling that 
	\begin{align*}
		\underset{1\leq j,k\leq d}{\max}\left|\frac{1}{\sqrt{n}}\sum_{i=1}^{n}\epsilon^s_i(\mathcal{E}^V_{i,jk}-\mathcal{E}^Z_{i,jk})\right|\leq B_{n,1,d}
	\end{align*}
	on $\mathcal{A}_s$ we obtain by the tower property of conditional probabilities that for any $B_{n,1,s+1}\geq B_{n,1,s}+t$
	\begin{align*}
		\p\left(\underset{1\leq j,k\leq d}{\max}\left|\frac{1}{\sqrt{n}}\sum_{i=1}^{n}\epsilon^s_i(\mathcal{E}^V_{i,jk}-\mathcal{E}^Z_{i,jk})\right|>B_{n,1,d+1}\,\Big|\, \mathcal{A}_s\right)\leq \frac{2p^2}{(nd)^3}\leq \frac{2}{n} \,.
	\end{align*}
	
	We recall $\eqref{ThirdMomentBounds1}$ and $\eqref{ThirdMomentBounds2}$ which follow by Conditions P and B. Hence we find that 
	\begin{align*}
		\frac{32}{n}\sum_{i=1}^{n}(\mathcal{E}^V_{i,jkl}-\mathcal{E}^Z_{i,jkl})^2\leq CB_n^4(\log(dn))^{4/\beta}
	\end{align*}
	for some  constant $C$ only depending on $C_v, C_p$ and $C_b$. We hence obtain by the same arguments as above
	\begin{align*}
		\p\left(\left|\frac{1}{\sqrt{n}}\sum_{i=1}^{n}\epsilon^{s+1}_i(\mathcal{E}^V_{i,jkl}-\mathcal{E}^Z_{i,jkl})\right|>\left|\frac{1}{\sqrt{n}}\sum_{i=1}^{n}\epsilon^{s}_i(\mathcal{E}^V_{i,jkl}-\mathcal{E}^Z_{i,jkl})\right|+t \bigg| \epsilon^s \right)\\
		\leq 2\exp\left(-\frac{nt^2}{32\sum_{i=1}^{n}(\mathcal{E}^V_{i,jkl}-\mathcal{E}^Z_{i,jkl})^2}\right)\leq 2\exp\left(-\frac{t^2}{128C(\log(dn))^{4/\beta}}\right)\,.
	\end{align*}
	Applying this inequality with $t=\sqrt{3C}B_n^2(\log(dn))^{1/2+2/\beta}$ yields that for any $B_{n,2,s+1}\geq B_{n,2,s}+t$, we have
	\begin{align*}
		\p\left(\underset{1\leq j,k,l\leq d}{\max}\left|\frac{1}{\sqrt{n}}\sum_{i=1}^{n}\epsilon^{s+1}_i(\mathcal{E}^V_{i,jkl}-\mathcal{E}^Z_{i,jkl})\right|>B_{n,2,s+1}\,\Big|\, \mathcal{A}_s\right)\leq \frac{2p^3}{(nd)^3}\leq \frac{2}{n}\,. 
	\end{align*}
	Thus $1-\p(\mathcal{A}_{s+1}|\mathcal{A}_s)\leq 4/n$ which completes the proof.
\end{proof}	
\begin{Lemma}
	For any constant $\phi>0$ such that \eqref{mSumBound2} holds we have
	\begin{align*}
		\E[\varrho_{\epsilon^s}\1\{\mathcal{A}_D\}]\leq \frac{1}{n}\,.
	\end{align*}
\end{Lemma}
\begin{proof}
	Recall that $D=[\log(n)]+1$ and note that $\varrho_{\epsilon^D}=0$ if $\epsilon^D=(0, \ldots ,0)$. Moreover, by Markov's Inequality,
	\begin{align*}
		\p(\epsilon^D\neq (0, \ldots ,0))&=\p\left(\sum_{i=1}^{n}\epsilon^D_i\geq 1\right)\leq \E[\sum_{i=1}^{n}\epsilon^D_i]\\
		&=\E[\E[\sum_{i=1}^{n}\epsilon^D_i|\sum_{i=1}^{n}\epsilon^{D-1}_i]]=\E[\frac{1}{2}\sum_{i=1}^{n}\epsilon^{D-1}_i]\\&= \ldots .=\E[\frac{1}{2^D}\sum_{i=1}^{n}\epsilon^0_i]=\frac{n}{2^D}\leq \frac{n}{2^{4\log(n)}}\leq \frac{1}{n}\,.
	\end{align*}
	It follows that
	\begin{align*}
		\E[\varrho_{\epsilon^D}\1\{\mathcal{A}_D\}]\leq \E[\varrho_{\epsilon^D}]\leq \p(\epsilon^D\neq (0, \ldots ,0))\leq \frac{1}{n}\,.
	\end{align*}
\end{proof}

\begin{proof}[\textbf{Proof of Theorem \ref{Lindeberg}}]
	Throughout the proof we will assume that 
	\begin{align*}
		C_p^4B_n^2(\log(dn))^{4+2/\beta}\leq n
	\end{align*}	
	since otherwise the claim follows immediately.
	
	Let $K$ be the constant from Lemma \ref{ExchLemma} and for all $s=0, \ldots ,D$ define $B_{n,1,s}=C_1B_n(s+1)(\log(nd))^{1/\beta}$ and $B_{n,2,s}=C_1B_n^2(s+1)(\log(nd))^{1/2+2/\beta}$ where $C_1=K+C_M$ so that both $\mathcal{A}_0$ and the requirements for Lemma \ref{ExchLemma} hold. Now we define for $s=0, \ldots ,D$
	\begin{align*}
		f_s=\inf \left\{x\geq 1: \E[\varrho_{\epsilon^s}\1\{\mathcal{A}_s\}]\leq x\left(\frac{B_n^2(\log(dn))^{4+2/\beta}}{n}\right)^{1/4}\right\}
	\end{align*}
	and for all $s=0, \ldots ,D$ we apply Lemma \ref{ExchLemma} with \begin{align*}
		\phi=\phi_s=\frac{n^{1/4}}{B_n^{1/4}(\log(dn))^{1/2+1/(2\beta)}((d+1)f_{s+1})^{1/3}}~.
	\end{align*}
	Noting that
	\begin{align*}
		\frac{B_n^2\phi^4(\log(dn))^{3+2/\beta}}{n^2}&\leq \frac{\log(dn)}{n}\leq \frac{B_n^2C_p(\log(dn))^{1/4}}{n^{1/4}}\leq \frac{C_p \sqrt{\log d}}{\phi}\\&\leq C_p((s+1)f_{s+1})^{1/3}\left(\frac{B_n^2(\log(nd))^{4+2/\beta}}{n}\right)^{1/4}
	\end{align*}
	and
	\begin{align*}
		\frac{B_{n,1,s}\phi^2\log d}{\sqrt{n}}\leq \frac{C_1(s+1)}{((s+1)f_{s+1})^{2/3}}\\
		\frac{B_{n,2,s}\phi^3(\log d)^2}{n}+\frac{B_n^2\phi^4(\log(dn))^3}{n}\leq \frac{C_1+1}{f_{s+1}}
	\end{align*}
	we get for $s=0, \ldots D$
	\begin{align*}
		\E[\varrho_{\epsilon^s}\1\{\mathcal{A}_s\}]\leq C_2(f^{2/3}_{s+1}+(s+1)^{2/3}+1)\left(\frac{B_n^2(\log(dn))^{4+2/\beta}}{n}\right)^{1/4}
	\end{align*}
	for some constant $C_2$ depending only on $C_v,C_p,C_b,C_a,C_m$. Hence we obtain
	\begin{align*}
		f_s\leq C_2(f^{2/3}_{s+1}+(s+1)^{2/3}+1)).
	\end{align*}
	
	Clearly $f_D=1$ due to the previous lemma.  A simple induction then shows that 
	\begin{align*}
		f_s\leq C(s+1)
	\end{align*}
	for some constant $C\geq 1$ depending only on $C_2$. We then finally obtain
	\begin{align*}
		\varrho_{\epsilon^0}\1\{\mathcal{A}_0\}=\E[\varrho_{\epsilon^0}\1\{\mathcal{A}_0\}]\leq C\left(\frac{B_n^2(\log(dn))^{4+2/\beta}}{n}\right)^{1/4}.
	\end{align*}
\end{proof}

\subsection{Sub-Weibull Random Variables}
\label{Weibull}
In this section we collect some results on sub-Weibull random variables, which are mainly taken from \citeSM{Weibull}.
Recalling the definition of the  Orlicz norm in \eqref{orl}, 
 a random variable $X$ is called  sub-Weibull of order $\beta$, denoted sub-Weibull$(\beta)$, if
	\begin{align*}
	\norm{X}_{\psi_{\beta}}<\infty  ~,
	\end{align*}
	where $
	\psi_\beta(x)=\exp(x^\beta)-1. $ 
	We also occasionally  call $\norm{X}_{\psi_{\beta}}$ its $\beta$-parameter. This definition includes 
the important   sub-exponential ($\beta=1$)  and sub-Gaussian ($\beta=2$) cases. Clearly sub-Weibull$(\beta)$ random variables possess exponential tail decay rates, more precisely $
	\p(|X|\geq t)\leq 2\exp  (-{t^\beta}/{\norm{X}_{\psi_\beta}^\beta})
$. The following result is a slight refinement of this statement, which for instance can be found in  \citeSM{Weibull}.
\begin{Lemma}
	\label{WeibullChar}
	For any random variable $X$ and constant $\beta>0$ the following are equivalent:
	\begin{itemize}
		\item[i)] $\norm{X}_{\psi_{\beta}}=K_1$,
		\item[ii)] $\p(|X|\geq t)\leq 2\exp\left(-\frac{t^\beta}{K_2^\beta}\right)$,
		\item[iii)] $\sup_{p\geq 1} \frac{\norm{X}_p}{p^{1/\beta}}=K_3$,
	\end{itemize} 
	where we have $K_1 \lesssim K_2 \lesssim K_3 \lesssim K_1$ up to constants only depending on $\beta$. 
	Note that the third formulation yields a quasi-triangle inequality for the $\beta$-parameter of sums of finitely many  random variables.
\end{Lemma}	

\begin{proof}
If 	i) holds,  ii) follows from Markov's inequality.
\smallskip

If 	ii) holds, it follows that
	\begin{align*}
	\E[|X|^p]&=\int_{0}^{\infty}\p(|X|^p>t)dt=\int_{0}^{\infty}\p(|X|>t^{1/p})dt\\
	&\leq 2\int_{0}^{\infty}\exp(-t^{\beta/p}/K_2^\beta)dt=2K_2^p\frac{p}{\beta}\Gamma(p/\beta)\,.
	\end{align*}
	Taking the $p-$th root and recalling $\Gamma(x)\leq x^{1/x}$ then yields
$
	\frac{\norm{X}_p}{p^{1/\beta}}\leq C_\beta K_2 ,
$
	which implies iii).
\smallskip

If 	iii) holds,  we have  for some $K>0$
	\begin{align*}
	\E[\exp(|X|^\beta/K^\beta)-1]=\sum_{n=1}^{\infty}\frac{\E[|X|^{\beta n}]}{K^{\beta n}n!}\leq  \sum_{n=1}^{\infty} (\beta n)^{n}\frac{K^{\beta n}_3}{K^{\beta n}n!} \leq \sum_{n=1}^{\infty}(\beta e)^n\left( \frac{K_3}{K}\right)^{\beta n}\,.
	\end{align*}
	
	Now there exists a constant $C_\beta$ only depending on $\beta$ such that $ K=C_\beta K_3$ yields that the last term is bounded by 1. This implies i).
	
\end{proof}

\begin{Lemma}
	\label{WeibCond}
	Let $X$ be a random variable with $\norm{X}_{\psi_{\beta}}<\infty$. Then for any sigma algebra $\mathcal{B}$ we have that $\norm{\E[X|\mathcal{B}]}_{\psi_{\beta}}\leq C_\beta\norm{X}_{\psi_{\beta}}$. 
\end{Lemma}
\begin{proof}
	This follows immediately from Lemma \ref{WeibullChar} and the fact that conditional expectations are $\mathcal{L}_p$ contractions.
\end{proof}

\begin{Lemma}
	\label{SubgaussAverage}
	Let $\bar{X}_n=\frac{1}{n}\sum_{k=1}^{n}X_k$ be the average of sub-Weibull$(2)$ random variables with 2-parameter $\sigma$. Then $\bar{X}_n$ is sub-Weibull$(2)$ with 2-parameter at least $\frac{\sigma C}{\sqrt{n}}$  and at most $\frac{\sigma \bar{C}}{\sqrt{n}}$ for some universal constants $C,\bar{C}>0$.
\end{Lemma}

\begin{Lemma}
	\label{ProductNorm}
	Let $X_1, \ldots, X_n$ be random variables with $\norm{X_k}_{\psi_{\beta_k}}<\infty$ ($k=1, \ldots, n$). Then for 
$
	\frac{1}{\beta}=\sum \frac{1}{\beta_k}
$
	we have
	\begin{align*}
	\norm{\prod_{k=1}^{n} X_k}_{\psi_\beta}\leq\prod_{k=1}^{n} \norm{X_k}_{\psi_{\beta_k}}\,.
	\end{align*}
\end{Lemma}

\begin{Lemma}
	\label{MaxBound}
	Assume that $X_i=(X_{i1},\ldots, X_{id})^{\top}$, $1\le i\le n$, are random vectors whose components $X_{ij}$, $1\le j\le d$, are sub-Weibull$(\beta)$ random variables  with $\norm{X_{ij}}_{\psi_{\beta}}\leq K$. Then for $d\ge 2$ we have	
	\begin{align*}
		\underset{1\leq i \leq d}{\max}\norm{X_i}_\infty\leq K\left(5\log(dn)\right)^{1/\beta}
	\end{align*}
	with probability at least $1-1/(2n^4)$.
\end{Lemma}

\begin{proof}
	Using the union bound and Lemma~\ref{WeibullChar}, we obtain for any $x>0$, 
	\begin{align*}
		\p(\underset{1\leq i \leq n, 1 \leq j \leq d}{\max}|X_{ij}|>x)\leq dn \underset{1\leq i \leq n, 1 \leq j \leq d}{\max} \p(|X_{ij}|>x)\leq 2dn \exp\left(-\frac{x^\beta}{K^\beta}\right)\,.
	\end{align*}
	Taking $x=K\left(5\log(dn)\right)^{1/\beta}$ yields the desired claim.
\end{proof}

\begin{Lemma}[\citeSM{Weibull}, Theorem 3.4]
	\label{WeibConc}
	 Let $X_1,\ldots,X_n$ be independent $d$-dimensional random vectors with mean zero and components satisfying $\norm{X_{ij}}_{\psi_\beta}\leq K_n$ for some $\beta\leq 2$. Setting
	\begin{align*}
	\Gamma_n:=\underset{1 \leq j \leq d}{\max}\frac{1}{n}\sum_{i=1}^{n}\E[X_{ij}^2]
	\end{align*}
	we have for $t>0$, with probability at least $1-3e^{-t}$,
	\begin{align*}
	\norm{\frac{1}{n}\sum_{i=1}^{n}X_i}_\infty\lesssim \sqrt{\frac{\Gamma_n(t+\log d)}{n}}+K_n\frac{(\log(2n))^{1/\beta}(t+\log d)^{1/\beta^*}}{n}
	\end{align*}
	 up to some constant depending only on $\beta$ and where $\beta^*=\min(1,\beta)$. In particular, noting that $\Gamma_n \lesssim K^2_n$ up to a constant depending only on $\beta$, we have for $t=\log d$ and $\frac{1}{n}\lesssim (\log d)^{1-2/\beta^*}$ that 
	\begin{align*}
	\norm{\frac{1}{n}\sum_{i=1}^{n}X_i}_\infty\lesssim K_n\sqrt{\frac{\log d}{n}}
	\end{align*}
	holds with probability at least $1-3/d$. When $\log d \lesssim n^\gamma$ this holds
	as long as $\gamma \leq \frac{1}{2/\beta^*-1}$.

\end{Lemma}

\subsection{Further technical details}

All results in this section are taken from \citeSM{ImprovedApprox}, but we will list them here for sake of completeness.
\begin{Lemma}
	\label{Permutation}
	\citeSM[Lemma 7.2]{ImprovedApprox} Let $\mathcal{S}_n$ be the set of all permutations of $\{1, \ldots ,n\}$. Let $X_1, \ldots ,X_n$, $Y_1, \ldots ,Y_n$ be sequences of vectors in $\R^d$. Let $U$ be a random variable with uniform distribution on $[0,1]$ and $\sigma$ be uniformly distributed on $\mathcal{S}_n$ and also independent from $U$.
	For $k=1,\ldots, n$ denote 
	\begin{align*}
		W^\sigma_k=\sum_{j=1}^{k-1}X_{\sigma(j)}+\sum_{j=k+1}^{n}Y_{\sigma(j)}
	\end{align*}
	and 
	\begin{align*}
		W_k=\begin{cases}
			W^\sigma_{\sigma^{-1}(k)}+X_k\,, \quad \text{if } U\leq \frac{\sigma^{-1}(k)}{n+1}\,,\\
			W^\sigma_{\sigma^{-1}(k)}+Y_k\,, \quad \text{if } U > \frac{\sigma^{-1}(k)}{n+1}\,.
		\end{cases}
	\end{align*}
	
	Then the distribution of $W_k$ does not depend on $k$ and there exists a random vector  $\epsilon=(\epsilon_1, \ldots ,\epsilon_n)$ with values in $\{0,1\}^n$ such that the distribution of $W_k$ is equal to that of
	\begin{align*}
		\sum_{i=1}^{n}\left(\epsilon_iX_i+(1-\epsilon_i)Y_i\right)
	\end{align*}
	In particular, the random variables $\epsilon_i$ are exchangeable and their sum is uniformly distributed on $\{0, \ldots ,n\}$.
\end{Lemma}

\begin{Lemma} \label{ExchangeableConcentration}
	\citeSM[Lemma7.1]{ImprovedApprox} Let $a_1, \ldots ,a_n$ be some constants in $\R$ and let $X_1, \ldots ,X_n$ be exchangeable random variables such that $|X_i|\leq 1$ almost surely. Then
	\begin{align*}
		\p\left(\left|\sum_{i=1}^{n}a_iX_i\right|\geq \left|\sum_{i=1}^{n}a_i\right|+t\right)\leq 2\exp\left(-\frac{t^2}{32\sum_{i=1}^{n}a_i^2}\right)
	\end{align*}
	for all $t>0$.
\end{Lemma}

\begin{Lemma}
	\label{MaxIneqCentSum}
	\citeSM[Lemma7.1]{Chernozhukov2017}
	Let $X_1, \ldots ,X_n$ be independent centered random vectors in $\R^d$ with $d\geq 2$. Define  $Z=\max_{1 \leq j \leq d} |\sum_{i=1}^{n}X_{ij}|, M=\max_{1 \leq i \leq n, 1 \leq j \leq d}|X_{ij}|$ and $\sigma^2=\max_{1 \leq j \leq d}\sum_{i=1}^n\E[X_{ij}^2]$. Then
	\begin{align*}
		E[Z]\leq L(\sigma\sqrt{\log d}+\sqrt{E[M^2]}\,\log d)
	\end{align*}
	for some universal constant $L$.
	Moreover, for every $\nu>0, \beta\in (0,1]$ and $t>0$ we have
	\begin{align*}
		\p(Z\geq (1+\nu)\E[Z]+t)\leq \exp(-t^2/(3\sigma^2))+3\exp\left(-\frac{t^\beta}{K\norm{M}_{\psi_{\beta}}^\beta}\right)
	\end{align*}
	for some universal constant $K$ that depends only on $\nu$ and $\beta$.
\end{Lemma}

\subsection{Concentration Inequalities for U-Statistics}

\begin{definition}
	Consider a symmetric and measurable function $h=(h_1,\ldots,h_d)^{\top}:\left(\R^p\right)^m\rightarrow \R^d$ together with a collection of iid random variables $X_1,\ldots,X_n \in \R^p$. We define the associated U-statistic $U_n$ of order $m$ by
	\begin{align*}
	U_n={n \choose m}^{-1}\sum_{1\leq l_1< \ldots <l_m\leq n}h(X_{l_1}, \ldots ,X_{l_m})\,.
	\end{align*}
	For $x\in \R^d$ we write
	\begin{align*}
	h_{1,i}(x)=\E[h(X_1, \ldots ,X_{m})|X_1=x]\,, \qquad 1\le i\le d\,,
	\end{align*}
	and set $h_{(1)}(x)=(h_{1,1}(x),\ldots,h_{1,d}(x))^{\top}$.
	
\end{definition}

\begin{Lemma}
	\label{UStatConc}
	Consider a mean zero U-Statistic $U_n$ of order $m$ as defined above. Provided that $\max_{1\le i \le d} \norm{h_i(X_1, \ldots ,X_m)}_{\psi_{\beta}}\leq K$ for some $2 \geq \beta>0$ and that $\log d=o(n^\gamma)$ for $\gamma \leq \frac{1}{2/\beta+1}$ it holds
		\begin{align*}
		\norm{U_n}_\infty \lesssim K\sqrt{\frac{\log d}{n}}
		\end{align*}
		with probability at least $1-3/d-C(\log d)^{1/2+1/\beta}/\sqrt{n}$ for some universal constant $C>0$. Note that the same bound with $\log (nd)$ instead of $\log d$ holds with probability at least  $1-3/(nd)-C(\log d)^{1/2+1/\beta}/\sqrt{n}$
\end{Lemma}
\begin{proof}
	By Theorem 5.1 from \citeSM{songetal2019} we obtain that
	\begin{align*}
	\E\norm{U_n-\frac{m}{n}\sum_{k=1}^nh_{(1)}(X_k)}_\infty\lesssim K\left(\frac{(\log d)^{1+1/\beta}}{n}\right)
	\end{align*}
	up to some universal constant that depends only on $m$ and $\beta$. Using Markov's inequality we deduce that
	\begin{align}
	\label{DegenConc}
	\p\left(\norm{U_n-\frac{m}{n}\sum_{k=1}^nh_{(1)}(X_k)}_\infty>t\right)\lesssim \frac{K}{t} \left(\frac{(\log d)^{1+1/\beta}}{n}\right)\,.
	\end{align}

	For the linear part of $U_n$ we obtain by Lemmas \ref{WeibConc} and \ref{WeibCond} that 
	\begin{align*}
	\norm{\frac{m}{n}\sum_{k=1}^{n}h_{(1)}(X_k)}_\infty \lesssim K\sqrt{\frac{\log d}{n}}
	\end{align*}
	with probability at least $1-3/d$ as long as $\log d=o(n^\gamma)$ where $\gamma\leq \frac{1}{2/\beta^*-1}$. Setting $t=K\sqrt{\frac{\log d}{n}}$ in \eqref{DegenConc} then yields
	\begin{align*}
	\norm{U_n}_\infty \lesssim K\sqrt{\frac{\log d}{n}}
	\end{align*}
	with probability at least $1-3/d-C(\log d)^{1/2+1/\beta}/\sqrt{n}$ up to some universal constant C that depends only on $m$ and $\beta$. The second bound is obtained by the same arguments but with a different choice of $t$.
\end{proof}

\bibliographystyle{apalike}
\setlength{\bibsep}{2pt}
\bibliography{reference}

\end{document}